\documentclass[reqno]{amsart}

\usepackage[english, activeacute]{babel}
\usepackage{amsmath,amsthm,amsxtra}
\usepackage{fix-cm}
\usepackage{bm}
\usepackage{anysize}
\usepackage{graphics}
\usepackage{graphicx}	
\usepackage{fancyhdr}
\usepackage{float}
\usepackage{anysize}
\usepackage{thmtools}
\usepackage{stmaryrd}
\usepackage{amssymb,mathrsfs}
\usepackage{arydshln}
\usepackage{enumerate}
\usepackage{mathtools}
\usepackage{epsfig}
\usepackage{amssymb}
\usepackage{float}
\usepackage{tikz-cd}
\usetikzlibrary{calc,backgrounds}
\usepackage{dsfont}
\usetikzlibrary{fit,calc,positioning,decorations.pathreplacing,matrix}
\usepackage{pstricks-add}
\usepackage{pst-node,pst-plot}

\usepackage{latexsym}
\usepackage{amsfonts}
\usepackage{hyperref}
\usepackage{xfrac}
\usepackage[capitalize]{cleveref}
\usepackage{enumitem}

\usepackage{tikz}
\usetikzlibrary{automata,positioning,matrix}

\newtheorem{theorem}{Theorem}[section]
\newtheorem*{theorem*}{Theorem}
\newtheorem{corollary}[theorem]{Corollary}

\newtheorem{lemma}[theorem]{Lemma}
\newtheorem{question}[theorem]{Question}
\newtheorem{proposition}[theorem]{Proposition}

\theoremstyle{definition}
\newtheorem{definition}[theorem]{Definition}
\newtheorem{remark}[theorem]{Remark}

\newtheorem{claim}{Claim}

\theoremstyle{definition}
\newtheorem{example}[theorem]{Example}

\newcommand{\N}{\mathbb{N}}
\newcommand{\Z}{\mathbb{Z}}

\usepackage{stackengine}

\def\T{\mathbb{T}}

\def\R{\mathbb{R}}

\DeclareMathOperator{\conv}{conv}
\DeclareMathOperator{\Mod}{mod}
\DeclareMathOperator{\lcm}{lcm}

\def\SS{\mathbb{S}}
\def \id {\textrm{id}}
\def\N{\mathcal{N}}
\def\NN{\mathbb{N}}

\def\A{\mathcal{A}}
\def\B{\mathcal{B}}
\def\T{\mathcal{T}}

\def\NN{\mathbb{N}}

\DeclareMathOperator{\ND}{ND}
\DeclareMathOperator{\DD}{DD}

\DeclareMathOperator{\Aut}{Aut}
\DeclareMathOperator{\Fac}{Fac}
\DeclareMathOperator{\supp}{supp}

\DeclareMathOperator{\dist}{dist}

\DeclareMathOperator{\End}{End}

\DeclareMathOperator{\diam}{diam}
\DeclareMathOperator{\Homeo}{Homeo}
\DeclareMathOperator{\Hom}{Hom}

\DeclareMathOperator{\Ext}{Ext}
\DeclareMathOperator{\Aff}{Aff}
\DeclareMathOperator{\cl}{cl}
\DeclareMathOperator{\ri}{ri}
\DeclareMathOperator{\cone}{cone}
\DeclareMathOperator{\diag}{diag}
\newcounter{sigmavariable}
\newcommand{\bigdotcup}{\mathop{\mathchoice{\mathaccent"702E{\smash{\bigcup\nolimits}}}{\mathaccent"7201{\smash{\bigcup\nolimits}}}{\mathaccent"7201{\smash{\bigcup\nolimits}}}{\mathaccent"7201{\smash{\bigcup\nolimits}}}\vphantom{\bigcup}}}

\definecolor{zzttqq}{rgb}{0.6,0.2,0.}

\title[Homomorphisms between multidimensional constant-shape substitutions]{Homomorphisms between multidimensional constant-shape substitutions}
\author{Christopher Cabezas}
\address{Laboratoire Ami\'enois de Math\'ematiques Fondamentales et Appliqu\'ees, CNRS-UMR 7352, Universit\'{e} de Picardie Jules Verne, 33 rue Saint Leu, 80039   Amiens cedex 1, France.}
\curraddr{University of Li\`ege, Department of Mathematics, All\'ee de la d\'ecouverte 12 (B37), B-4000 Li\`ege, Belgium.}
\email{ccabezas@uliege.be}



\begin{document}
	\begin{abstract}
		We study a class of $\Z^{d}$-substitutive subshifts, including a large family of constant-length substitutions, and homomorphisms between them, i.e., factors modulo isomorphisms of $\Z^{d}$. We prove that any measurable factor map and even any homomorphism associated to a matrix commuting with the expansion matrix, induces a continuous one. We also get strong restrictions on the normalizer group, proving that any endomorphism is invertible, the normalizer group is virtually generated by the shift action and the quotient of the normalizer group by the automorphisms is restricted by the digit tile of the substitution.
	\end{abstract}
	\maketitle
	
	\section{Introduction}
	
	In this article, we study multidimensional \emph{constant-shape substitutions} and \emph{homomorphisms} between them, i.e., continuous maps $\phi:X\to Y$ (called \emph{isomorphisms} for invertible ones) such that for some matrix $M\in GL(d,\Z)$ and any ${\bm n}\in \Z^{d}$, $\phi\circ S^{{\bm n}}=S^{M{\bm n}}\circ \phi$, where $(X,S,\Z^{d})$ and $(Y,S,\Z^{d})$ are subshifts given by substitutions with a uniform support. Such a map gives an orbit equivalence with constant orbit cocycle via a linear map. When the matrix $M$ is the identity, surjective homomorphisms are called \emph{factor maps}, and \emph{conjugacies} when it is invertible. We refer to the conjugacies as \emph{automorphisms} when the dynamical systems are the same. In the one-dimensional case, homomorphisms lead to the notion of \emph{flip conjugacy} of dynamical systems \cite{bezuglyi2008fullgroups} and by this fact are also called \emph{reversing symmetries} (see \cite{goodson1999inverse},\cite{baake2006structure}). The relation between homomorphisms and factor maps becomes less clear in higher dimensions, since $GL(d,\Z)$ is infinite for $d\geq 2$ (see for example \cite{baake2019number}). 
	
	The study of factors and automorphisms of a dynamical system is a classical problem. It mainly concerns their algebraic and dynamical properties in relation with the one of the system $(X,S, \Z^{d})$. The automorphisms can be algebraically defined as elements of the centralizer of the action group $\left\langle S \right \rangle$, seen as a subgroup of all homeomorphisms $\Homeo(X)$ from $X$ to itself. With this algebraic point of view, isomorphisms can be seen as elements of the normalizer group of $\left\langle S \right \rangle$ seen as a subgroup of $\Homeo(X)$. The automorphism group is always nonempty, but in general, the existence of isomorphisms for a particular matrix $M\in GL(d,\Z)$ is an open problem. 
	
	In this context, the rich family of symbolic systems exhibits rigidity properties of factor maps and automorphisms already in the one-dimensional case. For instance, the famous Curtis-Hedlund-Lyndon theorem \cite{hedlund1969endomorphism}, ensures that any factor map between subshifts is a \emph{sliding block code}, showing that the automorphism group is countable. Among the simplest nontrivial zero-entropy symbolic systems, are the substitutive ones introduced by W.H. Gottschalk in \cite{gottschalk1963substitution} (see \cite{queffelec2010substitution} for a good bibliography on this subject). They also present rigidity properties. B. Host and F. Parreau in \cite{host1989homomorphismes} gave a complete description of factor maps between subshifts arising from certain constant-length substitutions, proving that any measurable factor map induces a continuous one, and the automorphism group is virtually generated by the shift action. Moreover, any finite group can be realized as a quotient group $\Aut(X,S,\Z)/\left\langle S\right\rangle$ for these subshifts as proved by M. Lema\'nczyk and M. K. Mentzen in \cite{lemanczyk1988metric}. Later, I. Fagnot \cite{fagnot1997facteurs} proved that the problem of whether there exists a factor map between two constant-length substitution subshifts is decidable, using the first-order logic framework of Presburger arithmetic. Some years later, F. Durand in \cite{durand2000linearly} showed that \emph{linearly recurrent subshifts} (in particular substitutive subshifts) have finitely many symbolic factors, up to conjugacy. Using the self-induced properties of substitutive subshifts, V. Salo and I. T\"{o}rm\"{a} provide in \cite{salo2015blockmaps} a 
	renormalization process of the factor maps to extend the description obtained in \cite{host1989homomorphismes}. In \cite{donoso2016lowcomplexity} the authors proved that the automorphism group of a minimal subshift with non-superlinear complexity is virtually generated by the shift action, using the concept of \emph{asymptotic pairs}. Next, C. M\"ullner and R. Yassawi \cite{yassawi2020automorphisms} demonstrated that any topological factor of a constant-length substitutive shift is conjugate to a constant-length substitution via a letter-to-letter map. More recently, F. Durand and J. Leroy \cite{durand2018decidability} showed the decidability of the existence problem of a factor map between two minimal substitutive subshifts. 
	
	In the multidimensional setting, substitutive systems are originally motivated by physical reasons with the discovery of the aperiodic structure of quasicrystals modelized by the Penrose tiling \cite{penrose1974role}, where the symmetries play a fundamental role. Substitutions also occur in different topics such as combinatorics, diophantine approximations and theoretical computer science, with the minimal Robinson subshift of finite type being one of the most fundamental example \cite{gahler2012combinatorics}. Characterizations of the isomorphisms of the chair tiling, together with the full shift and Ledrappier's shift, were given in \cite{baake2018reversing}. The chair tiling, the table tiling and the minimal Robinson tiling belong to the class of constant-shape substitutions, which is a multidimensional analogue of the so-called constant-length substitutions. As a difference with the one-dimensional case, these substitutions may not be linearly recurrent (\cref{ExampleNonLinearMultidimensionalSubstitution}). In \cite{bustos2020extended,bustos2022admissible} was studied the case of bijective block substitutions.
	
	In this article, we pursue the study of isomorphisms to homomorphisms, and more general multidimensional substitutions (nondiagonal expansion matrix, nonrectangular support, and a weaker version of bijectivity. See \cite{frank2022spectral} for recent results on their spectral properties). We also obtain rigidity properties about homomorphisms. First, we prove that any aperiodic symbolic factor of a constant-shape substitution is conjugate to a constant-shape substitution via a letter-to-letter map (\cref{FactorConjugateSubstitution}), extending the mentioned one-dimensional result from \cite{yassawi2020automorphisms}. Then, we show that any measurable factor map and any homomorphism associated with a matrix commuting with some power of the expansion matrix of the substitution induces a continuous one, and we give an explicit bound on the radius of its block maps (\cref{MainTheorem} and \cref{NormalizerHostParreau}). These are analogue results of B. Host and F. Parreau's from \cite{host1989homomorphismes}. These imply that certain constant-shape substitutions are \emph{coalescent} (\cref{Coalescence}), and the automorphism group is virtually generated by the shift action (\cref{AutomoprhismVirtuallyZd}). Finally, we give algebraic and geometrical properties of the normalizer group for \emph{polytope substitutions}, i.e., in the case where the convex hull of the digit tile generated by the expansion matrix and the support of the substitution is a polytope. To do this, we relate the \emph{nondeterministic directions} of substitutive subshifts to the supporting hyperplanes to the convex hull of the digit tile (\cref{NonExpansiveHalfspacesSupportingConvexHull}). We deduce that any homomorphism of the substitutive subshift is invertible, and the normalizer group is virtually generated by the shift action (\cref{FinalTheoremNormalizerGroupPolytopeCase}). Moreover, the \emph{linear representation group}, defined as the set of matrices associated with a homomorphism, is finite and we give explicit bounds for the norm of these matrices (\cref{PropositionAboutMatricesPolytopeCase}). Together with the former bound on the radii of the block maps, these restrictions enable an algorithmic description of the normalizer group whenever the expansion matrix is proportional to the identity. These recover results and answer some questions in \cite{bustos2022admissible}.
	
	This article is organized as follows. The basic definitions and background are introduced in \cref{SectionBackground}. In \cref{SectionRecognizabilityPropertyGeneral} we prove \cref{FactorConjugateSubstitution} characterizing the aperiodic symbolic factors of substitutive subshifts. For this we study a recognizability property of these symbolic factors (\cref{RecognizabilityFactors}) and we determine their maximal equicontinuous factor (\cref{MaximalEquiContinuousFactorMultidimensionalSubstitution}). We also give a polynomial bound on the repetitivity function for substitutive subshifts (\cref{GrowthRepetititvtyFunction}). \cref{Sectionproof} is devoted to the proofs of the measurable rigidity properties of homomorphisms: \cref{MainTheorem} and \cref{NormalizerHostParreau}. Then, we deduce the coalescence (\cref{Coalescence}) and that the automorphism group of substitutive subshifts is virtually $\Z^{d}$ (\cref{AutomoprhismVirtuallyZd}). Finally, in \cref{SectionBijectiveSubstitutions} we describe the nondeterministic directions of substitutive subshifts through the digit tile for \emph{bijective on the extremities substitutions} (\cref{NonExpansiveHalfspacesSupportingConvexHull}). Moreover, these directions are computable in terms of the combinatorics of the substitution (\cref{CorollaryNonExpansiveHalfspacesBijectiveSubstitutions}). This enables us to provide algebraic restrictions and to bound elements of the linear representation group (\cref{PropositionAboutMatricesPolytopeCase}). The last theorem (\cref{FinalTheoremNormalizerGroupPolytopeCase}) summarizes all the results of this last study.
	
	\subsection*{Acknowledgments} The author thanks Samuel Petite for all of his support, dedication and guidance during the process of this work, Natalie Priebe Frank for useful discussions in the subject and Julien Leroy for a careful reading on Section 3 and for many helpful comments. The author also thanks the anonymous referee for very helpful comments and suggestions that improved this article.
	
	\section{General setting and notions}\label{SectionBackground}
	
	\subsection{Basic definitions and notation}
	\subsubsection{Notation} Throughout this article we will denote by ${\bm n}=(n_{1},\ldots,n_{d})$ the elements of $\Z^{d}$ and by ${\bm x}=(x_{1},\ldots,x_{d})$ the elements of $\R^{d}$. If $F\subseteq \Z^{d}$ is a finite set, it will be denoted by $F\Subset \Z^{d}$ and we use the notation $\Vert F\Vert =\max\limits_{{\bm n}\in F}\Vert {\bm n}\Vert$, where $\Vert\cdot\Vert$ is the standard Euclidean norm of $\R^{d}$. The standard cartesian product in $\R^{d}$ will be denoted by $\left\langle \cdot,\cdot\right\rangle$. If $L\in \mathcal{M}(d,\R)$ is a matrix, we denote $\Vert L\Vert=\max\limits_{{\bm x}\in \R\setminus\{{\bm 0}\}} \Vert L({\bm x})\Vert/\Vert {\bm x}\Vert$ as the \emph{matrix norm of} $L$. We denote $GL(d,\Z)$ as the set of $d\times d$ matrices $M$ with integer coefficients such that $|\det(M)|=1$. The matrices $M\in GL(d,\Z)$ represent the automorphisms of $\Z^{d}$.
	
	We will call a sequence of finite sets $(A_{n})_{n>0}\subseteq \Z^{d}$ \emph{a F}\o\emph{lner sequence}\footnote{In the literature, especially group theory, it is common to also ask that the union of the sequence of sets $(F_{n})_{n>0}$ is equal to $\Z^{d}$ for a sequence to be F\o lner, but we will not use it in this article.} if for all ${\bm n}\in \Z^{d}$ we have that
	$$\lim\limits_{n\to \infty}\dfrac{|A_{n}\Delta ({\bm n}+A_{n})|}{|A_{n}|}=0.$$
	
	For any $r>0$ and $F\Subset \Z^{d}$ we denote $F^{\circ r}$ as the set of all elements ${\bm f}\in F$ such that ${\bm f}+(B({\bm 0},r)\cap \Z^{d})\subseteq F$, i.e.,
	$$F^{\circ r}=\{{\bm f}\in F\colon {\bm f}+(B({\bm 0},r)\cap \Z^{d})\subseteq F\}.$$
	
	Note that the F\o lner assumption implies that for any $r>0$
	\begin{equation}\label{LargerballsFolner}
		\lim\limits_{n\to \infty}\dfrac{|F_{n}^{\circ r}|}{|F_{n}|}=1.
	\end{equation}	
	
	\subsubsection{Convex geometry}\label{Subsectionconvexgeometry}
	
	A set $C\subseteq \R^{d}$ is said to be \emph{convex} if for all ${\bm x}, {\bm y}\in C$ the set $[{\bm x},{\bm y}]=\{{\bm z}\in \R^{d}\colon {\bm z}=t{\bm x}+(1-t){\bm y}, t \in [0,1]\}$ is included in $C$. Recall that the image of a convex set under an affine map is also a convex set, and the intersection of an arbitrary family of convex sets is also a convex set. This leads to the notion of convex hull of a set.
	
	If $A\subseteq \R^{d}$ we define the \emph{convex hull of} $A$, denoted by $\conv(A)$, as the intersection of all convex sets containing $A$.
	
	A set $S\subseteq \R^{d}$ is an \emph{affine set} if for any ${\bm x},{\bm y}\in S$ the line $\{t{\bm x}+(1-t){\bm y}\colon t \in \R\}$ is contained in $S$. For any set $A\subseteq \R^{d}$ we define the \emph{affine hull of} $A$, denoted by $\Aff(A)$, as the intersection of all affine sets containing $A$.
	
	A fundamental characterization of convex sets is provided by Carath\'eodory's theorem.
	
	\begin{theorem}[Carath\'eodory's theorem]
		For any $A\subseteq \R^{d}$, any element of $\conv(A)$ can be represented as a convex combination of no more than $(d+1)$ elements of $A$.
	\end{theorem}

	We now recall some basic topological concepts associated with convex sets. A point ${\bm x}\in A$ is said to be \emph{relative-interior} for $A$, if $A$ contains the intersection of a ball centered at ${\bm x}$ with $\Aff(A)$, i.e., $\exists r>0,\ B({\bm x},r)\cap \Aff(A)\subseteq A$. The set of all relative-interior points of $A$ is called the \emph{relative interior of} $A$ and is denoted by $\ri(A)$. We can also define the \emph{relative boundary} $\partial_{\ri}(A)$ as the set difference of the closure and the relative interior, i.e., $\partial_{\ri}(A)=\cl(A)\setminus \ri(A)$.
	
	An important notion for convex sets are the supporting hyperplanes. Let $C\subseteq \R^{d}$ be a closed convex set and ${\bm x}\in C$ be a point in the relative boundary of $C$. An affine hyperplane $\partial H[{\bm a};c]=\{{\bm y}\in \R^{d}\colon \left\langle {\bm a},{\bm y}\right\rangle=c\}$, for some ${\bm a}\in \R^{d}\setminus\{0\}$ and $c\in \R$ is called a \emph{supporting hyperplane to} $C$ \emph{at} ${\bm x}$ if
	${\bm x}\in \partial H[{\bm a}:c]$ and
	$$\inf\limits_{{\bm y}\in C}\left\langle {\bm a},{\bm y}\right\rangle< \left\langle {\bm a},{\bm x}\right\rangle=c = \sup\limits_{{\bm y} \in C} \left\langle {\bm a},{\bm y}\right\rangle.$$
	
	We now recall some basic notions about cones and polyhedral sets. A nonempty set $C\subseteq \R^{d}$ is said to be a \emph{cone} if for every ${\bm x}\in C$, the set $C$ contains the positive ray $\R_{+}{\bm x}=\{t{\bm x}\colon t>0\}$ spanned by ${\bm x}$. A translation of a cone by a non zero vector is called an \emph{affine convex cone}. A cone $C\subseteq \R^{d}$ is said to be \emph{finitely generated} if it can be written as
	$$C=\left\{\sum\limits_{i=1}^{p}t {\bm u}_{i}\colon {\bm u}_{i}\in \R^{d},\ t_{i}\geq 0,\ i=1,\ldots,p\right\}.$$
	
	For a given nonempty set $A\subseteq \R^{d}$, the smallest cone containing the set $A$ is called the \emph{positive hull} (or \emph{conical hull}) of $A$. This set is given by
	$$\cone(A)=\{t{\bm x}\colon {\bm x}\in A,\ t\geq 0\}.$$
	
	The positive hull is also said to be the \emph{cone generated by} $A$.
	
	Convex sets can be represented, but it requires the notion of faces. A point ${\bm x}$ in a convex set $C$ is called an \emph{extreme point}, if it cannot be written as the convex combination of two different points in $C$, i.e., if ${\bm x}$ is equal to $t{\bm u}+(1-t){\bm v}$ for some $0\leq t\leq 1$, with ${\bm u},{\bm v}\in C$, then ${\bm u}={\bm v}={\bm x}$. We denote by $\Ext(C)$ the set of the extreme points of a convex set $C$. A compact convex set is called a \emph{polytope} if it has a finite number of extreme points.
	
	Extreme points are special cases of \emph{faces} of a convex set. A convex subset ${\bm F}\subseteq C$ is called a \emph{face} of $C$ if for every ${\bm x}\in {\bm F}$ and every ${\bm y}, {\bm z}\in C$ such that ${\bm x}=t{\bm y}+(1-t){\bm z}$, with $0< t< 1$, we have that ${\bm y}, {\bm z}\in {\bm F}$. The \emph{dimension} of a face ${\bm F}$ of $C$ is the dimension of its affine hull. The 0-dimensional faces of $C$ are exactly the extreme points of $C$, and the bounded $1$-dimensional faces are called \emph{segments} or \emph{edges}. An \emph{extreme ray} of a convex set $C$ is the direction of an affine half-line, that is, a face of $C$.  A useful result about representation of closed convex sets in $\R^{d}$ is the following.
	
	\begin{theorem}[Krein-Milman theorem for unbounded convex sets]
		If a nonempty closed convex set $C\subseteq \R^{d}$ has at least one extreme point, i.e., does not contain an affine line, then $C$ can be written as the sum of the convex hull of its extreme points and the cone generated by its extreme rays.
	\end{theorem}

	A useful relation between faces and the convex hull of a set that we will use in this article is the following. A proof can be found in \cite[Section 18]{rockafellar1970convex}.
	
	\begin{theorem}\label{PropertyConvexHull}
		Let $C=\conv(A)\subseteq \R^{d}$ be the convex hull of a set $A\subseteq \R^{d}$ and let ${\bm F}\subseteq C$ be a nonempty face of $C$. Then ${\bm F}=\conv(A\cap {\bm F})$.
	\end{theorem}

	Some useful notion for closed convex sets corresponds to their \emph{normal cones}. Let ${\bm F}$ be a nonempty face of a closed convex set $C$. The \emph{opposite normal cone}\footnote{The word \emph{opposite} comes from the fact that the usual normal cone is related to the outward normal vectors of convex sets and in this article we will use the inward normal vectors.} $\hat{N}_{{\bm F}}(C)$ \emph{of} $C$ \emph{at} ${\bm F}$ is defined as
	$$\hat{N}_{{\bm F}}(C)=\left\{{\bm v}\in \R^{d}\colon \min\limits_{{\bm t}\in C}\left\langle {\bm v},{\bm t}\right\rangle=\left\langle {\bm v},{\bm p}\right\rangle,\ \forall {\bm p}\in {\bm F}\right\}.$$
	
	The \emph{opposite normal fan of} $C$ is the collection of all opposite normal cones of $C$:
	$$\hat{\N}(C)=\left\{\hat{N}_{{\bm F}}(C)\colon {\bm F}\ \text{is a proper face of}\ C\right\}.$$
	
	The following are simple statements on the normal fan:
	
	\begin{itemize}
		\item $\dim(\hat{N}_{{\bm F}}(C))=d-\dim({\bm F})$.
		\item If ${\bm F}$ is a face of ${\bm G}$, which is a face of $C$, then $\hat{N}_{{\bm G}}(C)$ is a face of $\hat{N}_{{\bm F}}(C)$.
		\item The set $\bigcup\limits_{{\bm F}\ \text{face of}\ C}\hat{N}_{{\bm F}}(C)$ is equal to $\R^{d}$. 
	\end{itemize}

	\cref{FigureExampleNormalCones} illustrate the opposite normal cones of a triangle.
	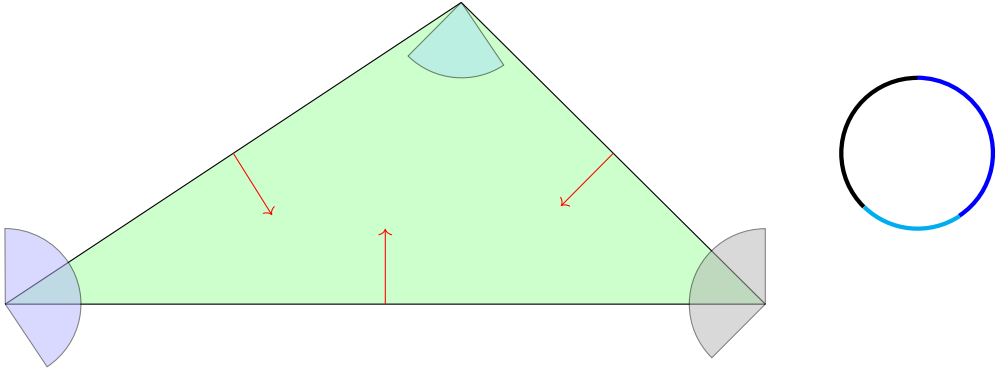
\begin{figure}[H]
		\centering
		\begin{tikzpicture}
			\draw[fill=green,opacity=0.2] (0,0)--(10,0)--(6,4)--(0,0);
			\path[thin] (0,0) edge (10,0);
			\path[thin] (10,0) edge (6,4);
			\path[thin] (0,0) edge (6,4);
			
			\path[thin,->,color=red] (3,2) edge (3.51,1.18);
			\path[thin,->,color=red] (5,0) edge (5,1);
			\path[thin,->,color=red] (8,2) edge (7.31,1.3);
			
			\draw[fill=blue!30,opacity=0.5] (0,0) -- (0.55,-0.83) arc[start angle=-56.31, end angle=90,radius=1cm] -- (0,0);
			
			\draw[fill=black!30,opacity=0.5] (10,0) -- (10,1) arc[start angle=90, end angle=225.29,radius=1cm] -- (10,0);
			
			\draw[fill=cyan!30,opacity=0.5] (6,4) -- (5.3,3.29) arc[start angle=225.29, end angle=303.69,radius=1cm] -- (6,4);
			
			\draw[blue,ultra thick]  (12.55,1.17) arc[start angle=-56.31, end angle=90,radius=1cm];
			
			\draw[black,ultra thick] (12,3) arc[start angle=90, end angle=225.29,radius=1cm];
			
			\draw[cyan,ultra thick](11.3,1.29) arc[start angle=225.29, end angle=303.69,radius=1cm];
			
		\end{tikzpicture}
		\caption{Example of the opposite normal cones of a triangle and the stratification of the circle $\SS^{1}$ given by them.}
		\label{FigureExampleNormalCones}
	\end{figure}

	\subsubsection{Fractal Geometry}\label{SubsectionFractalGeometry} 
	Let $\mathcal{C}(\R^{d})$ be the collection of all nonempty compact subsets of $\R^{d}$. The \emph{Hausdorff metric} $h$ on $\mathcal{C}(\R^{d})$ is defined as
	$$\forall A,B \in \mathcal{C}(\R^{d}),\ h(A,B)=\inf\{\varepsilon\colon A\subseteq B_{\varepsilon}\ \wedge\ B\subseteq A_{\varepsilon}\},$$
	
	\noindent where $A_{\varepsilon}=\{{\bm t}\in \R^{d}\colon \Vert {\bm t}-{\bm y}\Vert\leq \varepsilon,\ \text{for some}\ {\bm y}\in A\}$. With this metric $(\mathcal{C}(\R^{d}),h)$ is a complete metric space.
	
	A map $f:\R^{d}\to \R^{d}$ is said to be a \emph{contraction} if there exists $0<c<1$ such that ${\Vert f({\bm x})-f({\bm y})\Vert \leq c \Vert {\bm x}-{\bm y}\Vert}$ for all ${\bm x},{\bm y}\in \R^{d}$. Let $\{f_{i}\}_{i=1}^{N}$ be a set of contraction maps on $\R^{d}$ and define the map
	$$\begin{array}{llll}
		F: & (\mathcal{C}(\R^{d}),h) & \to & (\mathcal{C}(\R^{d}),h) \\ 
		& \multicolumn{1}{c}{A} & \mapsto & \bigcup\limits_{i=1}^{N}f_{i}(A) \\ 
	\end{array}$$
	
	This map is a contraction on $(\mathcal{C}(\R^{d}),h)$. By the Banach fixed-point theorem (or particularly the IFS theorem), there exists a unique set ${T\in \mathcal{C}(\R^{d})}$ (called \emph{digit tile}) such that $T=\bigcup\limits_{i=1}^{N}f_{i}(T)$. A way to approximate this set is by iterations
	\begin{equation}\label{ApproximationDigitTile}
		T=\lim\limits_{n\to \infty} F^{n}(T_{0}),
	\end{equation}
 
	\noindent where $T_{0}$ is an arbitrary compact set of $\R^{d}$ and the limit is with respect to the Hausdorff metric.
	
	Since the convex hull of a compact set in $\R^{d}$ is compact, the map $\conv:\mathcal{C}(\R^{d})\to \mathcal{C}(\R^{d})$, which gives for any set $A\in \mathcal{C}(\R^{d})$ its convex hull, is well defined and is well known to be continuous.
	
	\subsection{Topological dynamical systems}\label{SectionTopologicalDynamicalSystem}
	
	A \emph{topological dynamical system} is a triple $(X,T,G)$, where $(X,\rho)$ is a compact metric space, $G$ is a group of self-homeomorphisms of the space $X$ and $T:X\times G\to X$ is a continuous map, satisfying $T(x,e)=x$, and $T(T(x,g),h)=T(x,gh)$ for all $x\in X$ and $g,h\in G$. We denote $T^{g}$ the homeomorphism $T(\cdot,g)$.
	
	If $(X,\rho)$ is a compact metric space, we denote $\Homeo(X)$ the group of self-homeomorphisms of $X$. If $T\in \Homeo(X)$, we use $(X,T,\Z)$ to denote the topological dynamical system $(X,T,\{T^{n}\colon n\in \Z\})$. Similarly, if $T_{1},\ldots,T_{d}$ are $d$ commuting homeomorphisms on $X$, we use $(X,T,\Z^{d})$ to denote the topological dynamical system $(X,T,\left\langle\{T_{1},\ldots,T_{d}\}\right\rangle)$.
	
	For a point $x\in X$, we define its \emph{orbit} as the set $\mathcal{O}(x,G)=\{T^{g}(x)\colon g\in G\}$. If $A\subseteq X$, we say that $A$ is $G$-\emph{invariant} if for all $x\in A$, $\mathcal{O}(x,G)$ is included in $A$.
	
	If $(X,T,G)$ is a topological dynamical system, a subset $K\subseteq X$ is called a \emph{minimal set} if $K$ is closed, nonempty, $G$-invariant and has no proper closed nonempty invariant subsets, i.e., if $N\subseteq K$ is closed and $G$-invariant, then $N=\emptyset$ or $N=K$. In this case, we say that $(K,T|_{K},G)$ is a \emph{minimal system}, where $T|_{K}:K\times G\to K$ corresponds to the restriction of $T$ to $K$. It is easy to see that a system is minimal if and only if it is the closure orbit of all of its points.
	
	\begin{definition}
		Let $(X,T,\Z^{d})$, $(Y,T,\Z^{d})$ be two topological dynamical systems and $M\in GL(d,\Z)$. A \emph{homomorphism associated with} $M$ is a continuous map $\phi:X\to Y$ such that for all ${\bm n}\in \Z^{d}$, we have that $\phi\circ T^{{\bm n}}= T^{M{\bm n}}\circ \phi$. If $\phi$ is surjective, then $\phi$ is an \emph{epimorphism} and if it is invertible, then $\phi$ is an \emph{isomorphism}.
	\end{definition}

	Note that a homomorphism between two minimal systems is always an epimorphism. In the following, we fix the different notations that we will use throughout this article:
	
	\begin{itemize}
		\item We denote the set of all homomorphisms associated with $M$ between $(X,T,\Z^{d})$ and $(Y,T,\Z^{d})$ by $\Hom_{M}(X,Y,T,\Z^{d})$.
		
		\item The \emph{set of homomorphisms} between two dynamical systems, is defined as the collection of all of homomorphisms, i.e.,
		$$\Hom(X,Y,T,\Z^{d})=\bigcup\limits_{M\in GL(d,\Z)}\Hom_{M}(X,Y,T,\Z^{d}).$$
		
		\item In the special case where $M$ is the identity matrix, the homomorphisms are called \emph{factor maps} and we denote $\Fac(X,Y,T,\Z^{d})$ the collection of all factor maps between $(X,T,\Z^{d})$ and $(Y,T,\Z^{d})$. If a factor map is invertible, then it is called a \emph{conjugacy}.
		
		\item In the case $(X,T,\Z^{d})=(Y,T,\Z^{d})$, we simply denote these sets as $N_{M}(X,T,\Z^{d})$ and $N(X,T,\Z^{d})$. The last set is called the \emph{normalizer semigroup of} $(X,T,\Z^{d})$. A factor map is called an \emph{endomorphism}, and  a conjugacy is called an \emph{automorphism}. We denote the set of all endomorphisms and automorphisms of a topological dynamical system as $\End(X,T,\Z^{d})$ and $\Aut(X,T,\Z^{d})$, respectively.
		
		\item We define the \emph{linear representation semigroup} $\vec{N}(X,T,\Z^{d})$ \emph{of} $(X,T,\Z^{d})$ as the collection of all matrices $M\in GL(d,\Z)$ with ${N_{M}(X,T,\Z^{d})\neq \emptyset}$.
		
		\item A topological dynamical system $(X,T,\Z^{d})$ is said to be \emph{coalescent} if every endomorphism of $(X,T,\Z^{d})$ is an automorphism.  
	\end{itemize}

	Note that the linear representation semigroup of a topological dynamical system is an invariant under conjugation. Now, if $\phi\in N_{M_{1}}(X,T,\Z^{d})$ and $\psi\in N_{M_{2}}(X,T,\Z^{d})$, then $\phi\psi$ is in $N_{M_{1}M_{2}}(X,T,\Z^{d})$, so the sets $N_{M}(X,T,\Z^{d})$ are not semigroups (except if $M$ is the identity matrix). Now, even though the matrices $M\in GL(d,\Z)$ are invertible in $\Z^{d}$, the linear representation semigroup $\vec{N}(X,T,\Z^{d})$ is not necessarily a group, since the existence of a homomorphism associated with a matrix $M$ does not necessarily imply the existence of a homomorphism associated with the matrix $M^{-1}$.
	
	The groups $\left\langle T\right\rangle$ and $\Aut(X,T,\Z^{d})$ are normal subgroups of $N^{*}(X,T,\Z^{d})$ (the group of isomorphisms), and the centers of $N^{*}(X,T,\Z^{d})$ and $\Aut(X,T,\Z^{d})$ are the same.  In fact, we have the following short exact sequences
		\begin{align}
			1 & \to & \left\langle T\right\rangle\quad\quad\quad & \to & \Aut(X,T,\Z^{d})\quad & \to & \Aut(X,T,\Z^{d})/\left\langle T\right\rangle \quad & \to & 1\,\,\\
			1 & \to & \Aut(X,T,\Z^{d})\quad & \to&  N^{*}(X,T,\Z^{d}) \quad& \to & \vec{N^{*}}(X,T,\Z^{d})\quad\quad & \to& 1. \label{ExactSequenceForNormalizer}
		\end{align}
	
	If $\pi:(X,T,\Z^{d})\to (Y,T,\Z^{d})$ is a factor map between two minimal systems, and there exists $y\in Y$ such that $|\pi^{-1}(\{y\})|=1$, then this property is satisfied in a $G_{\delta}$ dense subset $Y_{0}\subseteq Y$. In this case, we say that $\pi$ is \emph{almost 1-to-1}. If for some $c>0$, $|\pi^{-1}(\{y\})|=c$ for all $y$ in a $G_{\delta}$ dense subset of $Y$, then we say that $\pi$ is \emph{almost} $c$-\emph{to-1}. If $|\pi^{-1}(\{y\})|\leq c<\infty$ for all $y\in Y$, we say that $\pi$ is \emph{finite-to-1}.
	
	An important type of topological dynamical systems are the equicontinuous ones. A topological dynamical system $(X,T,\Z^{d})$ is said to be \emph{equicontinuous} if the set of maps $\{T^{{\bm n}}\colon {\bm n}\in \Z^{d}\}$ forms an equicontinuous family of homeomorphisms. The equicontinuous systems are, in some sense, the simplest dynamical systems, in fact, there exists a complete characterization of them, and every topological dynamical system has at least one equicontinuous factor: the system given by one point. In fact, for every topological dynamical system there exists a \emph{maximal equicontinuous factor}, i.e., a factor $\pi_{eq}:(X,T,\Z^{d})\to(X_{eq},T_{eq},\Z^{d})$ such that $(X_{eq},T_{eq},\Z^{d})$ is an equicontinuous system and for every equicontinuous factor $\pi:(X,T,\Z^{d})\to (Y,T,\Z^{d})$, there exists a factor map $\phi:(X_{eq},T_{eq},\Z^{d})\to (Y,T_{eq},\Z^{d})$ such that $\pi=\phi\circ \phi_{eq}$. 
	
	\subsection{Measure-preserving systems}
	
	A \emph{measure-preserving system} is a 4-tuple $(X,\mu,T,G)$, where $(X,\mathcal{F},\mu)$ is a probability space and $G$ is a countable group of measurable and measure-preserving transformations acting on $X$ (where the action is denoted by $T$), i.e., $\forall A \in \mathcal{F},\forall g\in G,\ \mu(T^{g^{-1}}A)=\mu(A)$.
	
	We say that $(X,\mu,T,G)$ is \emph{ergodic} if for all $A\in \mathcal{F}$ we have that
	$$\left[(\forall g \in G)\ \mu(T^{g^{-1}}(A)\Delta A)=0\right] \implies \mu(A)=0\ \vee\ \mu(A)=1.$$
	
	We now recall the notions of measurable homomorphisms in the measure-theoretic framework.
	
	Let $(X,\mu,T,G)$ and $(Y,\nu,T,G)$ be measure-preserving systems and $M\in GL(d,\Z)$. A \emph{measurable homomorphism associated with} $M$ is a measure-preserving map $\phi:X'\to Y'$ where $X'$, $Y'$ are measurable subsets of $X, Y$ respectively, with $\mu(X')=\nu(Y')=1$ and for any $g\in G$, $T^{g}(X')\subseteq X'$, $T^{g}(Y')\subseteq Y'$ such that for any ${\bm n}\in \Z^{d}$ we have that $\phi \circ T^{{\bm n}}=T^{M{\bm n}}\circ \phi$ in $X'$.
	
	If there is a measurable factor map $\phi$ between $X$ and $Y$, then $X$ is said to be an \emph{extension} of $Y$. If $\phi$ is a bi-measurable bijection, we say that $\phi$ is a \emph{measurable conjugacy} and in this case $(X,\mu,T,G)$ and $(Y,\nu,T,G)$ are \emph{metrically isomorphic}.
	
	For topological dynamical $\Z^{d}$-actions, we always have at least one invariant probability measure (in fact, at least one ergodic probability measure). We define $\mathcal{M}(X,T,\Z^{d})$ the set of all invariant probability measures. This set is convex and compact on the weak-* topology. We say that $(X,T,\Z^{d})$ is \emph{uniquely ergodic} if ${|\mathcal{M}(X,T,\Z^{d})|=1}$, and \emph{strictly ergodic} if it is minimal and uniquely ergodic.
	
	In the special case of strictly ergodic topological dynamical systems $(X,T,\Z^{d})$, $(Y,T,\Z^{d})$ we denote $m\Hom(X,Y,T,\Z^{d})$, $m\Fac(X,Y,T,\Z^{d})$ the collection of all measurable homomorphisms and factor maps between $(X,T,\Z^{d})$ and $(Y,T,\Z^{d})$, respectively.
	
	\subsection{d-dimensional Odometer systems}
	
	Let $Z_{0}\geq Z_{1}\geq \ldots \geq Z_{n}\geq Z_{n+1}\geq \ldots$ be a nested sequence of finite-index subgroups of $\Z^{d}$ such that $\bigcap\limits_{n\geq 0}Z_{n}=\{{\bm 0}\}$, and let $\alpha_{n}:\Z^{d}/Z_{n+1}\to \Z^{d}/Z_{n}$ be the function induced by the inclusion map. Following the notation in \cite{cortez2008godometers}, we consider the inverse limit of these groups
	$$\overleftarrow{\Z^{d}}_{(Z_{n})}=\lim\limits_{\leftarrow n} (\Z^{d}/Z_{n},\alpha_{n}),$$
	
	\noindent i.e., $\overleftarrow{\Z^{d}}_{(Z_{n})}$ is the subset of the product $\prod\limits_{n\geq 0} \Z^{d}/Z_{n}$ consisting of the elements $\overleftarrow{g}=({\bm g}_{n})_{n\geq 0}$ such that $\alpha_{n}({\bm g}_{n+1})={\bm g}_{n}\ (\text{mod}\ Z_{n})$ for all $n\geq 0$. This set is a group equipped with the addition defined coordinate-wise, i.e.,
	$$\overleftarrow{g}+\overleftarrow{h} = ({\bm g}_{n}+{\bm h}_{n})_{n\geq 0}.$$
	
	Every group $\Z^{d}/Z_{n}$ is endowed with the discrete topology, so $\prod\limits_{n\geq 0}(\Z^{d}/Z_{n})$ is a compact metric space. The odometer $\overleftarrow{\Z^{d}}_{(Z_{n})}$ is a compact topological group whose topology is spanned by the cylinder sets
	$$[{\bm a}]_{n}=\left\{\overleftarrow{g}\in \overleftarrow{\Z^{d}}_{(Z_{n})}: {\bm g}_{n}={\bm a}\right\},$$
	
	\noindent with ${\bm a}\in \Z^{d}/Z_{n}$, and $n\geq 0$. Now, consider the group homomorphism $\kappa_{(Z_{n})}:\Z^{d}\to \prod\limits_{n\geq 0} \Z^{d}/Z_{n}$  defined for ${\bm n}\in \Z^{d}$ by
	$$\kappa_{(Z_{n})}({\bm n})=[{\bm n}\ (\text{mod}\ Z_{n})]_{n\geq 0}.$$
	
	The image of $\Z^{d}$ by $\kappa_{(Z_{n})}$ is dense in $\overleftarrow{\Z^{d}}_{(Z_{n})}$, so the $\Z^{d}$-action ${\bm n}(\overleftarrow{g})=\kappa_{(Z_{n})}({\bm n})+\overleftarrow{g}$, with ${\bm n}\in \Z^{d}$, $\overleftarrow{g}\in \overleftarrow{\Z^{d}}_{(Z_{n})}$, is well defined and $(\overleftarrow{\Z^{d}}_{(Z_{n})},+_{(Z_{n})},\Z^{d})$ is a minimal equicontinuous system as proved in \cite{cortez2008godometers}. We call $(\overleftarrow{\Z^{d}}_{(Z_{n})},+_{(Z_{n})},\Z^{d})$ an \emph{odometer system}. Odometer systems have been extensively studied before (see \cite{cortez2006toeplitz,cortez2008godometers,downarowicz2008finiterank}). The next result characterizes the factor odometers systems of a fixed odometer system.
	
	\begin{lemma}\cite[Lemma 1]{cortez2006toeplitz}\label{CharacterizationFactorOdometer}
		Let $(\overleftarrow{\Z^{d}}_{(Z_{n}^{j})},+_{(Z_{n}^{j})},\Z^{d})$ be two odometer systems $(j=1,2)$.  There exists a factor map $\pi:(\overleftarrow{\Z^{d}}_{(Z_{n}^{1})},+_{(Z_{n}^{1})},\Z^{d})\to (\overleftarrow{\Z^{d}}_{(Z_{n}^{2})},+_{(Z_{n}^{2})},\Z^{d})$ if and only if for every $Z_{n}^{2}$ there exists some $Z_{m}^{1}$ such that $Z_{m}^{1}\leq Z_{n}^{2}$.
	\end{lemma} 
	
\subsection{Symbolic Dynamics}
Let $\A$ be a finite alphabet and $d\geq 1$ be an integer. We define a topology on $\A^{\Z^{d}}$ by endowing $\A$ with the discrete topology and considering in $\A^{\Z^{d}}$ the product topology, which is generated by cylinders. Since $\A$ is finite, $\A^{\Z^{d}}$ is a metrizable compact space. In this space $\Z^{d}$ acts by translations (or shifts), defined for every ${\bm n}\in \Z^{d}$ as
$$S^{{\bm n}}(x)_{{\bm k}}=x_{{\bm n}+{\bm k}},\ x\in \A^{\Z^{d}},\ {\bm k}\in \Z^{d}.$$
	
The $\Z^{d}$-action $(\A^{\Z^{d}},S,\Z^{d})$ is called the \emph{fullshift}. 
	
Let $P\subseteq \Z^{d}$ be a finite subset. A \emph{pattern} is an element $\texttt{p}\in \A^{P}$. We say that $P$ is the \emph{support} of $\texttt{p}$, and we denote $P=\supp(\texttt{p})$. A pattern \emph{occurs in} $x\in \A^{\Z^{d}}$, if there exists ${\bm n}\in \Z^{d}$ such that $\texttt{p}=x|_{{\bm n}+P}$ (identifying ${\bm n}+P$ with $P$ by translation). In this case we denote it $\texttt{p}\sqsubseteq x$ and we call this ${\bm n}$ an \emph{occurrence in} $x$ of $\texttt{p}$.
	
A \emph{subshift} $(X,S,\Z^{d})$ is given by a closed subset $X\subseteq \A^{\Z^{d}}$ which is invariant by the $\Z^{d}$-action. A subshift also can be defined by its \emph{language}. For $P\Subset \Z^{d}$ we define
$$\mathcal{L}_{P}(X)=\{\texttt{p}\in \A^{P}: \exists x \in X,\ \texttt{p}\sqsubseteq x\}.$$
	
We define the \emph{language} of a subshift $X$ by
$$\mathcal{L}(X)=\bigcup\limits_{P\Subset \Z^{d}}\mathcal{L}_{P}(X).$$
	
Let $(X,S,\Z^{d})$ be a subshift and $x\in X$. We say that ${\bm p}\in \Z^{d}$ is a \emph{period} of $x$ if for all ${\bm n}\in \Z^{d}$, $x_{{\bm n}+ {\bm p}}=x_{{\bm n}}$. We say that the subshift $(X,S,\Z^{d})$ is \emph{aperiodic} if there are no nontrivial periods.
	
Let $\B$ be another finite alphabet and $Y\subseteq \B^{\Z^{d}}$ be a subshift. For $P\Subset \Z^{d}$, we define a $P$-\emph{block map} as a map of the form $\Phi: \mathcal{L}_{P}(X)\to \B$. This induces a factor map $\phi:X\to Y$ given by
$$\phi(x)_{{\bm n}}= \Phi(x|_{{\bm n}+ P}).$$
	
The map $\phi$ is called the \emph{sliding block code} induced by $\Phi$ and $P$ is the support of the map $\phi$. In most of the cases we may assume that the support of the sliding block codes is a ball of the form $B({\bm 0},r)$, for $r\in \NN$. We define the \emph{radius} (and we denote by $r(\phi)$) as the infimum of $r\in\NN$ such that we can define a $B({\bm 0},r)$-block map which induces it. The next theorem characterizes the factor maps between two subshifts.
	
\begin{theorem}[Curtis-Hedlund-Lyndon] Let $(X,S,\Z^{d})$ and $(Y,S,\Z^{d})$ be two subshifts. A map $\phi:(X,S,\Z^{d})\to (Y,S,\Z^{d})$ is a factor map if and only if there exists a $B({\bm 0},r)$-block map $\Phi:\mathcal{L}_{B({\bm 0},r)}(X)\to \mathcal{L}_{1}(Y)$, such that $\phi(x)_{{\bm n}}=\Phi(x|_{{\bm n}+B({\bm 0},r)})$, for all ${\bm n}\in \Z^{d}$ and $x\in X$.
\end{theorem}
	
For homomorphisms we have a similar characterization, but we need to make a slight variation of this theorem.
	
\begin{theorem}[Curtis-Hedlund-Lyndon theorem for homomorphisms] Let $(X,S,\Z^{d})$ and $(Y,S,\Z^{d})$ be two subshifts and ${M\in GL(d,\Z)}$. A map $\phi:(X,S,\Z^{d})\to (Y,S,\Z^{d})$ is a homomorphism associated with $M$ if and only if there exists a $B({\bm 0},r)$-block map $\Phi:\mathcal{L}_{B({\bm 0},r)}(X)\to \mathcal{L}_{1}(Y)$, such that $\phi(x)_{{\bm n}}=\Phi(x|_{M^{-1}{\bm n}+B({\bm 0},r)})$, for all ${\bm n}\in \Z^{d}$ and $x\in X$.
\end{theorem}
	
\begin{proof}
We will only prove the nontrivial implication. Let $\phi:(X,S,\Z^{d})\to (Y,S,\Z^{d})$ be a homomorphism and let $r>0$ be such that $x|_{B({\bm 0},r)}=y|_{B({\bm 0},r)}$, implies $\phi(x)_{{\bm 0}}=\phi(y)_{{\bm 0}}$. Then, the local map $\Phi(x|_{B({\bm 0},r)})=\phi(x)_{{\bm 0}}$ is well defined by the very definition of $r$. Finally, note that $\phi(x)_{{\bm h}}= S^{{\bm h}}\phi(x)_{{\bm 0}}= \phi(S^{M^{-1}{\bm h}}x)_{{\bm 0}}=\Phi(x|_{M^{-1}{\bm h}+B({\bm 0},r)})$, which proves the claim.
\end{proof}

This means, for any homomorphism $\phi$ we can define a \emph{radius} (also denoted by $r(\phi)$), as the infimum of $r\in\NN$ such that we can define a $B({\bm 0},r)$-block map which induces it. In the case $r(\phi)=0$, we say that $\phi$ is induced by a \emph{letter-to-letter map}.
	
\subsection{Nondeterministic directions of a topological dynamical system}\label{SectionNonDeterminsiticDirections}

An interesting notion in the study of higher-dimensional dynamical systems is the so-called \emph{nonexpansive subspaces}, introduced by M. Boyle and D. Lind in \cite{boyle1997expansive}. When the phase space $X$ is infinite such subspaces always exist \cite[Theorem 3.7]{boyle1997expansive}. We will only focus on hyperplanes in $\R^{d}$, which leads to the notion of deterministic/nondeterministic directions. Let $\SS^{d-1}$ be the unit $(d-1)$-dimensional sphere. For ${\bm v}\in \SS^{d-1}$ define $H_{{\bm v}}=\{{\bm x}\in \R^{d}\colon \left\langle {\bm x},{\bm v}\right\rangle < 0 \}$ to be the open half-space with outward unit normal ${\bm v}$. We identify the set $\mathbb{H}_{d}$ of all half-spaces in $\R^{d}$ with the sphere $\SS^{d-1}$ using the parametrization ${\bm v} \longleftrightarrow H_{\bm v}$.

\begin{definition}
	Let $(X,S,\Z^{d})$ be a subshift and ${\bm v}$ be a unit vector of  $\R^{d}$. Then ${\bm v}$ is \emph{deterministic} for $(X,S,\Z^{d})$ if for all $x,y\in X$ we have that
	$$x|_{H_{{\bm v}}\cap \Z^{d}}=y|_{H_{{\bm v}}\cap \Z^{d}} \implies x=y.$$
	
	If ${\bm v}$ does not satisfy this condition, we say that ${\bm v}$ is \emph{nondeterministic for} $(X,S,\Z^{d})$. 
\end{definition}

For a subshift $(X,S,\Z^{d})$ we denote $\DD(X,S,\Z^{d})$ and $\ND(X,S,\Z^{d})$ the sets of deterministic and nondeterministic directions for $(X,S,\Z^{d})$, respectively. Using an analogous argument of \cite[Lemma 3.4]{boyle1997expansive} we conclude $\DD(X,S,\Z^{d})$ is an open set and $\ND(X,S,\Z^{d})$ is a compact set.

In \cite{guillon2015determinism} was introduced the notion of \emph{direction of determinism} for two-dimensional subshifts and in \cite{cyr2015nonexpansive} these objects were used to prove a weak version of Nivat's conjecture.

The following result establishes a link between nondeterministic directions and the linear representation group of a subshift.

\begin{proposition}\label{NormalizerActsInNonExpansiveHalfspaces}
	Let $(X,S,\Z^{d})$ be a subshift. Then, for all ${\bm v}\in \ND(X,S,\Z^{d})$ and $M\in \vec{N}(X,S,\Z^{d})$ the vector $(M^{*})^{-1}{\bm v}/\Vert (M^{*})^{-1}{\bm v}\Vert$ is a nondeterministic direction for $(X,S,\Z^{d})$, where $M^{*}$ is the algebraic adjoint of $M$.
\end{proposition}

\begin{proof}
	If ${\bm v}$ is in $\ND(X,S,\Z^{d})$, there exists $x\neq y\in X$ with  $x|_{H_{{\bm v}}\cap \Z^{d}}=y|_{H_{{\bm v}}\cap \Z^{d}}$. Set $M\in \vec{N}(X,T,\Z^{d})$ and $\phi\in N_{M}(X,S,\Z^{d})$. Then, we have that $\phi(x)|_{(MH_{{\bm v}})+{\bm n}}=\phi(y)|_{(MH_{{\bm v}})+{\bm n}}$, where ${\bm n}$ is a vector of radius $r(\phi)$. We note that  $S^{{\bm n}}\phi(x)|_{MH_{{\bm v}}}=S^{{\bm n}}\phi(y)|_{MH_{{\bm v}}}$. We conclude that $(M^{*})^{-1}{\bm v}/\Vert (M^{*})^{-1}{\bm v}\Vert \in (X,S,\Z^{d})$.
\end{proof}		

\subsection{Multidimensional constant-shape substitutions}

\subsubsection{Definitions and basic properties} Let $L\in \mathcal{M}(d,\Z)$ be an expansion matrix, i.e, $\Vert L\Vert>1$ and $\Vert L^{-1}\Vert <1$, such that ${L(\Z^{d})\subseteq \Z^{d}}$. Let $F$ be a fundamental domain of $L(\Z^{d})$ in $\Z^{d}$ with ${\bm 0}$ in $F$ and $\A$ be a finite alphabet. A \emph{multidimensional constant-shape substitution} $\zeta$ is a map $\A\to \A^{F}$. The set $F$ is called the \emph{support} of the substitution. These are multidimensional analogue of the so-called one-dimensional constant-length substitutions. \cref{examplesubstitution} shows an example of a constant-shape substitution with $L=\left(\begin{array}{cc}
	2 & 0\\0 & 2
\end{array}\right)$ and $F=\left\{\dbinom{0}{0},\dbinom{1}{0},\dbinom{0}{1},\dbinom{-1}{-1}\right\}$.

\begin{figure}[H]
	\centering
	\begin{tikzpicture}
		\node(s1) at (0,0)[scale=0.05]{\includegraphics{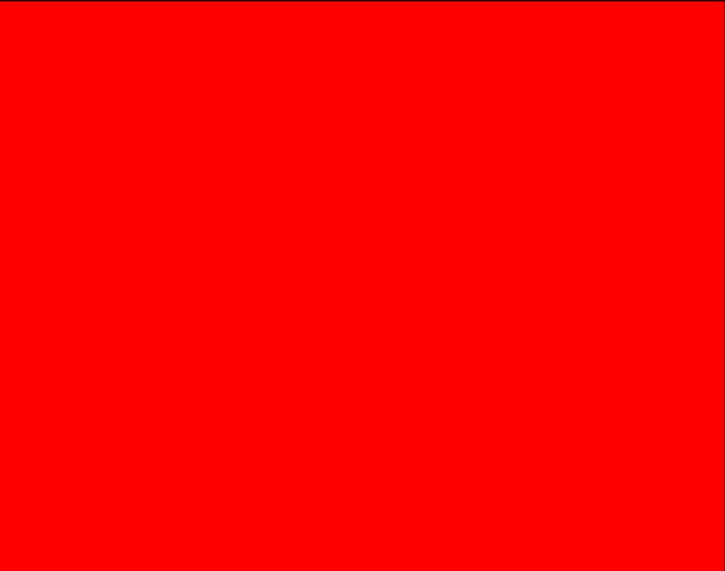}};
		
		\node(s2) at (0.7,0)[scale=1]{$\mapsto$};
		
		\node(s3) at (2.5,0)[scale=.06]{\includegraphics{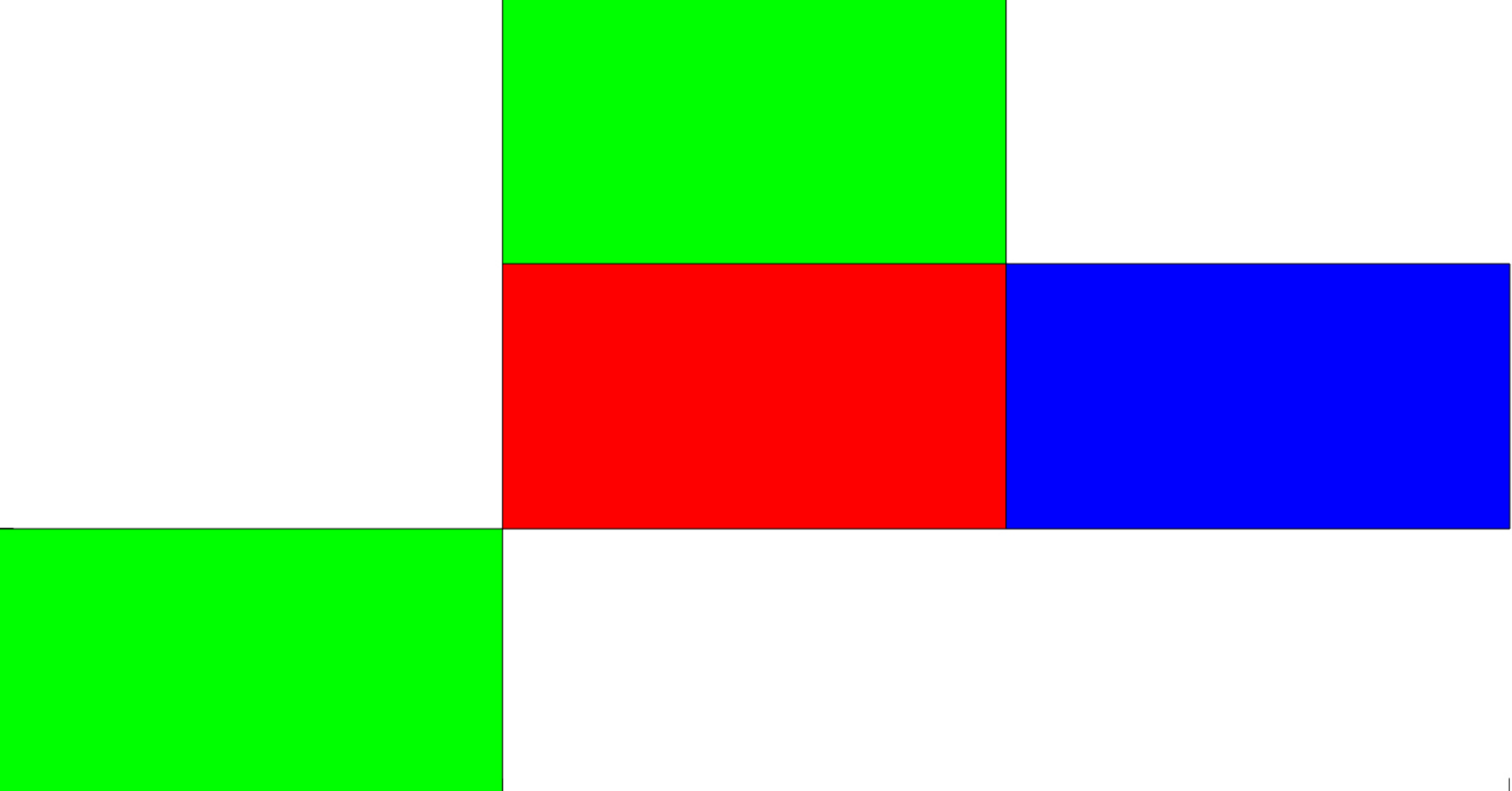}};
		
		\node(s4) at (7,0)[scale=0.05]{\includegraphics{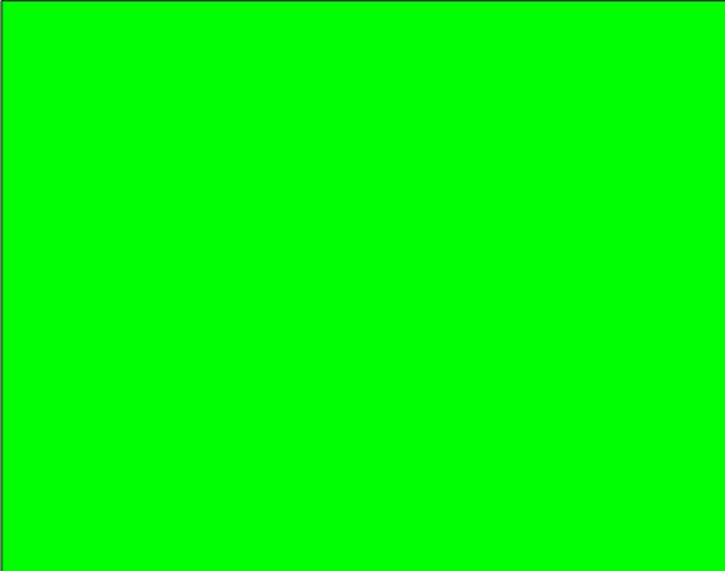}};
		
		\node(s5) at (7.7,0)[scale=1]{$\mapsto$};
		
		\node(s6) at (9.5,0)[scale=.06]{\includegraphics{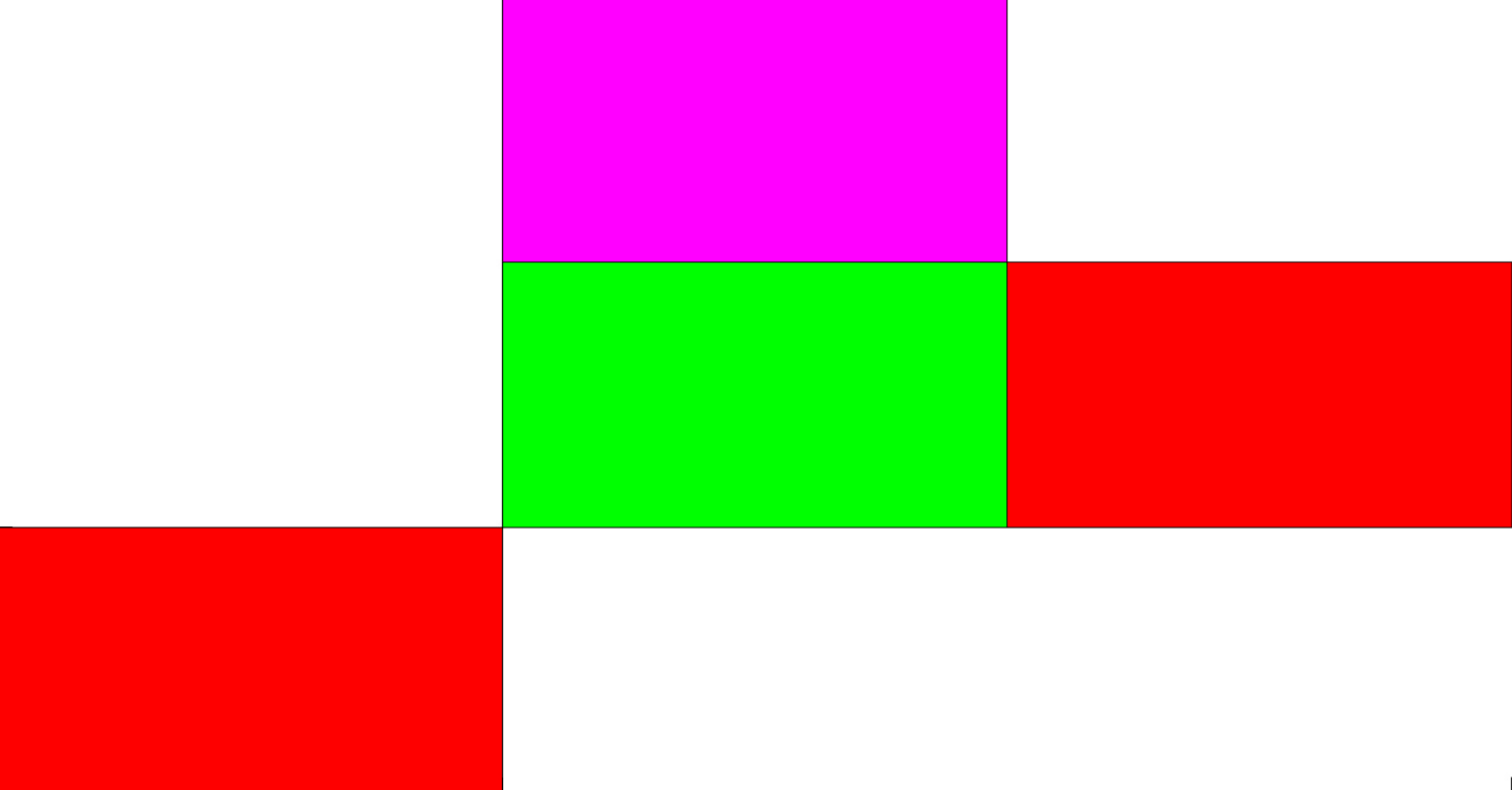}};
		
		\node(s7) at (0,-2)[scale=0.05]{\includegraphics{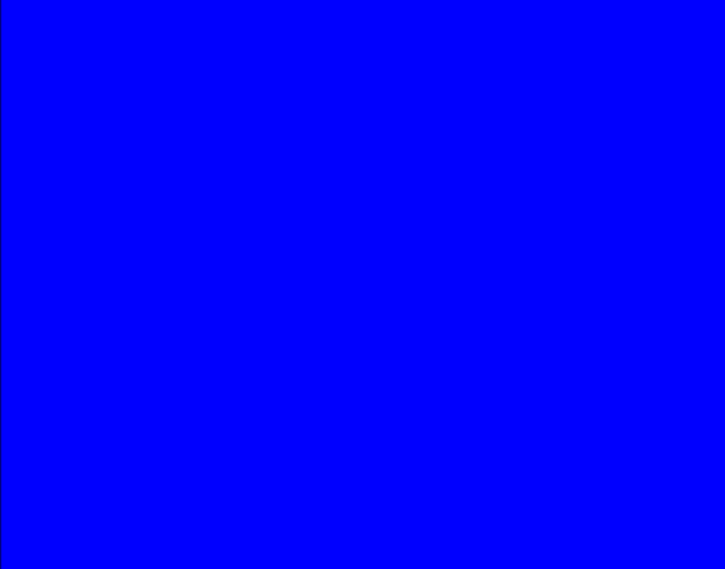}};
		
		\node(s8) at (0.7,-2)[scale=1]{$\mapsto$};
		
		\node(s9) at (2.5,-2)[scale=.06]{\includegraphics{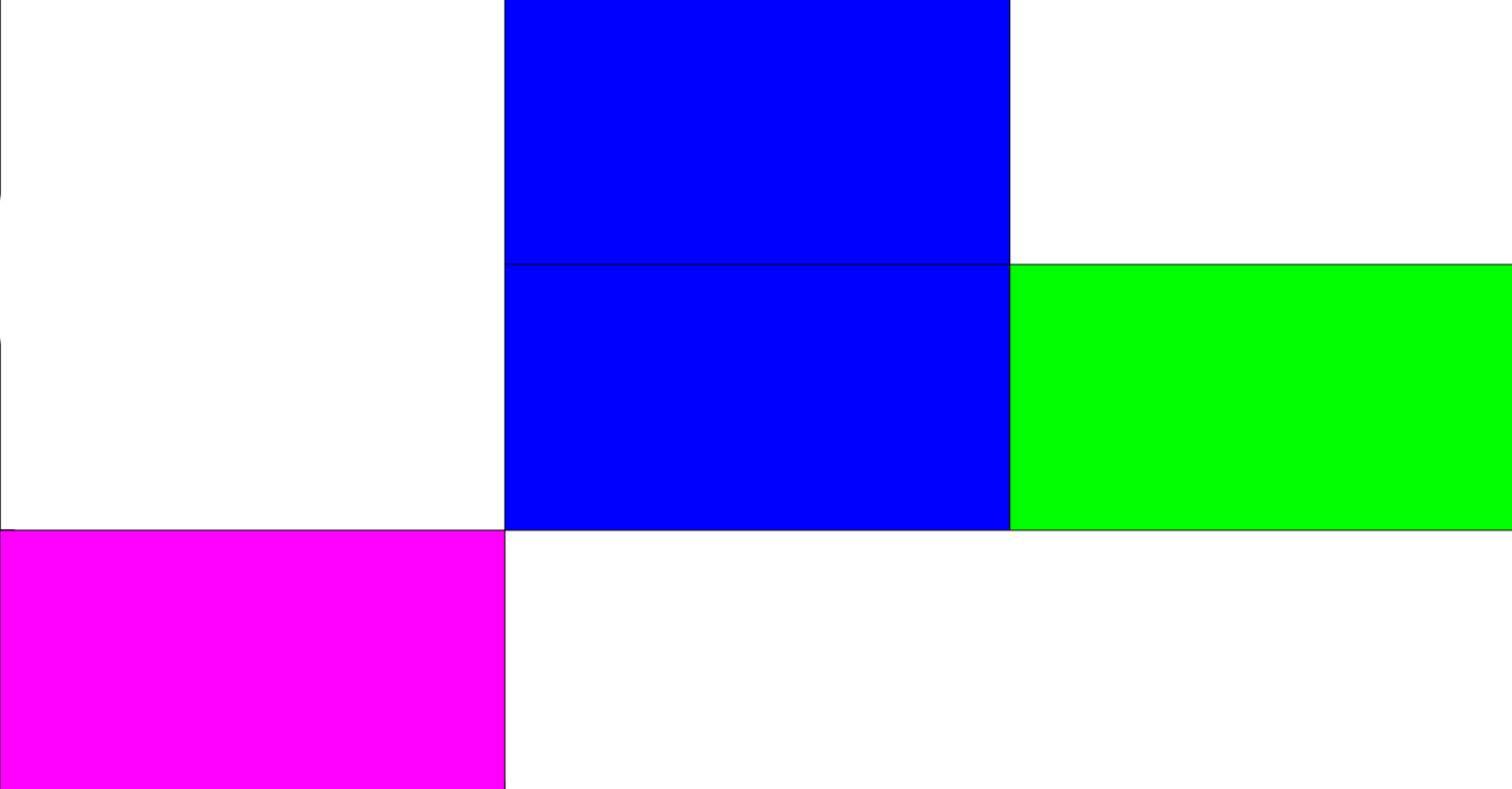}};
		
		\node(s10) at (7,-2)[scale=0.05]{\includegraphics{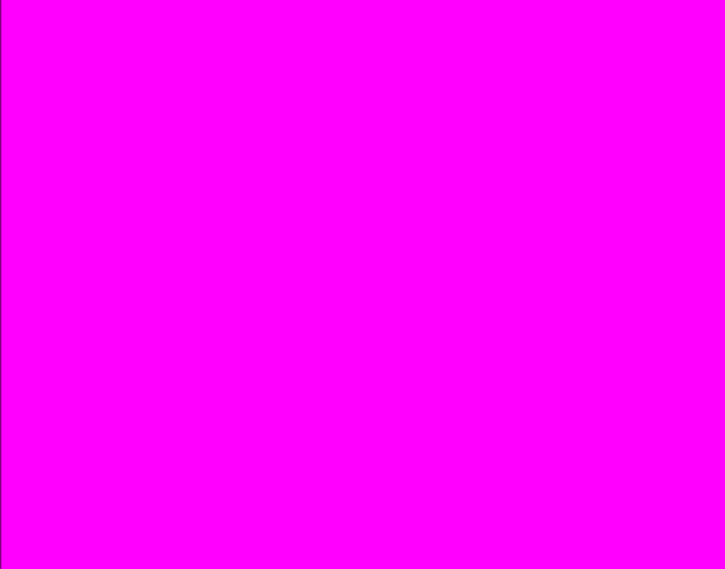}};
		
		\node(s11) at (7.7,-2)[scale=1]{$\mapsto$};
		
		\node(s12) at (9.5,-2)[scale=.06]{\includegraphics{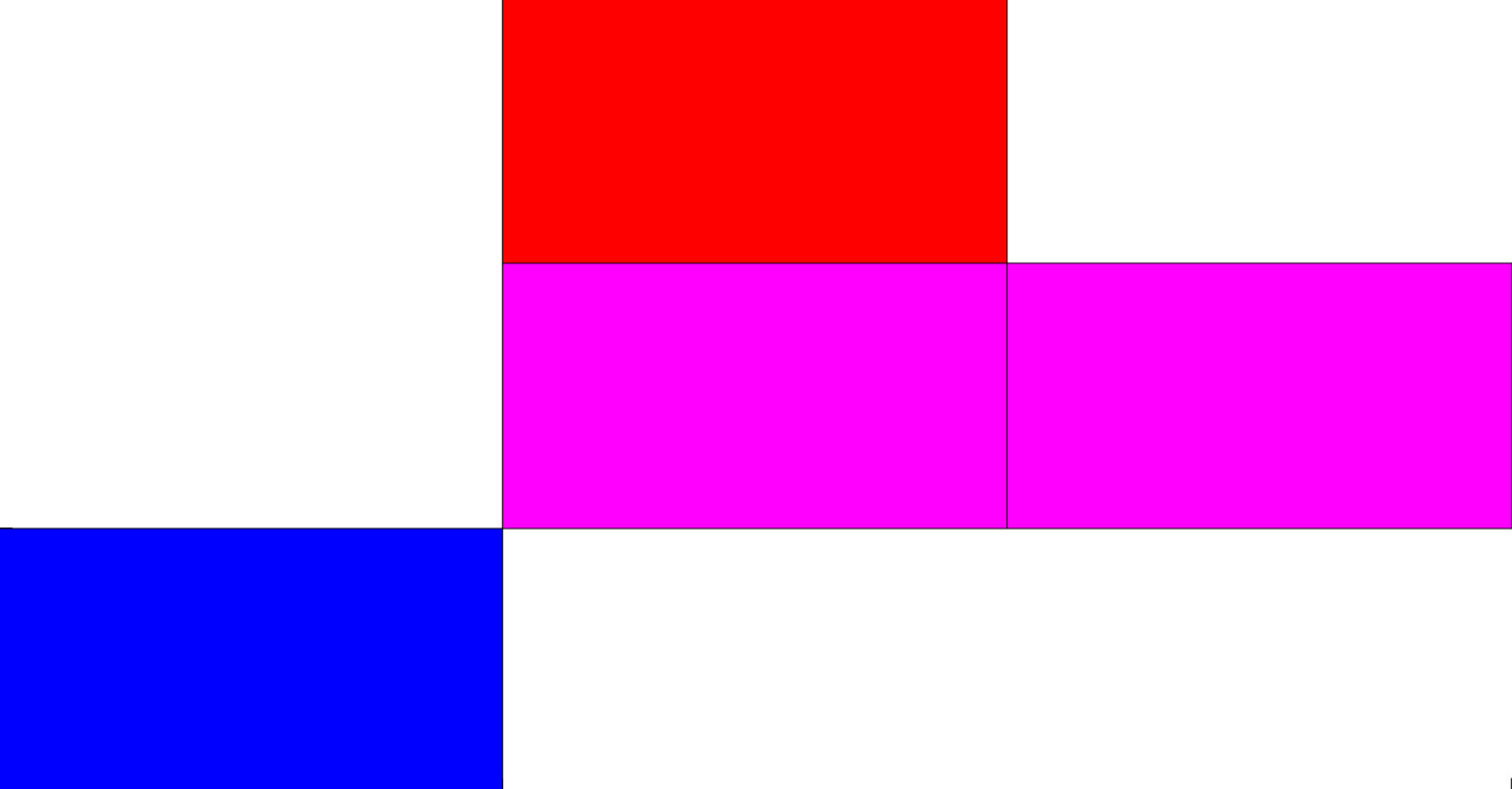}};
		
	\end{tikzpicture}
	\caption{An example of a constant-shape substitution over a four-letter alphabet.}
	\label{examplesubstitution}
\end{figure}

In the literature, constant-shape substitutions with a positive diagonal expansion matrix $L=\diag(l_{i})_{i=1,\ldots,d}$ and support equal to the standard $d$-dimensional parallelepiped $F_{1}=\prod\limits_{i=1}^{d}\llbracket 0,l_{i}-1\rrbracket$ are called \emph{block substitutions}.

Examples of constant-shape substitutions can be generated via constant-length substitutions as follows: Let $\{\zeta_{i}\}_{i=1}^{d}$ be $d$ aperiodic one-dimensional constant-length substitutions with alphabet $\A_{i}$ and length $q_{i}$ for $1\leq i\leq d$. We define the \emph{product substitution of} $\{\zeta_{i}\}_{i=1}^{d}$ as the constant-shape substitution $\zeta$ with alphabet $\A=\prod\limits_{i=1}^{d}\A_{i}$, expansion matrix equal to $L_{\zeta}=\diag(q_{i})$ and support $F_{1}^{\zeta}=\prod\limits_{i=1}^{d}\llbracket 0,q_{i}-1\rrbracket$, defined as $\zeta(a_{1},\ldots,a_{d})_{\bm j}=(\zeta_{1}(a_{1})_{j_{1}},\ldots,\zeta_{d}(a_{d})_{j_{d}})$.

Every element of $\Z^{d}$ can be expressed in a unique way as ${\bm p}= L({\bm j})+{\bm f}$, with ${\bm j} \in \Z^{d}$ and ${\bm f}\in F$, so we can consider the substitution $\zeta$ as a map from $\A^{\Z^{d}}$ to itself given by
$$\zeta(x)_{L({\bm j})+{\bm k}}=\zeta(x(\bm{j}))_{{\bm k}}.$$

Given a substitution $\zeta$, we will denote $L_{\zeta}$ its expansion matrix and $F_{1}^{\zeta}$ its support. For any $n>0$ we define the $n$-th iteration of the substitution $\zeta^{n}:\A\to \A^{F_{n}^{\zeta}}$ with expansion matrix $L_{\zeta}^{n}$ and the supports of these substitutions satisfying the recurrence $F_{n+1}^{\zeta}=F_{1}^{\zeta}+L_{\zeta}(F_{n}^{\zeta})$ for all $n\geq 1$. From now on, we may assume that the sequence of supports $(F_{n}^{\zeta})_{n>0}$ is a F\o lner sequence. 

\cref{FirstsIterationsExample} shows the first three iterations of the substitution given in \cref{examplesubstitution}.

\begin{figure}[H]
	\centering
	\begin{tikzpicture}
		
		\node(s9) at (0,-3)[scale=0.04]{\includegraphics{red.pdf}};
		
		\node(s10) at (0.7,-3)[scale=1]{$\mapsto$};
		
		\node(s11) at (2.5,-3)[scale=.07]{\includegraphics{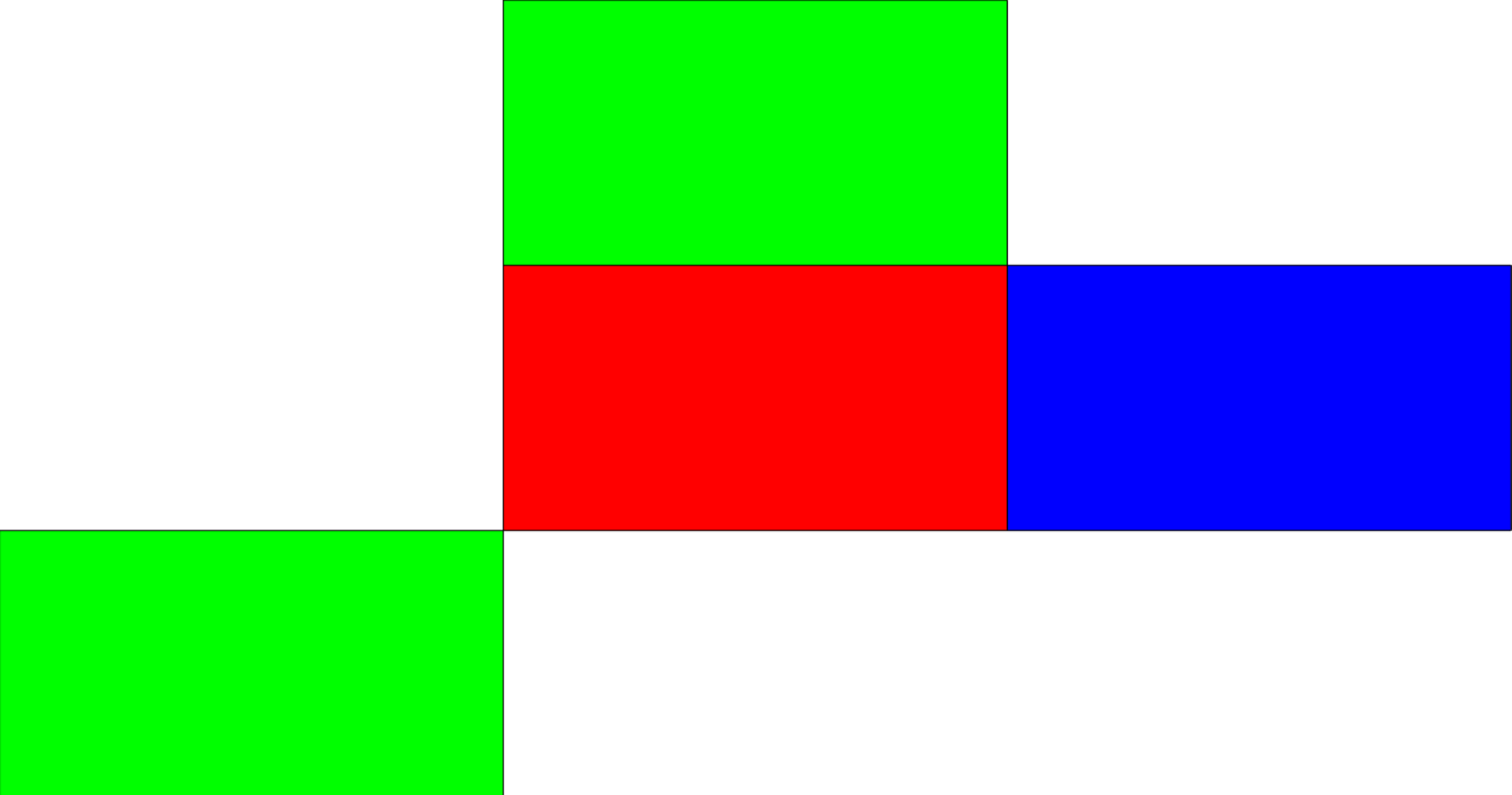}};
		
		\node(s12) at (4.3,-3)[scale=1]{$\mapsto$};
		
		\node(s13) at (7.5,-3)[scale=.15]{\includegraphics{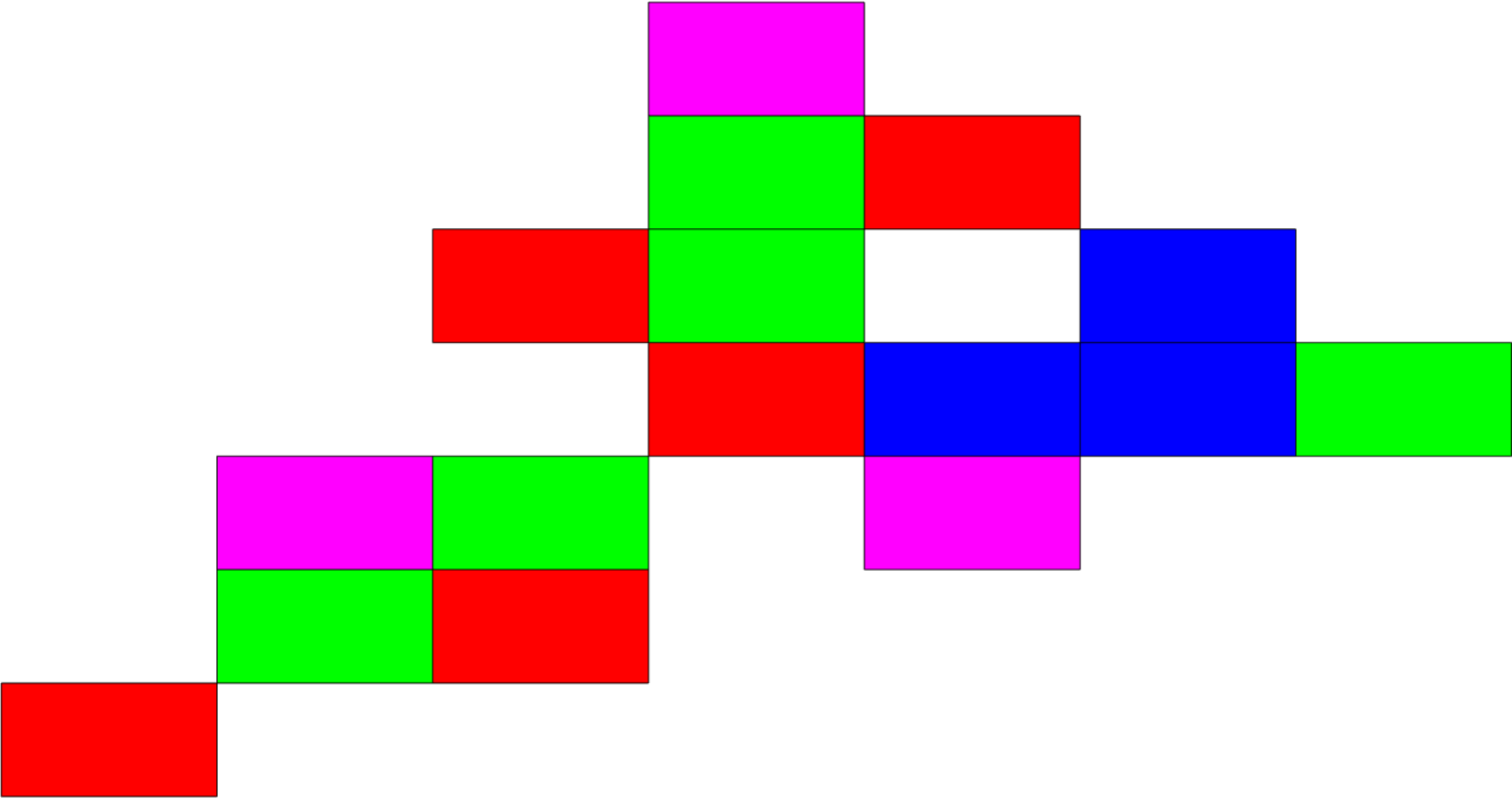}};
		
		\node(s14) at (0.7,-8)[scale=1]{$\mapsto$};
		
		\node(s15) at (7,-8)[scale=.25]{\includegraphics{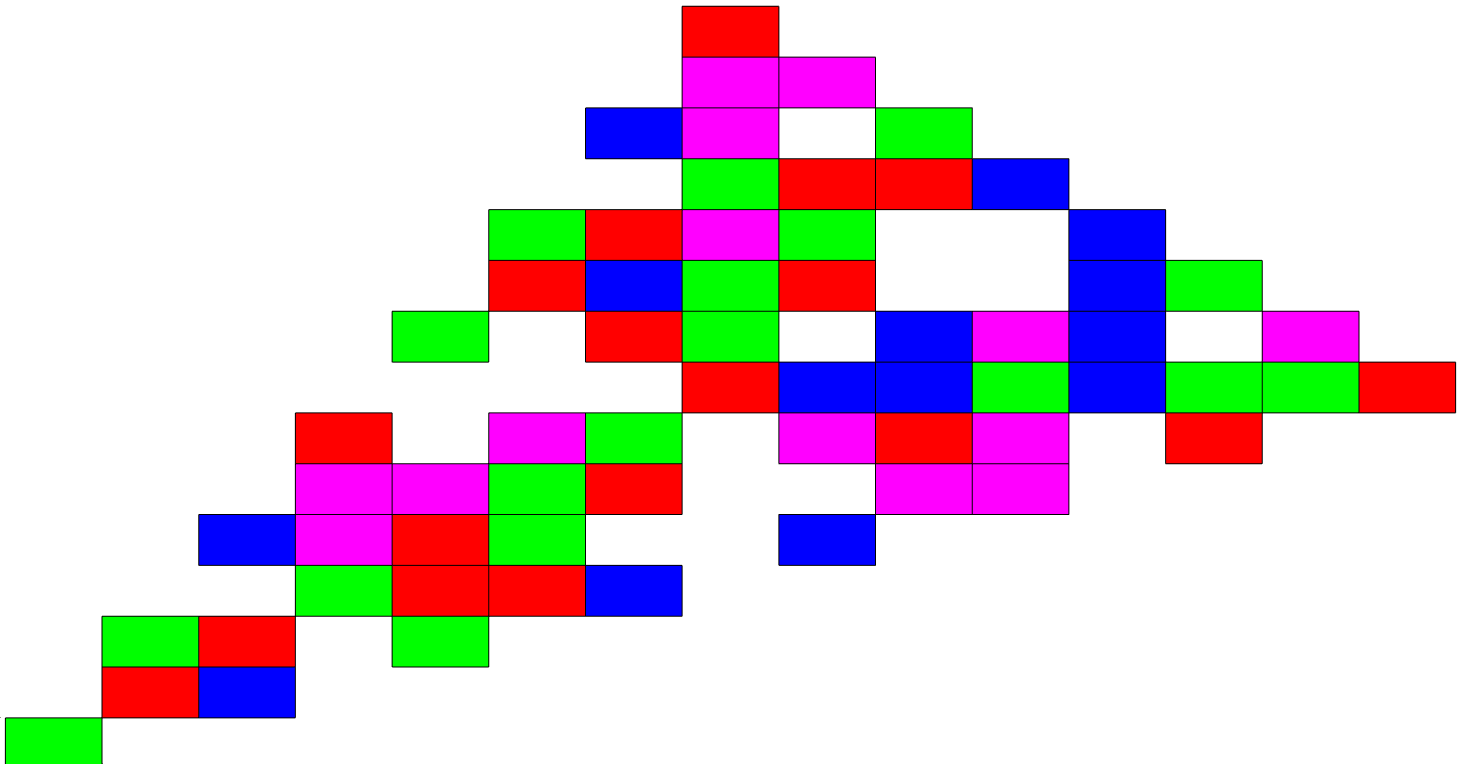}};
	\end{tikzpicture}
	\caption{An example of application of the first three iterations of the substitution illustrated in \cref{examplesubstitution}.}
	\label{FirstsIterationsExample}
\end{figure}

The \emph{language} of a substitution is the set of all patterns that appear in $\zeta^{n}(a)$, for some $n>0$, $a\in \A$, i.e.
$$\mathcal{L}_{\zeta}=\{\texttt{p}\colon \texttt{p}\sqsubseteq \zeta^{n}(a),\ \text{for some }n>0,\ a\in \A\}.$$

Using the language we define the subshift $X_{\zeta}$ associated with a substitution as the set of all sequences $x\in \A^{\Z^{d}}$ such that every pattern occurring in $x$ is in $\mathcal{L}_{\zeta}$. We denote $(X_{\zeta},S,\Z^{d})$ the \emph{substitutive subshift}. 

We will restrict our study to primitive substitutions. A substitution is called \emph{primitive} if there exists a positive integer $n>0$, such that for every $a,b\in \A$, $b$ occurs in $\zeta^{n}(a)$. If $\zeta$ is a primitive constant-shape substitution, the substitutive system $(X_{\zeta},S,\Z^{d})$ is minimal (the proof of Proposition 5.2 in \cite{queffelec2010substitution} for the one-dimensional case can be easily adapted to our case), and the existence of $\zeta$-\emph{periodic points} is well known, i.e., there exists at least one point $x_{0}\in X_{\zeta}$ such that $\zeta^{p}(x_{0})=x_{0}$ for some $p>0$. Under the primitivity assumption, the subshift is preserved replacing the substitution by a power of it, i.e., $X_{\zeta^{n}}$ is equal to $X_{\zeta}$ for all $n>0$. Then, we may assume that the substitution possesses at least one fixed point, i.e., there exists a point $x\in X_{\zeta}$ such that $x=\zeta(x)$. It is well known that this subshift is strictly ergodic (in \cite{lee2003consequences} can be found a proof of the unique ergodicity for substitutive tiling systems seen as substitutive Delone sets for $\R^{d}$-actions that can be adapted for our context thanks to the assumption that the sequence of supports $(F_{n}^{\zeta})_{n>0}$ form a F\o lner sequence). The unique ergodic measure is characterized in terms of the expansion matrix of $\zeta$ and we denote this measure as $\mu_{\zeta}$. For a cylinder set $[\texttt{p}]_{n}$, where $\texttt{p}$ is a pattern in $\mathcal{L}_{\zeta}$, the quantity $\mu_{\zeta}([\texttt{p}])$ represents the frequency of the pattern $\texttt{p}$ in any sequence in $X_{\zeta}$. The frequencies of patterns exist by unique ergodicity.

Since $L_{\zeta}$ is an expansion matrix, then $L_{\zeta}^{-1}$ defines a contraction map in $\R^{d}$. For any ${\bm g}\in F_{1}^{\zeta}$ define the map $f_{{\bm g}}(\cdot)=L_{\zeta}^{-1}(\cdot+{\bm g})$. As mentioned in \cref{SubsectionFractalGeometry}, by the IFS theorem, there exists a nonempty compact subset $T_{\zeta}$ (or denoted $T(L,F_{1})$ when there is no substitution defined) in $\R^{d}$ such that
$$L_{\zeta}(T_{\zeta})=\bigcup\limits_{{\bm g}\in F_{1}^{\zeta}}T_{\zeta}+{\bm g}.$$

As in \cite{vince2000digit} we call this set the \emph{digit tile} of the substitution. Using $T_{0}=\{{\bm 0}\}$ in \eqref{ApproximationDigitTile} we get that
\begin{equation}
	T_{\zeta}=\lim\limits_{n\to \infty} \sum\limits_{i=0}^{n-1}L_{\zeta}^{-i}(F_{1}^{\zeta})=\lim\limits_{n\to \infty} L_{\zeta}^{-n}(F_{n}^{\zeta}).
\end{equation}

\cref{ExamplesApproximationDigitTiles} shows some approximations of some digit tiles.

\begin{figure}[H]
	\begin{tabular}{cccc}
	\includegraphics[scale=0.22]{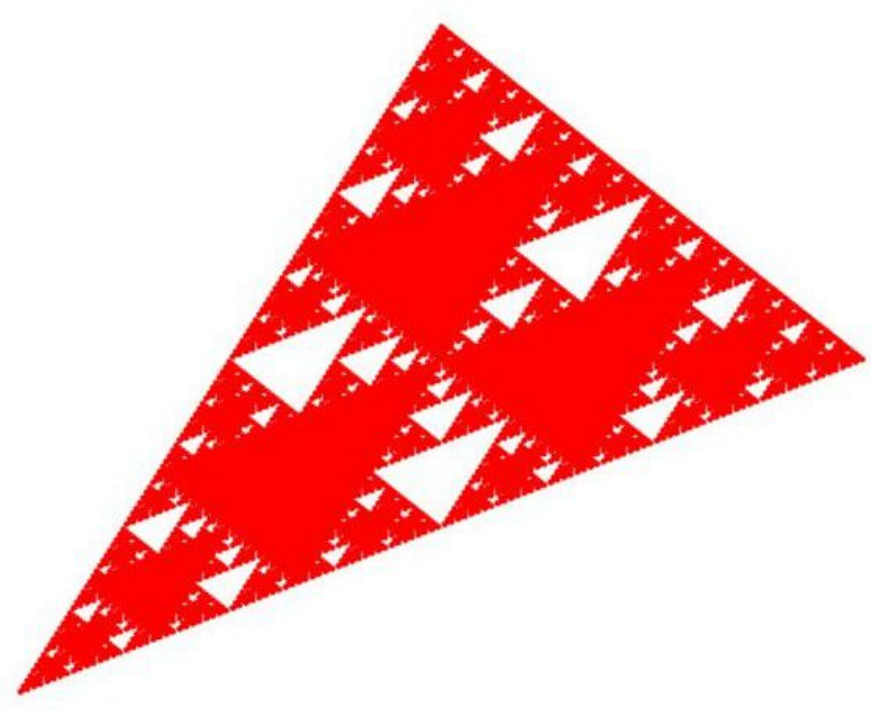} & \includegraphics[scale=0.22]{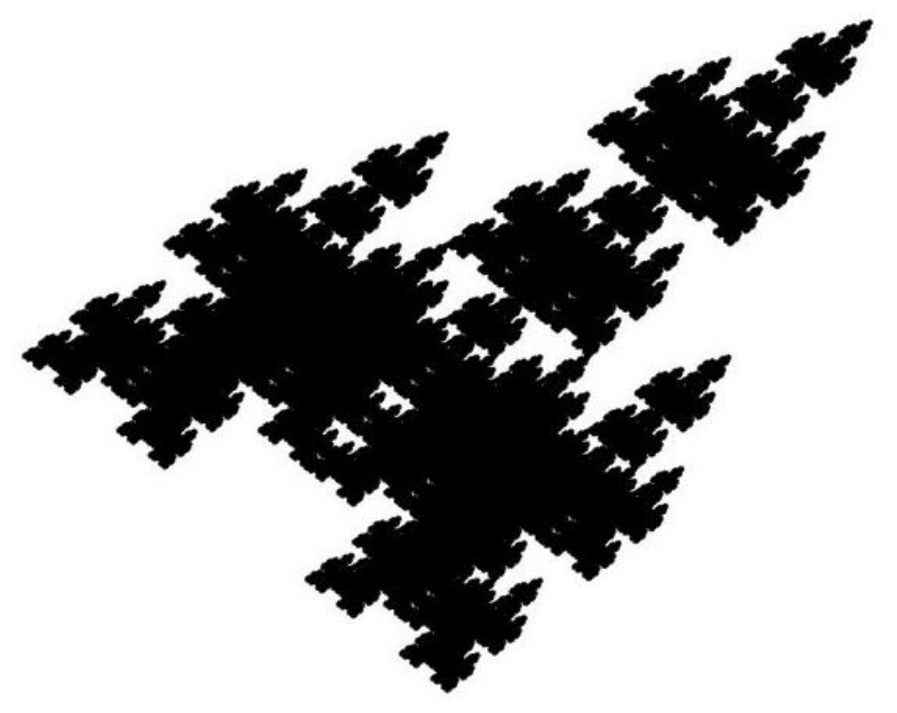} & \includegraphics[scale=0.22]{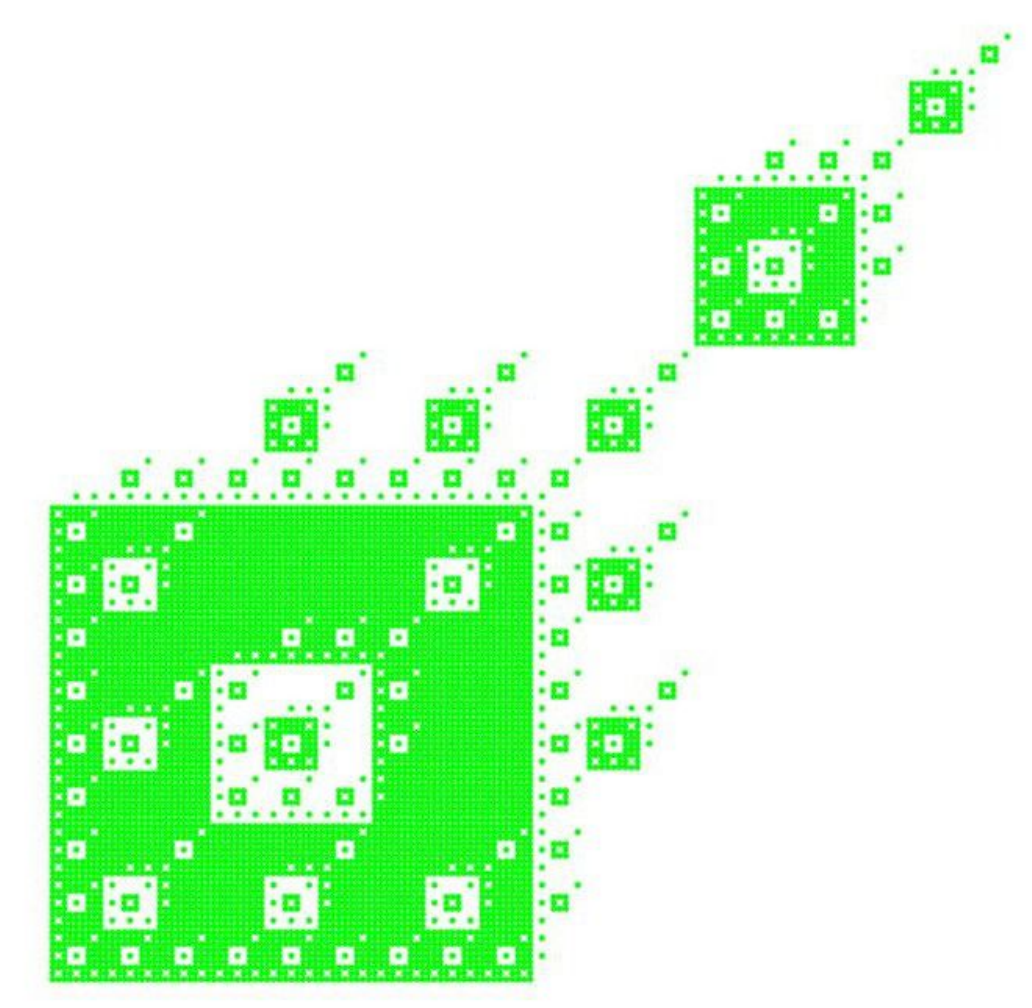} & \includegraphics[scale=0.22]{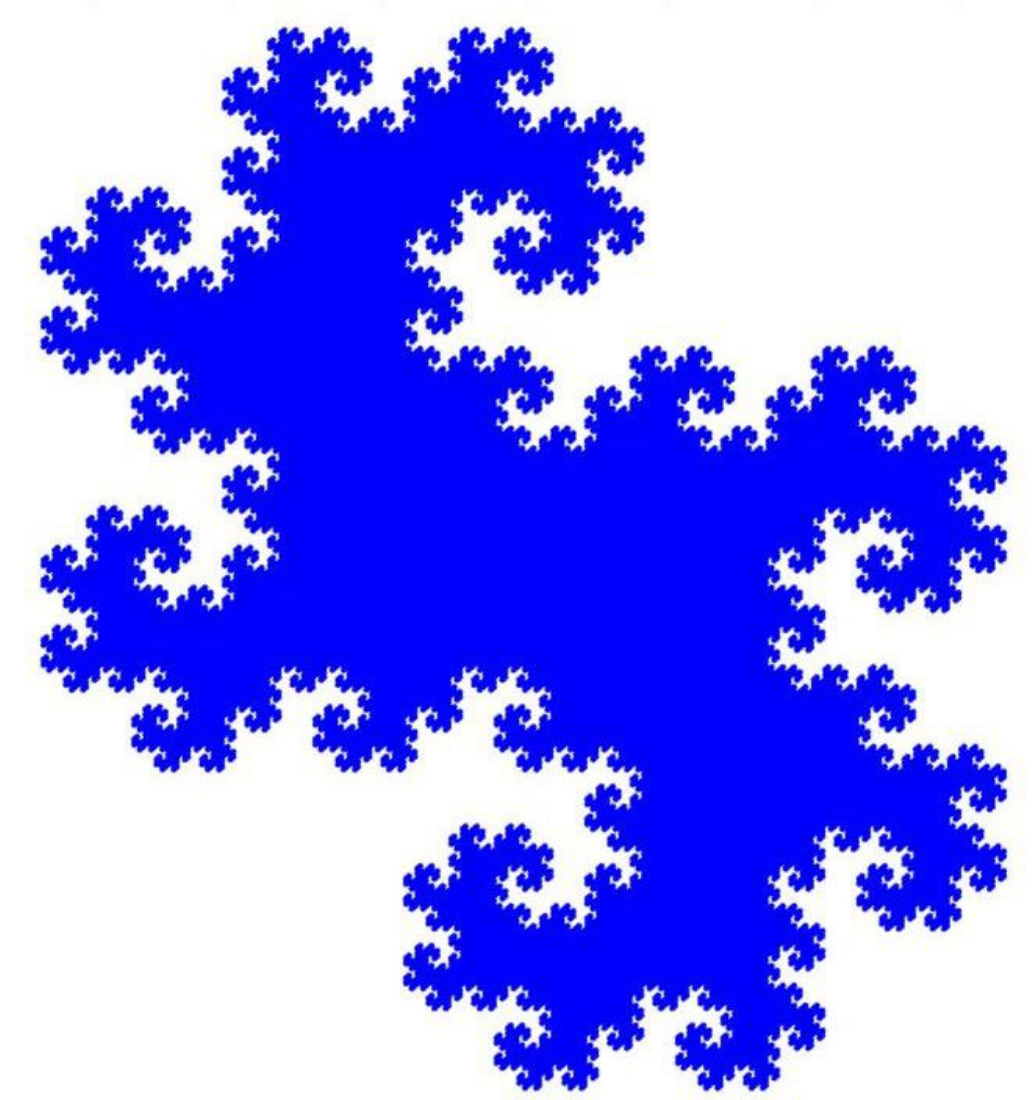}\\
	\multicolumn{1}{c}{(a)} & \multicolumn{1}{c}{(b)} & \multicolumn{1}{c}{(c)} & \multicolumn{1}{c}{(d)}
	\end{tabular}
\caption{Approximation of some digit tiles: (a) Gasket, (b) Rocket, (c) Shooter, (d) Twin Dragon. The names of these tiles come from \cite{vince2000digit}.}
\label{ExamplesApproximationDigitTiles}
\end{figure}

The expansion matrices and fundamental domains of these examples are the following:

\begin{enumerate}[label=(\alph*)]
	\item $\begin{array}{cc}
		L_{(a)} =\left(\begin{array}{cc}2 & 0\\0 & 2
		\end{array}\right),
	& F_{1}^{(a)}  =  \left\{\dbinom{0}{0},\dbinom{1}{0},\dbinom{0}{1},\dbinom{-1}{-1}\right\}.
	\end{array}$

\

\item $\begin{array}{cc}
	L_{(b)} =\left(\begin{array}{cc}3 & 0\\0 & 3
	\end{array}\right),
	& F_{1}^{(b)} = \left\{\dbinom{0}{0},\dbinom{1}{1},\dbinom{2}{2},\dbinom{-1}{0},\dbinom{-2}{0},\dbinom{-1}{1},\dbinom{0}{-1},\dbinom{0}{-2},\dbinom{1}{-1}\right\}.
\end{array}$

\

\item $\begin{array}{cc}
L_{(c)} =\left(\begin{array}{cc}3 & 0\\0 & 3
\end{array}\right),
& F_{1}^{(c)} = \left\{\dbinom{0}{0},\dbinom{1}{0},\dbinom{2}{0},\dbinom{0}{1},\dbinom{0}{2},\dbinom{2}{2},\dbinom{4}{4},\dbinom{2}{1},\dbinom{1}{2}\right\}.
\end{array}$

\

\item $\begin{array}{cc}
L_{(d)} =\left(\begin{array}{cc}1 & -1\\1 & 1
\end{array}\right), & 
F_{1}^{(d)} = \left\{\dbinom{0}{0},\dbinom{1}{0}\right\}.
\end{array}$
\end{enumerate}

\

As in the one-dimensional case, the following proposition shows that for any multidimensional constant-shape substitution there exists a finite subset $K\Subset \Z^{d}$ whose iterations of the substitution fill the whole $\Z^{d}$. This set determine the $\zeta$-periodic points of a primitive constant-shape substitution.

\begin{proposition}\label{FiniteSubsetFillsZd}
	Let $\zeta$ be a multidimensional constant-shape substitution. Then, the set $K_{\zeta}=\bigcup\limits_{m>0}((\id -L_{\zeta}^{m})^{-1}(F_{m}^{\zeta})\cap \Z^{d})$ is finite and satisfies		$$\bigcup\limits_{n\geq 0} L_{\zeta}^{n}(K_{\zeta})+F_{n}^{\zeta}=\Z^{d},$$
	
	\noindent using the notation $F_{0}^{\zeta}=\{{\bm 0}\}$.
\end{proposition} 

\begin{proof}
	Set ${\bm n}\in \Z^{d}$ and consider the sequence $({\bm a}_{m})_{m\geq 0}\subseteq \Z^{d}$ given by ${\bm a}_{0}={\bm n}$ and for any $m\geq 0$, ${\bm a}_{m+1}$ is defined as the unique element in $\Z^{d}$ such that there exists an element ${\bm f}_{m}\in F_{1}^{\zeta}$ with ${\bm a}_{m}=L_{\zeta}({\bm a}_{m+1})+{\bm f}_{m}$. Note that for any $m>0$, ${\bm n}=L_{\zeta}^{m}({\bm a}_{m})+\sum\limits_{i=1}^{m}L^{m-1}({\bm f}_{m})$, and $\Vert {\bm a}_{m+1}\Vert\leq \Vert L_{\zeta}^{-1}\Vert\cdot \Vert {\bm a}_{m}\Vert+\Vert L_{\zeta}^{-1}\Vert\cdot \Vert F_{1}^{\zeta}\Vert$. This implies that for all $m\geq 0$
	$$\Vert {\bm a}_{m}\Vert \leq \Vert L_{\zeta}^{-1}\Vert^{m} \cdot \Vert {\bm a}_{0}\Vert +\dfrac{\Vert L_{\zeta}^{-1}\Vert \cdot \Vert F_{1}^{\zeta}\Vert (1-\Vert L_{\zeta}^{-1}\Vert^{m})}{1-\Vert L_{\zeta}^{-1}\Vert},$$
	
	\noindent hence $({\bm a}_{m})_{m\geq 0}$ is a bounded sequence in $\Z^{d}$. By the Pigeonhole Principle there exist $m\geq0$ and $n>0$ such that ${\bm a}_{m}={\bm a}_{m+n}$, i.e., ${\bm a}_{m}=L_{\zeta}^{n}({\bm a}_{m})+{\bm f}$, for some ${\bm f}\in F_{n}^{\zeta}$. It follows that the set $K_{\zeta}=\bigcup\limits_{m>0}((\id -L_{\zeta}^{m})^{-1}(F_{m}^{\zeta})\cap \Z^{d})$ satisfies the desired property. Now we prove that $K_{\zeta}$ is finite. Note that for any $m>0$
	$$\begin{array}{cl}
		\Vert (\id -L_{\zeta}^{m})^{-1}(F_{m}^{\zeta})\Vert & = \left\Vert \sum\limits_{i=0}^{m-1}(\id -L_{\zeta}^{m})^{-1}L_{\zeta}^{i}(F_{1}^{\zeta})\right\Vert\\
		& \leq \Vert F_{1}^{\zeta}\Vert\left\Vert \sum\limits_{i=0}^{m-1}(\id-L_{\zeta}^{m})^{-1}L_{\zeta}^{i}\right\Vert\\
		& \leq \Vert F_{1}^{\zeta}\Vert \left\Vert (\id -L_{\zeta}^{m})^{-1}(\id-L_{\zeta}^{m})(\id -L_{\zeta})^{-1}\right\Vert\\
		& \leq \Vert F_{1}^{\zeta}\Vert\Vert (\id -L_{\zeta})^{-1}\Vert.
	\end{array}$$
	
	We conclude that $K_{\zeta}$ is a finite set.
\end{proof}

\begin{remark} The following statements can be easily verified.
	\begin{enumerate}[label=(\arabic*)]
		\item In the one-dimensional case, a direct computation shows that for any constant-length substitution $\zeta$, the set $K_{\zeta}$ is equal to $\llbracket-1,0\rrbracket$. Similarly, a direct computation shows that for any $d$-dimensional block substitution, the set $K_{\zeta}$ is equal to $\llbracket -1,0\rrbracket^{d}$.
		
		\item Since $K_{\zeta}$ is a finite set, there exists $j>0$ such that $K_{\zeta}=\bigcup\limits_{m=0}^{j}((\id -L_{\zeta}^{m})^{-1}(F_{m}^{\zeta})\cap \Z^{d})$, therefore, by the fact that the sets $(\id -L_{\zeta}^{m})^{-1}(F_{m}^{\zeta})\cap \Z^{d}$ are nested, up to considering a power of $\zeta$, we can assume that the set $K_{\zeta}$ is of the form $(\id -L_{\zeta})^{-1}(F_{1}^{\zeta})\cap \Z^{d}$.
	\end{enumerate}
	
\end{remark}

The argument used in the proof of \cref{FiniteSubsetFillsZd} is inspired by the Euclidean Division Algorithm. A similar idea can be used to find different sets satisfying specific statements involving the supports $(F_{n}^{\zeta})_{n>0}$ that will be needed for some proofs. From now on, for any ${\bm n}\in \Z^{d}$ we denote ${\bm d_{\bm n}}\in \Z^{d}$ and ${\bm f }_{{\bm n}}\in F_{1}^{\zeta}$ the unique elements such that ${\bm n}=L_{\zeta}({\bm d}_{{\bm n}})+{\bm f}_{{\bm n}}$. The following result will be useful in a series of results throughout this article.

\begin{proposition}\label{SetDforFiniteInvariantOrbit}
	Set $A\Subset \Z^{d}$ and let $F\Subset \Z^{d}$ be such that $F_{1}^{\zeta}\subseteq F$. Define $B=\{{\bm d}_{\bm n}\}_{{\bm n}\in F+A}$. Then, there exists a finite subset $C$ of $\Z^{d}$ satisfying the following conditions:
	\begin{enumerate}
		\item $B\subseteq C$.
		\item $C+F+A\subseteq L_{\zeta}(C)+F_{1}^{\zeta}$.
		\item $\Vert C \Vert \leq \Vert B\Vert+\Vert L_{\zeta}^{-1}\Vert\left(\Vert A\Vert +\Vert F\Vert+\Vert F_{1}^{\zeta}\Vert\right)/\left(1-\Vert L_{\zeta}^{-1}\Vert \right)$.
	\end{enumerate}
\end{proposition} 

\begin{proof}
	We define two sequences of finite sets of $\Z^{d}$, $(B_{n})_{n\geq 0}$, $(C_{n})_{n\geq 0}$, with $B_{0}=B$, $C_{0}=B+F+A$, and for any $n\geq 0$, set $B_{n+1}\Subset \Z^{d}$ such that $B_{n+1}=\{{\bm d}_{{\bm n}}\}_{{\bm n}\in C_{n}}$, and $C_{n+1}\Subset \Z^{d}$ such that $C_{n+1}=B_{n+1}+F+A$. Note that
	$$\begin{array}{cl}
		\Vert B_{n+1}\Vert & \leq \Vert L_{\zeta}^{-1}\Vert \left(\Vert C_{n}\Vert +\Vert F_{1}^{\zeta}\Vert\right)\\
		& \leq \Vert L_{\zeta}^{-1}\Vert\left(\Vert B_{n}\Vert +\Vert A\Vert+\Vert F\Vert +\Vert F_{1}^{\zeta}\Vert\right)\\
		& \leq \Vert L_{\zeta}^{-1}\Vert \Vert B_{n}\Vert +\Vert L_{\zeta}^{-1}\Vert\left(\Vert A\Vert +\Vert F\Vert+\Vert F_{1}^{\zeta}\Vert\right).
	\end{array}$$
	
	Hence, for any $n>0$ we have that
	$$\Vert B_{n}\Vert \leq \Vert B\Vert \Vert L_{\zeta}^{-1}\Vert^{n}+\dfrac{1-\Vert L_{\zeta}^{-1}\Vert^{n}}{1-\Vert L_{\zeta}^{-1}\Vert}\left(\Vert L_{\zeta}^{-1}\Vert\left(\Vert A\Vert +\Vert F\Vert+\Vert F_{1}^{\zeta}\Vert\right)\right).$$
	
	Since $\Vert L_{\zeta}^{-1}\Vert$ is strictly smaller than 1, then $\Vert B_{n}\Vert \leq \Vert B\Vert+\Vert L_{\zeta}^{-1}\Vert\left(\Vert A\Vert +\Vert F\Vert+\Vert F_{1}^{\zeta}\Vert\right)/\left(1-\Vert L_{\zeta}^{-1}\Vert \right)$. This implies there exists $N\in \NN$ such that $\bigcup\limits_{n\leq N}B_{n}=\bigcup\limits_{n\leq N+1}B_{n}$. We conclude the proof taking $C=\bigcup\limits_{n=0}^{N}B_{n}$.
\end{proof}

\begin{remark}\label{RemarkSetDforInvariantOrbit} The following particular cases will be useful in the rest of the article.
	\begin{enumerate}
		\item Condition (2) implies that $C+A+F_{1}^{\zeta}\subseteq L_{\zeta}(C)+F_{1}^{\zeta}$ and a direct induction proves that for all $n\geq 0$, the following inclusion holds: $L_{\zeta}^{n}(C+A)+F_{n}^{\zeta}\subseteq L_{\zeta}^{n+1}(C)+F_{n+1}^{\zeta}$. So, for any set $A\Subset \Z^{d}$ and $F=F_{1}^{\zeta}$ the set $C$ is the minimal set (in terms of cardinality) such that for all $n\geq 0$, $L_{\zeta}^{n}(A)+F_{n}^{\zeta}\subseteq L_{\zeta}^{n+1}(C)+F_{n+1}^{\zeta}$, and the sets $\{L_{\zeta}^{n}(C)+F_{n}^{\zeta}\colon n\geq 0\}$ are nested.
		
		\item Using $F=F_{1}^{\zeta}+F_{1}^{\zeta}$ and $A=\{{\bm 0}\}$, we obtain a set $C\Subset \Z^{d}$ such that $C+F_{1}^{\zeta}+F_{1}^{\zeta}\subseteq L_{\zeta}(C)+F_{1}^{\zeta}$. Since ${\bm 0}\in F_{1}^{\zeta}$, then ${\bm 0}$ is in $C$, which implies that $F_{1}^{\zeta}+F_{1}^{\zeta}\subseteq L_{\zeta}(C)+F_{1}^{\zeta}$. Assume for some $n>0$ that $F_{n}^{\zeta}+F_{n}^{\zeta}\subseteq L_{\zeta}^{n}(C)+F_{n}^{\zeta}$. Then, we obtain that
		$$\begin{array}{cl}
			F_{n+1}^{\zeta}+F_{n+1}^{\zeta} & = F_{n}^{\zeta}+L_{\zeta}^{n}(F_{1}^{\zeta})+F_{n}^{\zeta}+L_{\zeta}^{n}(F_{1}^{\zeta}) \\
			& \subseteq L_{\zeta}^{n}(C)+F_{n}^{\zeta}+L_{\zeta}^{n}(F_{1}^{\zeta})+L_{\zeta}^{n}(F_{1}^{\zeta})\\
			& = L_{\zeta}^{n}(C+F_{1}^{\zeta}+F_{1}^{\zeta})+F_{n}^{\zeta}\\
			& \subseteq L_{\zeta}^{n+1}(C)+L_{\zeta}^{n}(F_{1}^{\zeta})+F_{n}^{\zeta}\\
			& = L_{\zeta}^{n+1}(C)+F_{n+1}^{\zeta},
		\end{array}$$
		
		\noindent so, by induction we have proven that for all $n>0$, $F_{n}^{\zeta}+F_{n}^{\zeta}\subseteq L_{\zeta}^{n}(C)+F_{n}^{\zeta}$. 
	\end{enumerate}
\end{remark}

The next result shows that for any aperiodic symbolic factor $(Y,S,\Z^{d})$ of $(X_{\zeta},S,\Z^{d})$ we can change $\zeta$ for an appropriate substitution $\zeta'$, having the same expansion matrix and fundamental domain, such that $(X_{\zeta'},S,\Z^{d})$ and $(X_{\zeta},S,\Z^{d})$ are conjugate, and there exists a factor map $\pi:(X_{\zeta'},S,\Z^{d})\to (Y,S,\Z^{d})$ induced by a letter-to-letter map.

\begin{lemma}\label{AllFactorAre0BlockMap}
	Let $\zeta$ be an aperiodic primitive constant-shape substitution and $\phi:(X_{\zeta},S,\Z^{d})\to(Y,S,\Z^{d})$ be an aperiodic symbolic factor of $(X_{\zeta},S,\Z^{d})$. Then, there exists a substitution $\zeta'$, having the same support and expansion matrix of $\zeta$, such that $(X_{\zeta'},S,\Z^{d})$ and $(X_{\zeta},S,\Z^{d})$ are conjugate and a factor map $\pi:(X_{\zeta'},S,\Z^{d})\to (Y,S,\Z^{d})$ induced by a letter-to-letter map.
\end{lemma}

\begin{proof}
	Suppose that $\phi:(X_{\zeta},S,\Z^{d})\to (Y,S,\Z^{d})$ is a factor map via a $B({\bm 0},r)$-block map. Set $A=B({\bm 0},r)$, by \cref{SetDforFiniteInvariantOrbit} there exists a set $C\Subset \Z^{d}$ such that $B(0,r)+F_{1}^{\zeta}+C\subseteq L_{\zeta}(C)+F_{1}^{\zeta}$. Set $D=L_{\zeta}(C)+F_{1}^{\zeta}$. We will define a substitution $\zeta^{(D)}$ considering the set $\mathcal{L}_{D}(X_{\zeta})$ as the alphabet, with the same expansion matrix and support of $\zeta$ in the following way: If $\texttt{p}\in \mathcal{L}_{D}(X_{\zeta})$, then for any ${\bm j}\in F_{1}^{\zeta}$ we set $\zeta^{(D)}(\texttt{p})_{{\bm j}}=\zeta(\texttt{p})|_{{\bm j}+D}$. It is straightforward to check that $x\in X_{\zeta}$ is a fixed point of the substitution $\zeta$, if and only if $y\in \mathcal{L}_{D}(X_{\zeta})^{\Z^{d}}$ such that $y_{{\bm n}}=x|_{{\bm n}+D}$ for all ${\bm n}\in \Z^{d}$ is a fixed point of the substitution $\zeta^{(D)}$. With this, we can define the following sliding block codes $\psi_{1}:(X_{\zeta},S,\Z^{d})\to (X_{\zeta^{(D)}},S,\Z^{d})$ given by the $D$-block map $\Psi_{1}(\texttt{p})=\texttt{p}$ and $\psi_{2}:(X_{\zeta^{(D)}},S,\Z^{d})\to (X_{\zeta},S,\Z^{d})$ given by the letter-to-letter map $\Psi_{2}(\texttt{p})=\texttt{p}_{{\bm 0}}$. These maps commute with the shift action and define a conjugacy between $X_{\zeta}$ and $X_{\zeta^{(D)}}$. We then, define a factor map $\phi^{(D)}:(X_{\zeta^{(D)}},S,\Z^{d})\to (Y,S,\Z^{d})$ given by the letter-to-letter map equal to $\psi_{2}\phi$.
\end{proof}

\section{Recognizability property of aperiodic  substitutions and dynamical consequences}\label{SectionRecognizabilityPropertyGeneral} 

In this section, we will study the recognizability property of aperiodic constant-shape substitutions and some dynamical consequences. Every aperiodic substitution satisfies it. To prove this result we first prove that there is a polynomial growth of the repetitivity function for substitutive subshifts (\cref{GrowthRepetititvtyFunction}) that allow us to prove that aperiodic symbolic factors of substitutive subshifts also satisfies a recognizability property (\cref{RecognizabilityFactors} and \cref{RecognizabilitySecondStep}). Then, we will present here several consequences of the recognizability property: there exists a finite number of orbits in $X_{\zeta}$ which are invariant by the substitution map (\cref{FinitelyManyInvariantOrbits}) and we determine the maximal equicontinuous factor of substitutive subshifts and their aperiodic symbolic factors (\cref{MaximalEquiContinuousFactorMultidimensionalSubstitution}). Thanks to these last descriptions, we get that any aperiodic symbolic factor of a substitutive subshift is conjugate to a substitutive subshift (\cref{FactorConjugateSubstitution}).

\subsection{The repetitivity function of a substitutive subshift}\label{RepetitivityFunction}

In order to prove that the recognizability property is satisfied, we study first the repetitivity function of a substitutive subshift. Let $\zeta$ be an aperiodic primitive constant-shape substitution and assume that $x\in X_{\zeta}$ is a fixed point of the substitution. The minimality property implies that the substitutive subshift is \emph{repetitive}, i.e., for every pattern $\texttt{p}\sqsubseteq x$ there is a radius $r>0$ such that for every ${\bm n}\in \Z^{d}$, the ball $B({\bm n},r)$ contains an occurrence of $\texttt{p}$ in $x$. The \emph{repetitivity function} is the map $M_{X_{\zeta}}:\R_{+}\to\R_{+}$ defined for $r>0$ as the smallest radius such that every ball $B({\bm n},M_{X_{\zeta}}(r))$ contains an occurrence in $x$ of every pattern with $\diam(\supp(\texttt{p}))\leq r$. We say that the substitution is \emph{linearly recurrent} or \emph{linearly repetitive} if the repetitivity function has a linear growth, i.e., there exists $C>0$ such that $M_{X_{\zeta}}(r)\leq C\cdot r$. It is well known that aperiodic primitive one-dimensional substitutions are linearly recurrent \cite{durand2000linearly}, but in the multidimensional case this is no longer true, as we can see in \cref{ExampleNonLinearMultidimensionalSubstitution}.

\begin{example}[A non-linearly repetitive  constant-shape substitution]\label{ExampleNonLinearMultidimensionalSubstitution}
	\stepcounter{sigmavariable}
	
	Consider the block substitution $\sigma_{\thesigmavariable}$, given by $L_{\sigma_{\thesigmavariable}}=\left(\begin{array}{cc}
		2 & 0 \\ 0 & 3
	\end{array}\right)$ and $F_{1}^{\sigma_{\thesigmavariable}}=\llbracket0,1\rrbracket\times \llbracket0,2\rrbracket$ defined by
	
	$$\begin{array}{lllllllllllllllll}
		\sigma_{\thesigmavariable}:& \multicolumn{1}{c}{} & \multicolumn{1}{c}{} & \multicolumn{1}{c}{b} & \multicolumn{1}{c}{c} & \multicolumn{1}{c}{} & \multicolumn{1}{c}{} & \multicolumn{1}{c}{} & \multicolumn{1}{c}{} & \multicolumn{1}{c}{a} & \multicolumn{1}{c}{c} & \multicolumn{1}{c}{} & \multicolumn{1}{c}{} & \multicolumn{1}{c}{} & \multicolumn{1}{c}{} & \multicolumn{1}{c}{c} & \multicolumn{1}{c}{b} \\ 
		& \multicolumn{1}{c}{a} & \multicolumn{1}{c}{\mapsto} & \multicolumn{1}{c}{c} & \multicolumn{1}{c}{b} & \multicolumn{1}{c}{\quad} & \multicolumn{1}{c}{\quad} & \multicolumn{1}{c}{b} & \multicolumn{1}{c}{\mapsto} & \multicolumn{1}{c}{c} & \multicolumn{1}{c}{b} & \multicolumn{1}{c}{\quad} & \multicolumn{1}{c}{\quad} & \multicolumn{1}{c}{c} & \multicolumn{1}{c}{\mapsto} & \multicolumn{1}{c}{a} & \multicolumn{1}{c}{c} \\ 
		& \multicolumn{1}{c}{} & \multicolumn{1}{c}{} & \multicolumn{1}{c}{a} & \multicolumn{1}{c}{b} & \multicolumn{1}{c}{} & \multicolumn{1}{c}{} & \multicolumn{1}{c}{} & \multicolumn{1}{c}{} & \multicolumn{1}{c}{b} & \multicolumn{1}{c}{c} & \multicolumn{1}{c}{} & \multicolumn{1}{c}{} & \multicolumn{1}{c}{} & \multicolumn{1}{c}{} & \multicolumn{1}{c}{c} & \multicolumn{1}{c}{b.} \\ 
	\end{array}
	$$
	
	For any $p\geq 1$, we consider the pattern $\texttt{w}_{p}=\sigma_{\thesigmavariable}^{p}(a)|_{\llbracket0,2^{p-1}\rrbracket\times\{0\}}\in \mathcal{L}_{\llbracket0,2^{p-1}\rrbracket\times \{0\}}(X_{\zeta})$. Observe that the pattern $ab$ appears horizontally only in the inferior corner of $\sigma_{\thesigmavariable}(a)$ between the three images of the substitution. So, a direct induction enables to prove that if for some $p\geq 1$ the pattern $\texttt{w}_{p}$ occurs in $\sigma_{\thesigmavariable}^{p}\left(\begin{array}{cc}\alpha & \beta\\ \gamma & \delta\end{array}\right)$, for $\alpha,\beta,\gamma,\delta\in \{a,b,c,\varepsilon\}$ (where $\varepsilon$ denotes the empty pattern), then one of the letters must be $a$. Moreover, $\texttt{w}_{p}$ only appears in the lower left corner of the pattern $\sigma_{\thesigmavariable}^{p}(a)$. These properties imply that there is only one occurrence of $\texttt{w}_{p}$ in $\sigma_{\thesigmavariable}^{p}(\texttt{w}_{p})$, which is in the lower left corner of $\sigma_{\thesigmavariable}^{p}(\texttt{w}_{p})$ as seen in \cref{zetapwp}:
	
	\begin{figure}[H]
		\centering
		\begin{tikzpicture}[scale=0.6]
			\node(b1) at (7, 6.5) [scale=1] {$\sigma_{\thesigmavariable}^{p}(\texttt{w}_{p})$};
			
			\node(b3) at (9,5.5) [scale=1]{$\cdots$};
			
			\node(b3) at (9,4.5) [scale=1]{$\cdots$};
			\node(b3) at (9,3.5) [scale=1]{$\cdots$};
			\node(b3) at (9,2.5) [scale=1]{$\cdots$};
			\node(b3) at (9,1.5) [scale=1]{$\cdots$};
			
			\node(b3) at (9,0.5) [scale=1]{$\cdots$};
			
			\path[thin](0,0) edge (12,0);
			\path[thin](10.,0.) edge (10.,6.);
			\path[thin](12.,0.) edge (12.,6.);
			\path[thin](12.,6.) edge (0.,6.);
			\path[thin](0.,6.) edge (0.,0.);
			\path[thin](2.,6.) edge (2.,0.);
			
			\node(b4) at (1,0.7) [scale=1]{$\texttt{w}_{p}$};
			\path[thin] (0,0.5) edge (2,0.5);
			
			\draw (8.3,3) arc (0:360:3cm);
			
			\draw [decoration={brace,raise=5pt},decorate,line width=1pt]  
			(0,0) -- (0,6);
			
			\node(b5) at (-0.7, 3)[scale=1]{$3^{p}$};
			
			\draw [decoration={brace,raise=5pt},decorate,line width=1pt]  
			(12,0) -- (0,0);
			
			\node(b5) at (6, -0.7)[scale=1]{$4^{p}$};
			
			\node (b6) at (1,6.5)[scale=1]{$\sigma_{\thesigmavariable}^{p}(a)$};
			
			\node (b7) at (3,6.5)[scale=1]{$\sigma_{\thesigmavariable}^{p}(b)$};
			
			\path[thin](4,6) edge (4,0);
		\end{tikzpicture}
		\caption{Decomposition of $\sigma_{\thesigmavariable}^{p}(\texttt{w}_{p})$.}
		\label{zetapwp}
	\end{figure}

	Then, there is a ball of radius $3^{p}/2$ in the support of $\sigma_{\thesigmavariable}^{p}(\texttt{w}_{p})$ with no occurrences of $\texttt{w}_{p}$. Since this is true for any $p$,  this implies that this substitution is not linearly recurrent. 
\end{example}

However, the repetitivity function has at most polynomial growth, with the exponent depending only on the expansion matrix of the substitution.

\begin{lemma}\label{GrowthRepetititvtyFunction}
	Let $\zeta$ be an aperiodic primitive  constant-shape substitution. Then, the repetitivity function $M_{X_{\zeta}}(r)$ is $\mathcal{O}\left(r^{-\frac{\log(\Vert L_{\zeta}\Vert)}{\log\left(\Vert L_{\zeta}^{-1}\Vert\right)}}\right)$.
\end{lemma}

\begin{proof}
	Let $x\in X_{\zeta}$ be a fixed point of $\zeta$. Using $A=\{{\bm 0}\}$ and $F=F_{1}^{\zeta}+F_{1}^{\zeta}$ in \cref{SetDforFiniteInvariantOrbit} we get a finite set $C_{1}\subseteq \Z^{d}$ such that for every $n\geq 0$, $F_{n+1}^{\zeta}+F_{n+1}^{\zeta}\subseteq L_{\zeta}^{n+1}(C_{1})+F_{n+1}^{\zeta}$.
	
	\begin{claim}\label{ClaimFOrIterationsOfRadius}
		For any $r>0$, we have that $L_{\zeta}^{n}(B({\bm 0},r))\cap \Z^{d}\subseteq L_{\zeta}^{n}(B({\bm 0}, r+\Vert \id-L_{\zeta}^{-1}\Vert \cdot \Vert F_{1}^{\zeta}\Vert)\cap \Z^{d})+F_{n}^{\zeta}$.
	\end{claim}

	\begin{proof}[Proof of \cref{ClaimFOrIterationsOfRadius}]
		Set ${\bm n}\in L_{\zeta}^{n}(B({\bm 0},r))\cap \Z^{d}$. Then, there exists ${\bm m}_{1}\in \Z^{d}$ and ${\bm f}\in F_{n}^{\zeta}$ such that ${\bm m}=L_{\zeta}^{n}({\bm m}_{1})+{\bm f}$, which implies that $\Vert {\bm m}_{1}+L_{\zeta}^{-n}({\bm f})\Vert \leq r$. We then get that 
		$$\Vert {\bm m}_{1}\Vert\leq r+\Vert L_{\zeta}^{-n}({\bm f})\Vert \leq r+\Vert \id-L_{\zeta}^{-1}\Vert \cdot \Vert F_{1}^{\zeta}\Vert.$$
	\end{proof}

	Consider $R_{1}>0$ as the maximum radius such that for any $F_{1}^{\zeta}+(B({\bm 0},R_{1})\cap \Z^{d})\subseteq L_{\zeta}(C_{1})+F_{1}^{\zeta}$. Set $R_{2}=R_{1}+\Vert \id-L_{\zeta}^{-1}\Vert \cdot \Vert F_{1}^{\zeta}\Vert$. Note that, by the definition of $C_{1}$ and \cref{ClaimFOrIterationsOfRadius}, for every $n>0$, \begin{equation}
		F_{n}^{\zeta}+(L_{\zeta}^{n}(B({\bm 0},R_{1}))\cap \Z^{d})\subseteq L_{\zeta}^{n}(C_{1}+(B({\bm 0},R_{2})\cap \Z^{d}))+F_{n}^{\zeta}.
	\end{equation}
	
	Consider $T=M_{X_{\zeta}}(\diam(C_{1}+(B({\bm 0},R_{2})\cap \Z^{d})))$ and for any $n>0$, let $T_{n}>0$ be such that every ball of radius $T_{n}$ contains an occurrence of any pattern of the form $\zeta^{n}(\texttt{w})$, with $\texttt{w}\in \mathcal{L}_{C_{1}+(B({\bm 0},R_{2})\cap \Z^{d})}(X_{\zeta})$.
	
	\begin{claim}\label{ClaimFirstForRepetitivity}
		We have that $T_{n}\leq \Vert L_{\zeta}\Vert^{n}(T+1/2)$.
	\end{claim}
	
	\begin{proof}[Proof of \cref{ClaimFirstForRepetitivity}]
		We recall that the lattice $L_{\zeta}^{n}(\Z^{d})$ is $\Vert L_{\zeta}\Vert^{n}/2$-relatively dense, i.e., any ball of radius $\Vert L_{\zeta}\Vert^{n}/2$ contains an element of $L_{\zeta}^{n}(\Z^{d})$. Set ${\bm n}\in \Z^{d}$. Consider ${\bm m}=L_{\zeta}^{n}({\bm p})\in \Z^{d}$ such that $\Vert {\bm n}-{\bm m}\Vert \leq \Vert L_{\zeta}\Vert^{n}/2$. Then, the ball $B({\bm p},T)\cap \Z^{d}$ contains an occurrence for any pattern $\texttt{w}\in \mathcal{L}_{C_{1}+(B({\bm 0},R_{2})\cap\Z^{d})}(X_{\zeta})$. Since $x$ is a fixed point, the set $L_{\zeta}^{n}(B({\bm m},T))\cap \Z^{d}$ contains an occurrence of any pattern of the form $\zeta^{n}({\bm w})$, with $\texttt{w}\in \mathcal{L}_{C_{1}+(B({\bm 0},R_{2})\cap \Z^{d})}(X_{\zeta})$. The fact that $L_{\zeta}^{n}(B({\bm m},T))\subseteq B({\bm m},\Vert L_{\zeta}\Vert ^{n}T)$ and the Cauchy-Schwarz inequality let us conclude that the ball $B({\bm n}, \Vert L_{\zeta}\Vert^{n}(T+1/2))\cap \Z^{d}$ contains an occurrence of any pattern $\zeta^{n}(\texttt{w})$, for $\texttt{w}\in \mathcal{L}_{C_{1}+(B({\bm 0},R_{2})\cap \Z^{d})}(X_{\zeta})$.
	\end{proof}

	Let $r>0$,  $\texttt{p}\in \mathcal{L}(X_{\zeta})$ a pattern such that $\diam(\supp(\texttt{p})\leq r$ and ${\bm n}\in \Z^{d}$ be an occurrence of $\texttt{p}$. Consider $n\geq 0$ such that $R_{1}/\Vert L_{\zeta}^{-1}\Vert^{n-1}\leq \diam(\supp(\texttt{p}))\leq R_{1}/\Vert L_{\zeta}^{-1}\Vert^{n}$. Then, there exists ${\bm n}_{1}\in \Z^{d}$ and ${\bm f}\in F_{n}^{\zeta}$ such that ${\bm n}=L_{\zeta}^{n}({\bm n}_{1})+{\bm f}$. Set $\texttt{w}=x|_{{\bm n}_{1}+C_{1}+(B({\bm 0},R_{2})\cap \Z^{d})}=\texttt{w}$. Noting that $B({\bm 0},R_{1}/\Vert L_{\zeta}^{-1}\Vert)\subseteq L_{\zeta}^{k}(B({\bm 0},R_{1}))$, by \cref{ClaimFOrIterationsOfRadius} we have that $\texttt{p}\sqsubseteq \zeta^{n}(\texttt{w})$. By \cref{ClaimFirstForRepetitivity}, any ball of radius $\Vert L_{\zeta}\Vert^{n}(T+1/2)$ contains an occurrence of $\zeta^{n}(\texttt{w})$ in $x$, so it also contains an occurrence of \texttt{p}. Set $t=-\log(\Vert L_{\zeta}\Vert)/\log(\Vert L_{\zeta}^{-1}\Vert)$. Hence
	$$\dfrac{M_{X}(r)}{r^{t}}\leq \dfrac{\Vert L_{\zeta}^{-1}\Vert^{t(n-1)}\Vert L_{\zeta}\Vert^{n}(T+1/2)}{R_{1}^{t}}=\Vert L_{\zeta}\Vert(T+1/2)R_{1}^{-t}=:C.$$
	
	We finally conclude that $M_{X_{\zeta}}(r)\leq Cr^{t}$, with $t=-\log(\Vert L_{\zeta}\Vert)/\log(\Vert L_{\zeta}^{-1}\Vert)$.
\end{proof}

\begin{remark} The following statements can be easily verified.
	\begin{enumerate}
		\item In the case of a symmetric expansion matrix for the substitution, a bound for $M_{X_{\zeta}}(R)$ is given by its eigenvalues: the repetitivity function $M_{X_{\zeta}}(R)$ is $\mathcal{O}\left(R^{(\log(|\lambda_{1}|))/(\log(|\lambda_{d}|))}\right)$, where $|\lambda_{1}|$, $|\lambda_{d}|$ are the maximum and minimum of the absolute values of the eigenvalues of $L_{\zeta}$, respectively.
		
		\item In the case of a self-similar substitution (where the expansion matrix satisfies $\Vert L_{\zeta}({\bm t})\Vert=\lambda\Vert {\bm t}\Vert$, for some $\lambda >0$), the norm matrix satisfies $\Vert L_{\zeta}\Vert=(\Vert L_{\zeta}^{-1}\Vert)^{-1} = \lambda$, so the repetitivity function has a sublinear growth. Hence self-similar substitutions are linearly recurrent, as it was already proved in \cite{solomyakrecognizability}.
		
		\item The sufficiency of the previous case is not true, there exist  constant-shape substitutions that are not self-similar, but are linearly recurrent.
	\end{enumerate}
\end{remark} 

\subsection{Recognizability of a  constant-shape substitution and their aperiodic symbolic factors}\label{SectionRecognizabilityProperty} The substitution $\zeta$ seen as a map from $X_{\zeta}$ to $\zeta(X_{\zeta})$ is continuous. Moreover, when the substitution is aperiodic and primitive, this map is actually a homeomorphism. This property is satisfied, even in the case where the substitution $\zeta:\A\to\A^{F_{1}^{\zeta}}$ is not \emph{injective on letters}, i.e., when there exists a pair of letters $a,b\in \A$ such that $\zeta(a)=\zeta(b)$. This comes from the notion of \emph{recognizability} of a substitution. 

\begin{definition}
	Let $\zeta$ be a primitive substitution and $x\in X_{\zeta}$ be a fixed point. We say that $\zeta$ is \emph{recognizable on $x$} if there exists some constant $R>0$ such that for all ${\bm i}, {\bm j}\in \Z^{d}$,
	$$x|_{B(L_{\zeta}({\bm i}),R)\cap \Z^{d}}=x|_{B({\bm j},R)\cap \Z^{d}} \implies (\exists {\bm k}\in \Z^{d}) (({\bm j}=L_{\zeta}({\bm k}))\wedge (x_{{\bm i}}=x_{{\bm k}})).$$
\end{definition}

This implies that for every $x\in X_{\zeta}$ there exist a unique $x'\in X_{\zeta}$ and a unique ${\bm j} \in F_{1}^{\zeta}$ such that $x=S^{{\bm j}}\zeta(x')$. With this, the set $\zeta(X_{\zeta})$ is a clopen subset of $X_{\zeta}$ and $\{S^{{\bm j}}\zeta(X_{\zeta})\colon\ {\bm j} \in F_{1}^{\zeta}\}$ is a clopen partition of $X_{\zeta}$ (in \cite[Section 5.6]{queffelec2010substitution} can be found a summary of these statements for the one-dimensional case. The proofs can be found in \cite[Proposition 5.17]{queffelec2010substitution}, \cite[Proposition 5.20]{queffelec2010substitution} and \cite[Corollary 7.2.3]{fogg2002substitutions} that can be easily adapted to our case). These properties are also true for the iterations $\zeta^{n}$, for all $n>0$. The recognizability property was first proved for any aperiodic primitive substitution by B. Moss\'e in \cite{mosse1992puissances} for the one-dimensional case, and in the multidimensional case by B. Solomyak in \cite{solomyakrecognizability} proved a recognizability property for aperiodic self-affine tilings with $\R^{d}$-actions.

In this section we will prove it for aperiodic symbolic factors of substitutive subshifts $(X_{\zeta},S,\Z^{d})$ (\cref{RecognizabilityFactors}). Like in Moss\'e's original proofs, the proof of the recognizability will go in two steps. This property will allows us to determine the maximal equicontinuous factors of aperiodic symbolic factors of substitutive subshifts. 

For the first step of the proof of the recognizability property we use the following propositions. The first one, called repulsion property, is a direct consequence of the growth of the repetitivity function. The proof is an adaptation of the proof of Lemma 2.4 in \cite{solomyakrecognizability}.

\begin{proposition}[Repulsion property]\label{RepulsionProperty} Let $\zeta$ be an aperiodic primitive constant-shape substitution, $x\in X_{\zeta}$ and set $t=-\log(\Vert L_{\zeta}\Vert)/\log(\Vert L_{\zeta}^{-1}\Vert)$. Then, there exists $N>0$ such that, if a pattern $\texttt{p}\sqsubseteq x$ with $B({\bm s},r)\cap \Z^{d}\subseteq \supp(\texttt{p})$, for some ${\bm s}\in \Z^{d}$ and $r>0$ has two occurrences ${\bm j}_{1}, {\bm j}_{2}\in \Z^{d}$ in $x$ such that $r\geq N\Vert {\bm j}_{1}-{\bm j}_{2}\Vert^{t}$, then ${\bm j}_{1}$ is equal to ${\bm j}_{2}$.	
\end{proposition}

\begin{figure}[H]
	\centering
	\begin{tikzpicture}[scale=0.3]
		\node(b1) at (8, 8.5) [scale=1] {$\texttt{p}$};
		
		\path[thin](0,0) edge (16,0);
		\path[thin](16.,0.) edge (16.,8.);
		\path[thin](16.,8.) edge (0.,8.);
		\path[thin](0.,8.) edge (0.,0.);
		
		\node(b8) at (10, 6.5) [scale=1,blue] {$\texttt{p}$};
		
		\path[thin,blue](1,-1) edge (17,-1);
		\path[thin,blue](17.,-1.) edge (17.,7);
		\path[thin,blue](17.,7.) edge (1.,7.);
		\path[thin,blue](1.,7.) edge (1.,-1.);

		\draw[->,thin] (0,0) -- (1,-1);
		
		\node(b15) at (-0.3,-0.3)[scale=1]{${\bm j}_{1}$};
		
		\node(b16) at (0.7,-1.3)[scale=1]{${\bm j}_{2}$};	
		
		\draw[red] (9,4) arc (0:360:3.5cm);
		
		\node(b17) at (5.5,4)[scale=1,red]{$B({\bm s},Nr_{1}^{t})$};
		
		\draw[brown] (1.6,-0.3) arc (0:360:1.9cm);
		
		\node(b18) at (-5,-0.4)[scale=1,brown]{$B({\bm j_{1}}, r_{1})$};		
	\end{tikzpicture}
	\caption{Illustration of a forbidden situation given by the repulsion property (\cref{RepulsionProperty}).}
	\label{illustrationrepulsionproperty}
\end{figure}
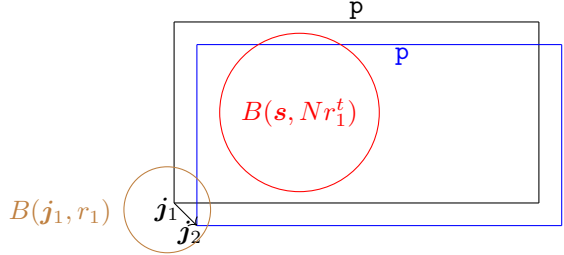

\begin{proof}
	For any ${\bm k}\in \Z^{d}$, we consider the pattern $\texttt{w}_{{\bm k}}=x|_{{\bm k}\cup ({\bm k}+{\bm j}_{2}-{\bm j}_{1})}$. Note that $\diam(\supp(\texttt{w}_{\bm k}))=\Vert {\bm j}_{2}-{\bm j}_{1}\Vert$. We are going to prove that the statement is true for $N>0$ such that $M_{X_{\zeta}}(\Vert {\bm j}_{2}-{\bm j}_{1}\Vert)\leq N\Vert {\bm j}_{2}-{\bm j}_{1}\Vert^{t}\leq r$. By \cref{GrowthRepetititvtyFunction} such $N>0$ exists. Indeed, since $r\geq M_{X_{\zeta}}(\Vert {\bm j}_{2}-{\bm j}_{1}\Vert)$, then the support of the pattern $\texttt{p}$ contains an occurrence in $x$ of any pattern $\texttt{w}_{{\bm k}}$. Since ${\bm j}_{1}$ is an occurrence of $\texttt{p}$ in $x$, we get that for any ${\bm k}\in \Z^{d}$, there exists ${\bm n}_{\bm k}\in \Z^{d}$ such that $x_{{\bm j}_{1}+{\bm n}_{{\bm k}}+{\bm k}}=x_{{\bm k}}$ and $x_{{\bm j}_{1}+{\bm n}_{{\bm k}}+({\bm j}_{2}-{\bm j}_{1}+{\bm k})}=x_{{\bm j}_{2}-{\bm j}_{1}+{\bm k}}$, which implies that $x_{{\bm j}_{2}+{\bm n}_{{\bm k}}+{\bm k}}=x_{{\bm j}_{2}-{\bm j}_{1}+{\bm k}}$. The fact that ${\bm j}_{2}$ is an occurrence of $\texttt{p}$ in $x$ let us conclude that for any ${\bm k}\in \Z^{d}$, $x_{{\bm j}_{2}-{\bm j}_{1}+{\bm k}}$ is equal to $x_{{\bm k}}$, i.e., ${\bm j}_{2}-{\bm j}_{1}$ is a period of $x$. Since $\zeta$ is aperiodic, we conclude that ${\bm j}_{1}={\bm j}_{2}$.
\end{proof}

As mentioned in \cite{durand2016constant}, a key argument for the proof of the recognizability property, is to prove the existence of an integer $p>0$ such that for all $a,b\in \A$, if $\zeta^{n}(a)=\zeta^{n}(b)$ for some $n>0$, then $\zeta^{p}(a)=\zeta^{p}(b)$. This was proved in the one-dimensional case in \cite{ehrenfeucht1978simplifcations}. We extend this result in our context.

\begin{proposition}\label{PropositionForNonInyectiveSubstitutions}
	Let $\zeta$ be a constant-shape substitution over an alphabet $\A$. For any patterns  $\texttt{w}_{1},\texttt{w}_{2}\in \mathcal{L}(X_{\zeta})$, we have that
	$$\zeta^{|\A|-1}(\texttt{w}_{1})\neq \zeta^{|\A|-1}(\texttt{w}_{2}) \implies \forall n\geq 0, \zeta^{n}(\texttt{w}_{1})\neq \zeta^{n}(\texttt{w}_{2}).$$
\end{proposition}

\begin{proof}
	We proceed by induction on $m=|\A|$. The statement is obviously true when $|\A|=1$. Assume now that the statement is true for $|\B|\leq m-1$.
	
	Suppose that $\texttt{w}_{1}\neq \texttt{w}_{2}$ (if $\texttt{w}_{1}=\texttt{w}_{2}$ there is nothing to prove). Then $\zeta$ is not injective on letters. Consider the equivalence relation  $a\sim b$ in $\A$, such that $a\sim b$ if $\zeta(a)=\zeta(b)$ and consider $\pi_{\sim}:\A\to\A/\sim$ the canonical projection. We define a morphism $\sigma:\A/\sim \to \A^{F}$ given by $\sigma([a])=\zeta(a)$. This morphism is well-defined by the very definition of $\sim$. Note that $\zeta=\sigma\circ \pi_{\sim}$. Then
	$$(\exists n\geq 1) \zeta^{n}(\texttt{w}_{1})=\zeta^{n}(\texttt{w}_{2}) \iff (\pi_{\sim}\sigma)^{n}\pi_{\sim}(\texttt{w}_{1})=(\pi_{\sim}\sigma)^{n}\pi_{\sim}(\texttt{w}_{2}),$$
	
	\noindent where $\pi_{\sim}\sigma:\A/\sim \to (\A/\sim)^{F}$ is the substitution with alphabet $\A/\sim$ such that for any ${\bm f}\in F_{1}^{\zeta}$ and $[a]\in \A/\sim$. $\pi_{\sim}\sigma([a])_{{\bm f}}$ is equal to $[\zeta(a)_{{\bm f}}]$. Note that $(\exists n\geq 1) \zeta^{n}(\texttt{w}_{1})=\zeta^{n}(\texttt{w}_{2})$ if and only if $(\exists n\geq 0) (\pi_{\sim}\sigma)^{n}\pi_{\sim}(\texttt{w}_{1})=(\pi_{\sim}\sigma)^{n}\pi_{\sim}(\texttt{w}_{2})$. Since the cardinality of $\A/\sim$ is smaller than $m-1$, by the inductive assumption we have that 
	$$(\exists n\geq 1) \zeta^{n}(\texttt{w}_{1})=\zeta^{n}(\texttt{w}_{2}) \iff (\pi_{\sim}\sigma)^{m-2}\pi_{\sim}(\texttt{w}_{1})=(\pi_{\sim}\sigma)^{m-2}\pi_{\sim}(\texttt{w}_{2}) \iff \pi_{\sim}\zeta^{m-2}(\texttt{w}_{1})=\pi_{\sim}\zeta^{m-2}(\texttt{w}_{2}).$$
	
	Now, $\pi_{\sim}\zeta^{m-2}(\texttt{w}_{1})=\pi_{\sim}\zeta^{m-2}(\texttt{w}_{2})$ implies that $\sigma\pi_{\sim}\zeta^{m-2}(\texttt{w}_{1})=\sigma\pi_{\sim}\zeta^{m-2}(\texttt{w}_{2})$, which is equivalent to $\zeta^{m-1}(\texttt{w}_{1})=\zeta^{m-1}(\texttt{w}_{2})$.
\end{proof}

We then proceed to prove the first step of the recognizability property for aperiodic symbolic factors of constant-shape substitutions. As proved in \cref{AllFactorAre0BlockMap}, we may assume that an aperiodic symbolic factor of a substitutive subshift is induced by a letter-to-letter map.

\begin{proposition}[First step of the recognizability property of aperiodic symbolic factors of substitutive subshifts]\label{RecognizabilityFactors}
	Let $\A,\ \mathcal{B}$ be two finite alphabets, $\zeta$ be an aperiodic primitive constant-shape substitution from the alphabet $\A$ and let $\T:\A\to \mathcal{B}$ be a map such that $(\tau(X_{\zeta}),S,\Z^{d})$ is an aperiodic subshift. Let $x\in X_{\zeta}$ be a fixed point of $\zeta$ and $y=\tau(x)$. Then, there exists $R>0$ such that if $y|_{{\bm i}+B({\bm 0},R)}=y|_{{\bm j}+B({\bm 0},R)}$ and ${\bm i}\in L_{\zeta}(\Z^{d})$, then ${\bm j}\in L_{\zeta}(\Z^{d})$.
\end{proposition} 

\begin{proof} By \cref{SetDforFiniteInvariantOrbit} there exists two finite sets $C,D\subseteq \Z^{d}$ such that for every $n>0$, $F_{n}^{\zeta}+F_{n}^{\zeta}\subseteq L_{\zeta}^{n}(C)+F_{n}^{\zeta}$ and $F_{n}^{\zeta}-F_{n}^{\zeta}\subseteq L_{\zeta}^{n}(D)+F_{n}^{\zeta}$. Set $r=N\Vert L_{\zeta}^{-1}\Vert^{|\A|-1}\Vert L_{\zeta}\Vert^{t(|\A|-1)}\Vert C+C+D\Vert^{t}+\Vert \id-L_{\zeta}^{-1}\Vert\cdot \vert F_{1}^{\zeta}\Vert$, where $t=-\log(\Vert L_{\zeta}\Vert)/\log(\Vert L_{\zeta}^{-1}\Vert)$ and $N$ is one given by \cref{RepulsionProperty}. We prove the statement by contradiction. Assume the contrary, then for every $n>0$ there exist ${\bm i}_{n}\in L_{\zeta}(\Z^{d})$, ${\bm j}_{n}\notin L_{\zeta}(\Z^{d})$ such that $$y|_{{\bm i}_{n}+L_{\zeta}^{n}(D+B({\bm 0},r)\cap \Z^{d})+F_{n}^{\zeta}}=y|_{{\bm j}_{n}+L_{\zeta}^{n}(D+B({\bm 0},r)\cap \Z^{d})+F_{n}^{\zeta}}.$$
	
For any $n>0$, we consider ${\bm a}_{n}\in \Z^{d}$ and ${\bm f}_{n}\in F_{n}^{\zeta}$ such that $L_{\zeta}^{n}({\bm a}_{n})+{\bm f}_{n}={\bm i}_{n}$. Note that $L_{\zeta}^{n}({\bm a}_{n})+L_{\zeta}^{n}(B({\bm 0},r)\cap \Z^{d})+F_{n}^{\zeta}\subseteq {\bm i}_{n}+L_{\zeta}^{n}(D+B({\bm 0},r)\cap \Z^{d})+F_{n}^{\zeta}$. Let $\texttt{u}_{n}\in \mathcal{L}_{B({\bm 0},r)\cap \Z^{d}}(X_{\zeta})$ be such that
$$y|_{{\bm i}_{n}+L_{\zeta}^{n}(B({\bm 0},r)\cap \Z^{d})+F_{n}^{\zeta}}=\tau(\zeta^{n}(\texttt{u}_{n}))=y|_{{\bm j}_{n}+L_{\zeta}^{n}(B({\bm 0},r)\cap \Z^{d})+F_{n}^{\zeta}}.$$

	\begin{figure}[H]
		\centering
		\begin{tikzpicture}[scale=0.3]
			\node(b1) at (5, 6.5) [scale=1] {$\tau(\zeta^{n}(\texttt{u}_{n}))$};
			
			\node(b2) at (5,5.5) [scale=1]{$\cdots$};
			
			\node(b3) at (5,4.5) [scale=1]{$\cdots$};
			\node(b4) at (5,3.5) [scale=1]{$\cdots$};
			\node(b5) at (5,2.5) [scale=1]{$\cdots$};
			\node(b6) at (5,1.5) [scale=1]{$\cdots$};
			
			\node(b7) at (5,0.5) [scale=1]{$\cdots$};
			
			\path[thin](0,0) edge (10,0);
			\path[thin](10.,0.) edge (10.,6.);
			\path[thin](10.,6.) edge (0.,6.);
			\path[thin](0.,6.) edge (0.,0.);
			\path[thin](2.,6.) edge (2.,0.);
			\path[thin](8.,6.) edge (8.,0.);
			\path[thin](0.,1.) edge (10.,1.);
			\path[thin](0.,2.) edge (10.,2.);
			\path[thin](0.,4.) edge (10.,4.);
			\path[thin](0.,5.) edge (10.,5.);
			
			\node(b8) at (25, 1.5) [scale=1,blue] {$\tau(\zeta^{n}(\texttt{u}_{n}))$};
			
			\node(b9) at (25,0.5) [scale=1]{$\cdots$};
			
			\node(b10) at (25,-0.5) [scale=1]{$\cdots$};
			\node(b11) at (25,-1.5) [scale=1]{$\cdots$};
			\node(b12) at (25,-2.5) [scale=1]{$\cdots$};
			\node(b13) at (25,-3.5) [scale=1]{$\cdots$};
			
			\node(b14) at (25,-4.5) [scale=1]{$\cdots$};
			
			\path[thin,blue](20,-5) edge (30,-5);
			\path[thin,blue](30.,-5.) edge (30.,1);
			\path[thin,blue](30.,1.) edge (20.,1.);
			\path[thin,blue](20.,1.) edge (20.,-5.);
			\path[thin,blue](22.,1.) edge (22.,-5.);
			\path[thin,blue](28.,1.) edge (28.,-5.);
			\path[thin,blue](20.,-4.) edge (30.,-4.);
			\path[thin,blue](20.,-3) edge (30.,-3.);
			\path[thin,blue](20.,-1) edge (30.,-1.);
			\path[thin,blue](20.,0) edge (30.,0.);
			
			\draw[->,thin] (0,0) -- (20,-5);
			
			\node(b15) at (6,3)[scale=1]{${\bm i}_{n}$};
			\node(b15) at (5.6,2.8)[scale=0.5]{$\times$};
			
			\node(b16) at (26,-1.8)[scale=1]{${\bm j}_{n}$};
			\node(b15) at (25.6,-2.2)[scale=0.5]{$\times$};
			
			\node(b17) at (11,-2.3)[scale=1,rotate=346]{${\bm j}_{n}-{\bm i}_{n}$};				
		\end{tikzpicture}
		\caption{Illustration of the pattern $\tau(\zeta^{n}(\texttt{u}_{n}))$ around the coordinates ${\bm i}_{n}$ (black) and ${\bm j}_{n}$ (blue).}
		\label{tauzetauandtauzetav}
	\end{figure}
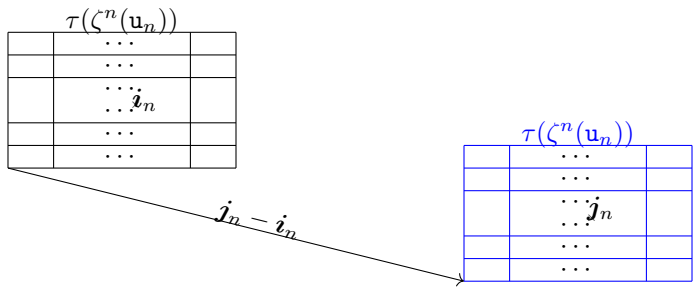	
	
	Note that ${\bm j}_{n}-{\bm f}_{n}$ is not necessarily in $L_{\zeta}^{n}(\Z^{d})$, so we set ${\bm b}_{n}\in \Z^{d}$ and ${\bm g}_{n}\in F_{n}^{\zeta}$ such that $L_{\zeta}^{n}({\bm b}_{n})+{\bm g}_{n}={\bm j}_{n}-{\bm f}_{n}$. Now, for any $n>0$ and $E\subseteq \Z^{d}$ we define the following sets
	$$\begin{array}{cl}
		G_{n,E} & =\{{\bm n}\in \Z^{d}\colon (L_{\zeta}^{n}({\bm n})+F_{n}^{\zeta})\cap ({\bm j}_{n}-{\bm f}_{n})+L_{\zeta}^{n}(E)+F_{n}^{\zeta}\neq \emptyset\}\\
		H_{n,E} & = \{{\bm n}\in \Z^{d}\colon (L_{\zeta}^{n}({\bm n})+F_{n}^{\zeta})\subseteq ({\bm j}_{n}-{\bm f}_{n})+L_{\zeta}^{n}(E)+F_{n}^{\zeta}\}.
	\end{array}$$	
	
	Since $x=\zeta(x)$, there exist a pattern $\texttt{v}_{n}\in \mathcal{L}_{G_{n,B({\bm 0},r)\cap \Z^{d}}-{\bm b}_{n}}(X_{\zeta})$, with $L_{\zeta}^{n}({\bm b}_{n})$ being an occurrence of $\zeta^{n}(\texttt{v}_{n})$ in $x$, such that ${\bm j}_{n}+L_{\zeta}^{n}(B({\bm 0},r)\cap \Z^{d})\subseteq L_{\zeta}^{n}(G_{n,B({\bm 0},r)\cap \Z^{d}})+F_{n}^{\zeta}$. In particular, $\zeta^{n}(\texttt{u}_{n})\sqsubseteq \zeta^{n}(\texttt{v}_{n})$ (hence $\tau(\zeta^{n}(\texttt{u}_{n}))\sqsubseteq \tau(\zeta^{n}(\texttt{v}_{n}))$), as illustrated in \cref{zetavyzetau}:
	
	\begin{figure}[H]
		\centering
		\begin{tikzpicture}[scale=0.6]
			\node(b1) at (5, 6.5) [scale=1] {$\tau(\zeta^{n}(\texttt{v}_{n}))$};
			
			\node(b2) at (10.1, 3) [scale=1] {\textcolor{blue}{$\tau(\zeta^{n}(\texttt{u}_{n}))$}};
			
			\node(b3) at (5,5.5) [scale=1]{$\cdots$};
			
			\node(b4) at (5,4.5) [scale=1]{$\cdots$};
			\node(b5) at (5,3.5) [scale=1]{$\cdots$};
			\node(b6) at (5,2.5) [scale=1]{$\cdots$};
			\node(b7) at (5,1.5) [scale=1]{$\cdots$};
			
			\node(b8) at (5,0.5) [scale=1]{$\cdots$};
			
			\path[thin](0,0) edge (10,0);
			\path[thin](10.,0.) edge (10.,6.);
			\path[thin](10.,6.) edge (0.,6.);
			\path[thin](0.,6.) edge (0.,0.);
			\path[thin](2.,6.) edge (2.,0.);
			\path[thin](8.,6.) edge (8.,0.);
			\path[thin](0.,1.) edge (10.,1.);
			\path[thin](0.,2.) edge (10.,2.);
			\path[thin](0.,4.) edge (10.,4.);
			\path[thin](0.,5.) edge (10.,5.);
			\path[thin,blue](0.6,0.3) edge (8.6,0.3);
			\path[thin,blue](8.6,0.3) edge (8.6,5.3);
			\path[thin,blue](8.6,5.3) edge (0.6,5.3);
			\path[thin,blue](0.6,5.3) edge (0.6,0.3);
			\path[thin,blue](0.6,1.3) edge (8.6,1.3);
			\path[thin,blue](0.6,2.3) edge (8.6,2.3);
			\path[thin,blue](0.6,4.3) edge (8.6,4.3);
			\path[thin,blue](2.6,0.3) edge (2.6,5.3);
			\path[thin,blue](6.6,0.3) edge (6.6,5.3);
			
			\node(b10) at (5.5,3) [scale=1]{${\bm j}_{n}$};
			\node(b11) at (5.25,2.85)[scale=1]{$\times$};
		\end{tikzpicture}
		\caption{Illustration of the patterns $\tau(\zeta^{n}(\texttt{v}_{n}))$ and $\tau(\zeta^{n}(\texttt{u}_{n}))$ around ${\bm j}_{n}$.}
		\label{zetavyzetau}
	\end{figure}

	\begin{claim}\label{ClaimForRecognizability1}
		For any $n>0$, ${\bm b}_{n}\in H_{n,B({\bm 0},r)\cap \Z^{d}}$ and $(G_{n,(B({\bm 0},r)\cap \Z^{d})}-{\bm b}_{n})$ is a bounded set.
	\end{claim}

	\begin{proof}[Proof of \cref{ClaimForRecognizability1}]
		Let ${\bm h}_{n}\in F_{n}^{\zeta}$. Note that ${\bm b}_{n}\in H_{n,B({\bm 0},r)\cap \Z^{d}}$ if and only if there exists ${\bm r}_{n}\in B({\bm 0},r)\cap \Z^{d}$ and ${\bm l}_{n}\in F_{n}^{\zeta}$ such that $L_{\zeta}^{n}({\bm b}_{n})+{\bm h}_{n}=L_{\zeta}^{n}({\bm b}_{n})+{\bm g}_{n}+L_{\zeta}^{n}({\bm r}_{n})+{\bm l}_{n}$, i.e., ${\bm h}_{n}=L_{\zeta}^{n}({\bm r}_{n})+{\bm g}_{n}+{\bm l}_{n}$, which is true since $r\geq \Vert D\Vert$.
		
		Now, set ${\bm m}\in (G_{n,B({\bm 0},r)\cap \Z^{d}}-{\bm b}_{n})$, i.e., there exists ${\bm h}_{n}\in F_{n}^{\zeta}$, ${\bm r}_{n}\in B({\bm 0},r)\cap \Z^{d}$ and ${\bm l}_{n}\in F_{n}^{\zeta}$ such that $L_{\zeta}^{n}({\bm m})+{\bm h}_{n}=L_{\zeta}^{n}({\bm b}_{n})+{\bm g}_{n}+L_{\zeta}^{n}({\bm r}_{n})+{\bm l}_{n}$, i.e., ${\bm m}-{\bm b}_{n}={\bm r}_{n}+L_{\zeta}^{-n}({\bm g}_{n}+{\bm l}_{n}-{\bm h}_{n})$, which implies that $\Vert {\bm m}-{\bm b}_{n}\Vert \leq r+\Vert L_{\zeta}^{-n}({\bm g}_{n}+{\bm l}_{n}-{\bm h}_{n})\Vert$. Note that $\Vert L_{\zeta}^{-n}({\bm g}_{n}+{\bm l}_{n}-{\bm h}_{n})\Vert\leq 3\Vert \id-L_{\zeta}^{-1}\Vert \Vert F_{1}^{\zeta}\Vert$, which let us conclude that $\Vert {\bm m}-{\bm b}_{n}\Vert \leq r+3\Vert \id-L_{\zeta}^{-1}\Vert \Vert F_{1}^{\zeta}\Vert$.
	\end{proof}
	
	By the Pigeonhole principle, there are an infinite set $J\subseteq \mathbb{N}$, a finite set $G\Subset \Z^{d}$ such that $G=(G_{n,(B({\bm 0},r)\cap \Z^{d})}-{\bm b}_{n})$ and $H=(H_{n,(B({\bm 0},r)\cap \Z^{d})}-{\bm b}_{n})$ for all $n\in J$ and patterns $\texttt{u}\in \mathcal{L}_{B({\bm 0},r)\cap \Z^{d}}(X_{\zeta}),\texttt{v}\in \mathcal{L}_{G}(X_{\zeta})$ such that for all $n\in J$, $\texttt{u}_{n}=\texttt{u}$ and $\texttt{v}_{n}=\texttt{v}$. Consider  $\texttt{w}\in \mathcal{L}_{H}(X_{\zeta})$ such that $x|_{L_{\zeta}^{n}({\bm b}_{n})+L_{\zeta}^{n}(H)+F_{n}^{\zeta}}=\zeta^{n}(\texttt{w})$. Note that $L_{\zeta}^{n}({\bm b}_{n})$ is an occurrence of $\zeta^{n}(\texttt{w})$ in $x$ and set $\texttt{a}_{n}=x|_{({\bm j}_{n}-{\bm f}_{n}+(L_{\zeta}^{n}(B({\bm 0},r)+F_{n}^{\zeta}\cap \Z^{d}))\setminus(L_{\zeta}^{n}({\bm b}_{n})+L_{\zeta}^{n}(H)+F_{n}^{\zeta})}$ as illustrated in \cref{zetavyzetauyzetaw}:
	
	\begin{figure}[H]
		\centering
		\begin{tikzpicture}[scale=0.6]
			\node(b1) at (5, 6.5) [scale=1] {$\tau(\zeta^{n}(\texttt{v}))$};
			
			\node(b2) at (9.9, 3) [scale=1] {\textcolor{blue}{$\tau(\zeta^{n}(\texttt{u}))$}};
			
			\node(b9) at (0.9, 3) [scale=1] {\textcolor{red}{$\tau(\zeta^{n}(\texttt{w}))$}};
			
			\node(b3) at (5,5.5) [scale=1]{$\cdots$};
			
			\node(b4) at (5,4.5) [scale=1]{$\cdots$};
			\node(b5) at (5,3.5) [scale=1]{$\cdots$};
			\node(b6) at (5,2.5) [scale=1]{$\cdots$};
			\node(b7) at (5,1.5) [scale=1]{$\cdots$};

			\path[thin](0,0) edge (10,0);
			\path[thin](10.,0.) edge (10.,6.);
			\path[thin](10.,6.) edge (0.,6.);
			\path[thin](0.,6.) edge (0.,0.);
			\path[thin](2.,6.) edge (2.,0.);
			\path[thin](8.,6.) edge (8.,0.);
			\path[thin](0.,1.) edge (10.,1.);
			\path[thin](0.,2.) edge (10.,2.);
			\path[thin](0.,4.) edge (10.,4.);
			\path[thin](0.,5.) edge (10.,5.);
			\path[line width=1.5pt,red](2,1) edge (2,5);
			\path[line width=1.5pt,red](2,1) edge (8,1);
			\path[line width=1.5pt,red](2,5)edge(8,5);
			\path[line width=1.5pt,red](8,1)edge(8,5);
			
			\draw[fill=zzttqq,fill opacity=0.10000000149011612] (0.6,0.3) rectangle (2,5.3);
			\draw[fill=zzttqq,fill opacity=0.10000000149011612] (2,5) rectangle (8.6,5.3);
			\draw[fill=zzttqq,fill opacity=0.10000000149011612] (8,0.3) rectangle (8.6,5);
			\draw[fill=zzttqq,fill opacity=0.10000000149011612] (2,0.3) rectangle (8,1);
			
			\node(b8) at (9.4,0.6) [scale=1]{\textcolor{zzttqq}{$\tau(\texttt{a}_{n})$}};
			
			\path[thin,blue](0.6,0.3) edge (8.6,0.3);
			\path[thin,blue](8.6,0.3) edge (8.6,5.3);
			\path[thin,blue](8.6,5.3) edge (0.6,5.3);
			\path[thin,blue](0.6,5.3) edge (0.6,0.3);
			\path[thin,blue](0.6,1.3) edge (8.6,1.3);
			\path[thin,blue](0.6,2.3) edge (8.6,2.3);
			\path[thin,blue](0.6,4.3) edge (8.6,4.3);
			\path[thin,blue](2.6,0.3) edge (2.6,5.3);
			\path[thin,blue](6.6,0.3) edge (6.6,5.3);
			
			\node(b9) at (3.5,2.6)[scale=1]{$L_{\zeta}^{n}({\bm b}_{n})$};
			\node(b12) at (4.15,3)[scale=1]{$\times$};
			
			\node(b10) at (5.5,3) [scale=1]{${\bm j}_{n}$};
			\node(b11) at (5.25,2.85)[scale=1]{$\times$};
		\end{tikzpicture}
		\caption{Illustration of the patterns $\zeta^{n}(\texttt{w})$ and $\texttt{a}_{n}$ in ${\bm j}_{n}$.}
		\label{zetavyzetauyzetaw}
	\end{figure}
	
	Set $m>n\geq |\A|\in J$. Applying $\zeta^{m-n}$ to $\zeta^{n}(\texttt{u})$, we obtain the patterns $\zeta^{m}(\texttt{a}_{n})$ and $\zeta^{m-n}(\zeta^{n}(\texttt{w}))=\zeta^{m}(\texttt{w})$.
	
	\begin{claim}\label{ClaimAnotherRecognizability}
		For any $n>0$ and any $E\subseteq \Z^{d}$, we have that $G_{n,E}\subseteq H_{n,E}+C+C+D$.
	\end{claim}
	
	\begin{proof}[Proof of \cref{ClaimAnotherRecognizability}]
		First, we are going to prove that for any $n>0$ and $E\Subset \Z^{d}$, we have that $G_{n,E}\subseteq H_{n,E+C+C}+D$. Set ${\bm m}\in G_{n,E}$. Then, there exists ${\bm h}_{n}\in F_{n}^{\zeta}$, ${\bm e}_{n}\in E$, ${\bm l}_{n}\in F_{n}^{\zeta}$ such that $L_{\zeta}^{n}({\bm m})+{\bm h}_{n}=L_{\zeta}^{n}({\bm b}_{n})+{\bm g}_{n}+L_{\zeta}^{n}({\bm e}_{n})+{\bm l}_{n}$. Set ${\bm d}_{n}\in D$ such that ${\bm l}_{n}-{\bm h}_{n}+{\bm g}_{n}=L_{\zeta}^{n}({\bm d}_{n})$. Hence ${\bm m}={\bm b}_{n}+{\bm e}_{n}+{\bm d}_{n}$. We prove that ${\bm m}-{\bm d}_{n}\in H_{n,E+C+C}$.
		
		Set ${\bm o}_{n}\in F_{n}^{\zeta}$. Then
		$$\begin{array}{cl}
			L_{\zeta}^{n}({\bm m}-{\bm d}_{n})+{\bm o}_{n} & =L_{\zeta}^{n}({\bm m})+{\bm h}_{n}-L_{\zeta}^{n}({\bm d}_{n})-{\bm h}_{n}+{\bm o}_{n}\\
			& = L_{\zeta}^{n}({\bm b}_{n})+{\bm g}_{n}+L_{\zeta}^{n}({\bm e}_{n})+{\bm l}_{n}-L_{\zeta}^{n}({\bm d}_{n})-{\bm h}_{n}+{\bm o}_{n}\\
			& = L_{\zeta}^{n}({\bm b}_{n})+L_{\zeta}^{n}({\bm e}_{n})+{\bm o}_{n}\\
		\end{array}$$ 
	
	Let ${\bm q}_{n}\in F_{n}^{\zeta}$ and ${\bm c}_{n}\in C$ be such that ${\bm g}_{n}+{\bm q}_{n}=L_{\zeta}^{n}({\bm c}_{n})$. We get that $L_{\zeta}^{n}({\bm m}-{\bm d}_{n})+{\bm o}_{n}=L_{\zeta}^{n}({\bm b}_{n})+{\bm g}_{n}+L_{\zeta}^{n}({\bm e}_{n}+{\bm c}_{n})+({\bm o}_{n}+{\bm q}_{n})$. Since $F_{n}^{\zeta}+{\bm q}_{n}\subseteq L_{\zeta}^{n}(C)+F_{n}^{\zeta}$, we conclude that ${\bm m}\in H_{n,E+C+C}+D$.
	
	To finish the proof, we note that a straightforward computation shows that for any $n>0$ and $A,B\Subset \Z^{d}$, we have that $H_{n,A+B}\subseteq H_{n,A}+B$. We then, conclude that $G_{n,E}\subseteq H_{n,E}+C+C+D$.
	\end{proof}
	
	 If $\zeta^{m-n}(\texttt{a}_{n})$ and $\texttt{a}_{m}$ are different, there is two occurrences of $\zeta^{m}(\texttt{w})$. By \cref{PropositionForNonInyectiveSubstitutions}, these patterns come from two patterns $\texttt{w}_{1},\texttt{w}_{2}\in \mathcal{L}_{H}(X_{\zeta})$ such that $\zeta^{|\A|-1}(\texttt{w}_{1})=\zeta^{|\A|-1}(\texttt{w}_{2})=\zeta^{|\A|-1}(\texttt{w})$, occurring in $\zeta^{|\A|-1}(\texttt{v})$. The distance between these two occurrences is smaller than $\max\limits_{{\bm t}\in C+C+D}\Vert L_{\zeta}^{|\A|-1}({\bm t})\Vert\leq \Vert L_{\zeta}\Vert^{|\A|-1}\Vert C+C+D\Vert$,  and by \cref{ClaimFOrIterationsOfRadius} we have that $L_{\zeta}^{|\A|-1}(B({\bm 0},r))\subseteq \supp(\zeta^{|\A|-1}(\texttt{w}))$, so $\supp(\zeta^{|\A|-1}(\texttt{w}))$ contains a ball of radius $1/\Vert L_{\zeta}^{-1}\Vert^{|\A|-1}r$. By the repulsion property (\cref{RepulsionProperty}), this is a contradiction, so  $\zeta^{m-n}(\texttt{a}_{n})=\texttt{a}_{m}$ as illustrated in \cref{zetavyzetauyzetawlm-njn}:
	
	\begin{figure}[H]
		\centering
		\begin{tikzpicture}[scale=0.6]
			\node(b1) at (5, 6.5) [scale=1] {$\tau(\zeta^{m}(\texttt{v}))$};
			
			\node(b2) at (9.9, 3) [scale=1] {\textcolor{blue}{$\tau(\zeta^{m}(\texttt{u}))$}};
			
			\node(b9) at (0.9, 3) [scale=1] {\textcolor{red}{$\tau(\zeta^{m}(\texttt{w}))$}};
			
			\node(b3) at (5,5.5) [scale=1]{$\cdots$};
			
			\node(b4) at (5,4.5) [scale=1]{$\cdots$};
			\node(b5) at (5,3.5) [scale=1]{$\cdots$};
			\node(b6) at (5,2.5) [scale=1]{$\cdots$};
			\node(b7) at (5,1.5) [scale=1]{$\cdots$};
			
			\node(b8) at (11.5,0.6) [scale=1]{\textcolor{zzttqq}{$\tau(\zeta^{m-n}(\texttt{a}_{n}))=\tau(\texttt{a}_{m})$}};
			
			\path[thin](0,0) edge (10,0);
			\path[thin](10.,0.) edge (10.,6.);
			\path[thin](10.,6.) edge (0.,6.);
			\path[thin](0.,6.) edge (0.,0.);
			\path[thin](2.,6.) edge (2.,0.);
			\path[thin](8.,6.) edge (8.,0.);
			\path[thin](0.,1.) edge (10.,1.);
			\path[thin](0.,2.) edge (10.,2.);
			\path[thin](0.,4.) edge (10.,4.);
			\path[thin](0.,5.) edge (10.,5.);
			\path[line width=1.5pt,red](2,1) edge (2,5);
			\path[line width=1.5pt,red](2,1) edge (8,1);
			\path[line width=1.5pt,red](2,5)edge(8,5);
			\path[line width=1.5pt,red](8,1)edge(8,5);
			
			\draw[fill=zzttqq,fill opacity=0.10000000149011612] (0.6,0.3) rectangle (2,5.3);
			\draw[fill=zzttqq,fill opacity=0.10000000149011612] (2,5) rectangle (8.6,5.3);
			\draw[fill=zzttqq,fill opacity=0.10000000149011612] (8,0.3) rectangle (8.6,5);
			\draw[fill=zzttqq,fill opacity=0.10000000149011612] (2,0.3) rectangle (8,1);
			
			\path[thin,blue](0.6,0.3) edge (8.6,0.3);
			\path[thin,blue](8.6,0.3) edge (8.6,5.3);
			\path[thin,blue](8.6,5.3) edge (0.6,5.3);
			\path[thin,blue](0.6,5.3) edge (0.6,0.3);
			\path[thin,blue](0.6,1.3) edge (8.6,1.3);
			\path[thin,blue](0.6,2.3) edge (8.6,2.3);
			\path[thin,blue](0.6,4.3) edge (8.6,4.3);
			\path[thin,blue](2.6,0.3) edge (2.6,5.3);
			\path[thin,blue](6.6,0.3) edge (6.6,5.3);
			
			\node(b10) at (6,3) [scale=0.8]{$L_{\zeta}^{m-n}({\bm j}_{n})$};
			
			\node(b12) at (1.85,0.8)[scale=1]{$\times$};
			\node(b11) at (5.25,2.85)[scale=1]{$\times$};
		\end{tikzpicture}
		\caption{Illustration of the patterns $\zeta^{m-n}(\texttt{a}_{n})$ in $L_{\zeta}^{m-n}({\bm j}_{n})$.}
		\label{zetavyzetauyzetawlm-njn}
	\end{figure}
	
	To finish the proof, we note that since $\zeta^{n}(\texttt{u})\sqsubseteq \zeta^{n}(\texttt{v})$, there exists ${\bm p}_{m}\in L_{\zeta}^{n}({\bm b}_{m})+L_{\zeta}^{n}(G)+F_{n}^{\zeta}$ such that $x|_{{\bm p}_{m}+L_{\zeta}^{n}(B({\bm 0},r)\cap \Z^{d})+F_{n}^{\zeta}}=\zeta^{n}(\texttt{u})$, which implies that $x|_{L_{\zeta}^{m-n}({\bm p}_{m})+L_{\zeta}^{m}(B({\bm 0},r)\cap \Z^{d})+F_{m}^{\zeta}}=\zeta^{m}(\texttt{u})$. Using the fact that $\zeta^{m-n}(\texttt{a}_{n})=\texttt{a}_{m}$, we get that $L_{\zeta}^{m-n}({\bm p}_{m})+L_{\zeta}^{m}(B({\bm 0},r)\cap \Z^{d})+F_{m}^{\zeta}={\bm j}_{m}-{\bm f}_{m}+L_{\zeta}^{m}(B({\bm 0},r)\cap \Z^{d})+F_{m}^{\zeta}$, i.e., ${\bm j}_{m}-{\bm f}_{m}=L_{\zeta}^{m-n}({\bm p}_{m})\in L_{\zeta}^{m}(\Z^{d})$. Since ${\bm i}_{m}\in L_{\zeta}(\Z^{d})$, then ${\bm f}_{m}\in L_{\zeta}(\Z^{d})$ we conclude that ${\bm j}_{m}\in L_{\zeta}(\Z^{d})$.
\end{proof}

We now prove the second part of the recognizability property. Note that if the substitution is injective on letters, the second step is a direct consequence of the first one. 

\begin{proposition}\label{RecognizabilitySecondStep}[Recognizability property of aperiodic symbolic factors of substitutive subshifts]
	Let $\zeta$ be an aperiodic primitive substitution from the alphabet $\A$ and let $\mathcal{T}:\A\to\B$ be a map such that $(\tau(X_{\zeta}),S,\Z^{d})$ is an aperiodic subshift. Let $x\in X_{\zeta}$ be a fixed point of $\zeta$ and $y=\tau(x)$. Consider $R>0$ as the recognizability radius from \cref{RecognizabilityFactors} for $\zeta^{|\A|}$ and $M=R+2\Vert F_{1}^{\zeta}\Vert( \Vert L_{\zeta}\Vert^{|\A|}-1)/(\Vert L_{\zeta}\Vert)$ Then, for any ${\bm i},{\bm j}\in \Z^{d}$
	$$y|_{B(L({\bm i}),M)}=y|_{B(L({\bm j}),M)} \implies y_{{\bm i}}=y_{{\bm j}}.$$
\end{proposition}

\begin{proof}
	Let ${\bm k}\in \Z^{d}$ and ${\bm f}=\sum\limits_{i=1}^{|\A|}L_{\zeta}^{i}({\bm f}_{i})\in F_{|\A|}^{\zeta}$ be such that $L_{\zeta}^{|\A|}({\bm k})+{\bm f}=L_{\zeta}({\bm i})$. Hence, we have that $L_{\zeta}^{|\A|-1}({\bm k})+\sum\limits_{i=1}^{|\A|}L_{\zeta}^{i-1}({\bm f}_{i})={\bm i}$. By the definition of $R>0$, we have the existence of ${\bm m}\in \Z^{d}$ such that $L_{\zeta}^{|\A|}({\bm m})+{\bm f}=L_{\zeta}({\bm j})$, which implies that ${\bm j}= L_{\zeta}^{|\A|-1}({\bm m})+\sum\limits_{i=1}^{|\A|}L_{\zeta}^{i-1}({\bm f}_{i})$. Note that, by the definition of $M>0$, we get that $y|_{L_{\zeta}^{|\A|}({\bm k})+F_{|\A|}^{\zeta}}=y|_{L_{\zeta}^{|\A|}({\bm m})+F_{|\A|}^{\zeta}}$. Hence $\tau(\zeta^{|\A|}(x_{{\bm k}}))=\tau(\zeta^{|\A|})(x_{{\bm m}})$. By \cref{PropositionForNonInyectiveSubstitutions} for the substitution $\tau\zeta$ on $\B$, we get that $\tau(\zeta^{|\A|-1}(x_{{\bm k}}))=\tau(\zeta^{|\A|-1})(x_{{\bm m}})$, which then implies that $y_{{\bm i}}=y_{{\bm j}}$.
\end{proof}

\subsection{Invariant orbits of substitutive subshifts} As mentioned in the previous subsection, we assume that aperiodic primitive constant-shape substitutions admit at least one fixed point for the map $\zeta:X_{\zeta}\to X_{\zeta}$. The orbits of these fixed points lead to the notion of $\zeta$-\emph{invariant orbits}. An orbit $\mathcal{O}(x,\Z^{d})$ is called $\zeta$-\emph{invariant} if there exists ${\bm j}\in \Z^{d}$ such that $\zeta(x)=S^{{\bm j}}x$, i.e., the orbit is invariant under the action of $\zeta$ in $X_{\zeta}$. Since for every ${\bm n}\in \Z^{d}$ we have $\zeta\circ S^{{\bm n}}=S^{L_{\zeta}{\bm n}}\circ \zeta$, the definition is independent of the choice of the point in the $\Z^{d}$-orbit of $x$. The orbit of a fixed point of the substitution map is an example of an invariant orbit. In the following, we will prove that for aperiodic primitive constant-shape substitutions there exist finitely many $\zeta$-invariant orbits. This property will be used to prove other properties about some constant-shape substitutions such as coalescence (\cref{Coalescence}) and that the automorphism group is virtually generated by the shift action (\cref{AutomoprhismVirtuallyZd}). 

\begin{proposition}\label{FinitelyManyInvariantOrbits}
	Let $\zeta$ be an aperiodic primitive constant-shape. substitution. Then, there exist finitely many $\zeta$-invariant orbits in the substitutive subshift $X_{\zeta}$. The bound is explicit and depends only on $d$, $|\A|$, $\Vert L_{\zeta}^{-1}\Vert$, $\Vert F_{1}^{\zeta}\Vert$ and $\det(L_{\zeta}-\id)$.
\end{proposition} 

\begin{proof}
	Let $x\in X_{\zeta}$ be such that $\zeta(x)=S^{{\bm j}_{x}}x$, for some ${\bm j}_{x}\in \Z^{d}$. For any ${\bm m}\in \Z^{d}$, we have that
	$$\zeta(S^{{\bm m}}x)=S^{L_{\zeta}{\bm m}}\zeta(x)=S^{L_{\zeta}{\bm m}+{\bm j}_{x}}x= S^{(L_{\zeta}-\id){\bm m}+{\bm j}_{x}}S^{{\bm m}}x,$$
	
	\noindent and thus ${\bm j}_{x}-{\bm j}_{S^{{\bm m}}x}\in (L_{\zeta}-\id)\Z^{d}$. Let $H\Subset \Z^{d}$ be a fundamental domain of $(L_{\zeta}-\id)(\Z^{d})$ in $\Z^{d}$ with ${\bm 0}\in H$. We may assume that $x\in X_{\zeta}$ is in a $\zeta$-invariant orbit with ${\bm j}_{x}\in H$. Let $K_{\zeta}\Subset \Z^{d}$ be from \cref{FiniteSubsetFillsZd}. Using $-H$ as the set $A$ and $F=F_{1}^{\zeta}$ in \cref{SetDforFiniteInvariantOrbit} we obtain a set $C\Subset \Z^{d}$ such that $L_{\zeta}^{n}(-H+C)+F_{n}^{\zeta}\subseteq L_{\zeta}^{n+1}(C)+F_{n+1}^{\zeta}$ for all $n>0$. Define $D\Subset \Z^{d}$ as $D=C+K_{\zeta}-H$. Suppose that there are more than $|\A|^{|D|}\cdot |H|$ $\zeta$-invariant orbits. By the Pigeonhole Principle, there exist ${\bm j}\in H$ and two different points $x\neq y \in X_{\zeta}$ such that $x|_{D}=y|_{D}$ and $\zeta(x)=S^{{\bm j}}x$, $\zeta(y)=S^{{\bm j}}y$. Note that
	$$\zeta(x|_{D})=\zeta(x)|_{L_{\zeta}(D)+F_{1}^{\zeta}}=x|_{{\bm j}+L_{\zeta}(D)+F_{1}^{\zeta}}$$
	
	Hence, the patterns $x|_{{\bm j}+L_{\zeta}(D)+F_{1}^{\zeta}}$, $y|_{{\bm j}+L_{\zeta}(D)+F_{1}^{\zeta}}$ are the same. Inductively, we obtain that for every $n\geq 0$
	$$x|_{\left(\sum\limits_{k=0}^{n}L_{\zeta}^{k}{\bm j}\right)+L_{\zeta}^{n+1}(D)+F_{n+1}^{\zeta}}=y|_{\left(\sum\limits_{k=0}^{n}L_{\zeta}^{k}{\bm j}\right)+L_{\zeta}^{n+1}(D)+F_{n+1}^{\zeta}}.$$
	
	Let $E_{0}$ be equal to $D$ and for all $n>0$, define $E_{n}=\left(\sum\limits_{k=0}^{n-1}L_{\zeta}^{k}{\bm j}\right)+L_{\zeta}^{n}(D)+F_{n}^{\zeta}$. We will prove that $\bigcup\limits_{n\geq 0}E_{n} =\Z^{d}$. This implies that $x=y$, which is a contradiction. To do this, we will prove that for every $n\geq 0$ the set $L_{\zeta}^{n}(K_{\zeta})+F_{n}^{\zeta}$ is included in $E_{n+1}$ and we conclude by \cref{FiniteSubsetFillsZd}. Note that $L_{\zeta}^{n}(K_{\zeta})+F_{n}^{\zeta}\subseteq E_{n+1}$ if and only if $L_{\zeta}^{n}(K_{\zeta}-{\bm j})+\left(\sum\limits_{k=0}^{n-1}L_{\zeta}^{k}(F_{1}^{\zeta}-{\bm j})\right)\subseteq L_{\zeta}^{n+1}(D)+F_{n+1}^{\zeta}$.
	
	\begin{claim}\label{SumsOfLkF1}
		For every $n\geq 0$ the set $\sum\limits_{k=0}^{n-1}L_{\zeta}^{k}(F_{1}^{\zeta}-{\bm j})$ is included in $L_{\zeta}^{n}(C)+F_{n}^{\zeta}$.
	\end{claim}
	
	\begin{proof}[Proof of \cref{SumsOfLkF1}]
		For $n=0$, note that $F_{1}^{\zeta}-{\bm j}$ is included in $L_{\zeta}(C)+F_{1}^{\zeta}$ by \cref{SetDforFiniteInvariantOrbit}. Assume that for some $n\geq 0$ $\sum\limits_{k=0}^{n}L_{\zeta}^{k}(F_{1}^{\zeta}-{\bm j})\subseteq L_{\zeta}^{n+1}(C)+F_{n+1}^{\zeta}$. We have that
		$$\begin{array}{cl}
			\sum\limits_{k=0}^{n+1}L_{\zeta}^{k}(F_{1}^{\zeta}-{\bm j}) & = \left(\sum\limits_{k=0}^{n}L_{\zeta}^{k}(F_{1}^{\zeta}-{\bm j})\right)+L_{\zeta}^{n+1}(F_{1}^{\zeta}-{\bm j})\\
			& \subseteq L_{\zeta}^{n+1}(C)+F_{n+1}^{\zeta}+L_{\zeta}^{n+1}(F_{1}^{\zeta}-{\bm j})\\
			& \subseteq L_{\zeta}^{n+1}(C+F_{1}^{\zeta}-{\bm j})+F_{n+1}^{\zeta}\\
			& \subseteq L_{\zeta}^{n+2}(C)+L_{\zeta}^{n+1}(F_{1}^{\zeta})+F_{n+1}^{\zeta}\quad \text{(by \cref{SetDforFiniteInvariantOrbit})}\\
			& = L_{\zeta}^{n+2}(C)+F_{n+2}^{\zeta}
		\end{array}$$
		
		We conclude that for every $n\geq 0$, $\sum\limits_{k=0}^{n}L_{\zeta}^{k}(F_{1}^{\zeta}-{\bm j})\subseteq L_{\zeta}^{n+1}(C)+F_{n+1}^{\zeta}$	
	\end{proof}
	
	By \cref{SumsOfLkF1}, the set $L_{\zeta}^{n}(K_{\zeta}-{\bm j})+\left(\sum\limits_{k=0}^{n-1}L_{\zeta}^{k}(F_{1}^{\zeta}-{\bm j})\right)$ is included in $L_{\zeta}^{n}(K_{\zeta}-{\bm j})+L_{\zeta}^{n}(C)+F_{n}^{\zeta}$ and $L_{\zeta}^{n}(K_{\zeta}+C-{\bm j})$ is a subset of $L_{\zeta}^{n+1}(D)+L_{\zeta}^{n}(F_{1}^{\zeta})$, by definition of $D$. We have that $L_{\zeta}^{n}(K_{\zeta}-{\bm j})+L_{\zeta}^{n}(C)+F_{n}^{\zeta}\subseteq L_{\zeta}^{n+1}(D)+L_{\zeta}^{n}(F_{1}^{\zeta})+F_{n}^{\zeta}=L_{\zeta}^{n+1}(D)+F_{n+1}^{\zeta}$ and we conclude the proof.
\end{proof}

\begin{remark} 
	Let $\zeta$ be an aperiodic primitive substitution with an expansion matrix $L_{\zeta}$ such that $|\det(L_{\zeta}-\id_{\R^{d}})|=1$. This implies that $(L_{\zeta}-\id_{\R^{d}})(\Z^{d})=\Z^{d}$. Let $x\in X_{\zeta}$ be a point in a $\zeta$-invariant orbit, i.e., there exists ${\bm j}\in \Z^{d}$ such that $\zeta(x)=S^{{\bm j}}x$ and set ${\bm m}\in \Z^{d}$ such that $(L_{\zeta}-\id_{\R^{d}})({\bm m})=-{\bm j}$. Then $\zeta(S^{{\bm m}}x)=S^{L_{\zeta}{\bm m}+{\bm j}}\zeta(x)=S^{{\bm m}+(L_{\zeta}-\id_{\R^{d}})({\bm m})+{\bm j}}x=S^{{\bm m}}x$. Hence $S^{{\bm m}}x$ is a fixed point of $\zeta$. We conclude that the only $\zeta$-invariant orbits in this case are the ones given by the fixed points of the substitution.
\end{remark}

In the following we show an example of a $\zeta$-invariant orbit that it is not the orbit of a fixed point of the substitution map.

\begin{example}[An example of a nonfixed point orbit, which is $\zeta$-invariant]
	Consider the one-dimensional substitution $\sigma_{\thesigmavariable}$ of length $3$ given by
	$$\begin{array}{cl}
		\sigma_{\thesigmavariable}: & 0 \mapsto 032,\\
		& 1 \mapsto 123\\
		& 2 \mapsto 013\\
		& 3 \mapsto 102.
	\end{array}$$

	This is an aperiodic primitive constant-length substitution. The set of words of length 2 (or with support $K_{\sigma_{\thesigmavariable}}=\{-1,0\}$) is $\{01,02,03,10,12,13,20,21,23,30,31,32\}$. The words $20,21,30,31$ generate the $4$ fixed points of the square of the substitution $\sigma_{\thesigmavariable}^{2}$:
	
	\begin{figure}[H]
		$$\begin{array}{c}
		\ldots032123102.032102013\ldots\\
		\ldots032123102.123013102\ldots\\
		\ldots123032013.032102013\ldots\\
		\ldots123032013.123013102\ldots
	\end{array}$$
	\caption{The four fixed points of $\sigma_{\thesigmavariable}^{2}$. The dot in the center represents the origin.}
	\end{figure}
	
	 In \cite{host1989homomorphismes} it was proved that the substitutive subshift $(X_{\sigma_{\thesigmavariable}},S,\Z)$ has an automorphism $\phi$ induced by a sliding block code of length 2, given by:
	$$
	\begin{array}{cllllll}
		\Phi: & 01\mapsto 2, & 02\mapsto 0,& 03\mapsto 2, & 10 \mapsto 3, & 12 \mapsto 3, & 13 \mapsto 1,\\
		 & 20 \mapsto 3, & 21 \mapsto 1, & 23 \mapsto 1, & 30 \mapsto 0, & 31 \mapsto 2, & 32 \mapsto 0,
	\end{array}$$
	and also that $S\phi\sigma_{\thesigmavariable}$ is equal to $\sigma_{\thesigmavariable}\phi$. This implies that \begin{equation}\label{CommutationEquationExampleOrbit}
		S^{4}\phi\sigma_{\thesigmavariable}^{2}=\sigma_{\thesigmavariable}^{2}\phi.
	\end{equation}
Let $x\in X_{\sigma_{\thesigmavariable}}$ be a fixed point of $\sigma_{\thesigmavariable}^{2}$. We have that $S^{4}\phi(x)=\sigma_{\thesigmavariable}^{2}(\phi(x))$, so the orbit of $\phi(x)$ is $\sigma_{\thesigmavariable}^{2}$-invariant. The orbit of $\phi(x)$ does not contain a fixed point of $\sigma_{\thesigmavariable}^{2}$. Indeed, suppose that there exists $j\in \Z$ such that $\sigma_{\thesigmavariable}^{2}(S^{j}(\phi(x))=S^{j}(\phi(x))$. By \eqref{CommutationEquationExampleOrbit}, we have that $S^{9j+4}\phi(x)=S^{j}\phi(x)$. The aperiodicity of $(X_{\sigma_{\thesigmavariable}},S,\Z)$ implies that $8j+4=0$, which is a contradiction since $j\in \Z$. We conclude that in the orbit of $\phi(x)$ there is no fixed point of $\sigma_{\thesigmavariable}^{2}$.
\end{example}

\subsection{Substitutive subshifts as extensions of d-dimensional odometers}\label{SectionSubstitutionExtensionOdometer}

The recognizability property establishes a factor map from the substitutive subshift to a $d$-dimensional odometer as follows: For every $n>0$, let $\pi_{n}: X_{\zeta} \to F_{n}^{\zeta}$ be the map satisfying $x \in S^{\pi_{n}(x)}\zeta^{n}(X_{\zeta})$. This map is well defined $(\text{mod}\ L_{\zeta}^{n}(\Z^{d}))$ thanks to the recognizability property. Using basic group theory arguments $\pi_{n+1}(x)=\pi_{n}(x)\ (\Mod L_{\zeta}^{n}(\Z^{d}))$.

Consider the odometer system $(\overleftarrow{\Z^{d}}_{(L_{\zeta}^{n}(\Z^{d}))},+_{(L_{\zeta}^{n}(\Z^{d}))},\Z^{d})$. The map $\pi:(X_{\zeta},S,\Z^{d})\to (\overleftarrow{\Z^{d}}_{(L_{\zeta}^{n}(\Z^{d}))},+_{(L_{\zeta}^{n}(\Z^{d}))},\Z^{d})$ given by $(\pi_{n}(x))_{n>0}$ is a factor map between $(X_{\zeta},S,\Z^{d})$ and $(\overleftarrow{\Z^{d}}_{(L_{\zeta}^{n}(\Z^{d}))},+_{(L_{\zeta}^{n}(\Z^{d}))},\Z^{d})$. Moreover it satisfies the following property.

\begin{lemma}\label{AperiodicUniformylBounded}Let $(X_{\zeta},S,\Z^{d})$ be the substitutive subshift from an aperiodic primitive constant-shape substitution $\zeta$. Then, for any $\overleftarrow{g}=({\bm g}_{n})_{n>0}$ in $\overleftarrow{\Z^{d}}_{(L_{\zeta}^{n}(\Z^{d}))}$ two different cases occur:
	
	\begin{enumerate}[label=\text{K}\arabic*:]
		\item Assume that $\bigcup\limits_{n>0} \left(-{\bm g}_{n}+F_{n}^{\zeta}\right)=\Z^{d}$. Then, there is at most $|\A|$ elements in $|\pi^{-1}(\{\overleftarrow{g}\})|$.
		
		\item On the other hand, $|\pi^{-1}(\{\overleftarrow{g}\})|$ is no greater than $|\A|^{|\overline{K}_{\zeta}|}$, where $\overline{K}_{\zeta}=K_{\zeta}+C$, with $C+F_{1}^{\zeta}+F_{1}^{\zeta}\subseteq L_{\zeta}(C)+F_{1}^{\zeta}$ is obtained by  \cref{SetDforFiniteInvariantOrbit} using $A=\{\bm{0}\}$ and $F=F_{1}^{\zeta}+F_{1}^{\zeta}$, and $K_{\zeta}$ is given by \cref{FiniteSubsetFillsZd}.
	\end{enumerate}
	
	In particular, the factor map $\pi:(X_{\zeta},S,\Z^{d})\to (\overleftarrow{\Z^{d}}_{(L_{\zeta}^{n}(\Z^{d}))},+_{(L_{\zeta}^{n}(\Z^{d}))},\Z^{d})$ is finite-to-1.
\end{lemma}

\begin{proof}
	We separate the proof in these two cases.
	
	\begin{enumerate}[label=\text{K}\arabic*:]
		\item Assume that $|\pi^{-1}(\{\overleftarrow{ g}\})|>|\mathcal{A}|$. Let $x_{0},\ldots,x_{|\mathcal{A}|}$ be in $\pi^{-1}(\{\overleftarrow{g}\})$. For all $n>0$ and all $j\in \{0,\ldots,|\mathcal{A}|\}$, there exist $y_{j}^{n}\in X_{\zeta}$ such that $x_{j}=S^{{\bm g}_{n}}\zeta^{n}(y_{j}^{n})$. By the Pigeonhole Principle, there exist $j_{1}\neq j_{2}\in \{0,\ldots,|\mathcal{A}|\}$ and an infinite set $E\subseteq \mathbb{N}$ with  $\left.y_{j_{1}}^{n}\right|_{{\bm 0}}=\left.y_{j_{2}}^{n}\right|_{{\bm 0}}$ for $n\in E$. This implies that for any $n\in E$, $x_{j_{1}}|_{-{\bm g}_{n}+F_{n}^{\zeta}}=x_{j_{2}}|_{-{\bm g}_{n}+F_{n}^{\zeta}}$. By definition of $\overleftarrow{g}$, the points $x_{j_{1}}, x_{j_{2}}$ are the same, which is a contradiction.
		
		\item Suppose that $\overleftarrow{g}$ does not satisfy K1. Proceeding as in the previous case, we assure the existence of two indices $j_{1}\neq j_{2}$ such that $y_{j_{1}}^{n}|_{\overline{K}_{\zeta}}=y_{j_{2}}^{n}|_{\overline{K}_{\zeta}}$ for any $n$ in an infinite subset $E\subseteq \NN$. This will imply that $x_{j_{1}}|_{-{\bm g}_{n}+L_{\zeta}^{n}(\overline{K}_{\zeta})+F_{n}}=x_{j_{2}}|_{-{\bm g}_{n}+L_{\zeta}^{n}(\overline{K}_{\zeta})+F_{n}^{\zeta}}$ for any $n\in E$. Now, note that for any $n\in \NN$, ${\bm g}_{n+1}-{\bm g}_{n}\in L^{n}(F_{1}^{\zeta})$, and a direct induction prove $-{\bm g}_{n}+L^{n}(\overline{K}_{\zeta})+F_{n}^{\zeta}$ is included in $-{\bm g}_{n+1}+L^{n+1}(\overline{K}_{\zeta})+F_{n+1}^{\zeta}$. Since $K_{\zeta}\subseteq \overline{K}_{\zeta}$, we conclude by \cref{FiniteSubsetFillsZd} that $x_{j_{1}}=x_{j_{2}}$ which will be a contradiction.
	\end{enumerate}
	
\end{proof}

\begin{remark}\label{RemarksFactorOdometer} Note that the set of points $\overleftarrow{g}\in \overleftarrow{\Z^{d}}_{(L_{\zeta}^{n}(\Z^{d}))}$ satisfying K1 is a $G_{\delta}$-set. Indeed, for any $M>0$ define $U_{M}=\left\{\overleftarrow{g}\in \overleftarrow{\Z^{d}}_{(L_{\zeta}^{n}(\Z^{d}))}\colon \llbracket-M,M\rrbracket^{d}\subseteq-{\bm g}_{n}+F_{n}^{\zeta},\ \text{for some}\ n>0\right\}$. Note that $U_{M}$ is an open subset of $\overleftarrow{\Z^{d}}_{(L_{\zeta}^{n}(\Z^{d}))}$ and the set $\bigcap\limits_{M>0}U_{M}$ are exactly the points satisfying K1. 
		
		Now, for any ${\bm n}$ in $\Z^{d}$ and $M>0$ the set $U_{M}$ is included in $U_{M+|{\bm n}|}$. Hence the set of points on the odometer satisfying K1 is invariant by the $\Z^{d}$-action.
\end{remark}

	In some particular cases, we can compute explicitly the fibers of the factor map. For any ${\bm f}\in F_{1}^{\zeta}$ we define $p_{{\bm f}}$ the restriction of $\zeta$ in ${\bm f}$, i.e., for $a\in \A$, we have that $p_{{\bm f}}(a)=\zeta(a)_{{\bm f}}$. We say the substitution is \emph{bijective} if for all ${\bm f}\in F_{1}^{\zeta}$ the maps $p_{{\bm f}}$ are bijective. We have the following result about bijective substitutions
	
	\begin{proposition}
		Let $\zeta$ be an aperiodic bijective primitive constant-shape substitution with alphabet $\A$. Then, the factor map $\pi:(X_{\zeta},S,\Z^{d})\to (\overleftarrow{\Z^{d}}_{(L_{\zeta}^{n}(\Z^{d}))},+_{(L_{\zeta}^{n}(\Z^{d}))},\Z^{d})$ is almost $|\A|$-to-1.
	\end{proposition}

	\begin{proof}
		The map $p_{{\bm 0}}$ is a permutation of $\A$, so we consider a power of $p_{{\bm 0}}$ such that $p_{{\bm 0}}^{n}$ is equal to the identity. Since $\zeta$ is primitive, we replace it by $\zeta^{n}$ and with this we may assume that $\zeta$ possesses at least $|\A|$ fixed points (one given by each letter of the alphabet). Now, let $\overleftarrow{g}\in \overleftarrow{\Z^{d}}_{(L_{\zeta}^{n}(\Z^{d}))}$ satisfy K1. For any $a\in \A$, consider a fixed point of $\zeta$, denoted as $x_{a}$, with $x_{a}({\bm 0})=a$ and define $x_{a}^{\overleftarrow{g}}=\lim\limits_{n_{m}\to \infty}S^{{\bm g}_{n_{m}}}x_{a}$ for some convergent subsequence. Since the sets $\{S^{{\bm g}_{n}}[\zeta^{n}(a)]\}_{a\in \A}$ are disjoint, for any $a\neq b\in \A$, $x_{a}^{\overleftarrow{g}}$ is different from $x_{b}^{{\overleftarrow{g}}}$. Finally, noticing that $\pi(x_{a}^{\overleftarrow{g}})=\overleftarrow{g}$, we have that $\pi^{-1}(\{\overleftarrow{g}\})\geq |\A|$. By \cref{AperiodicUniformylBounded} we conclude that $\pi^{-1}(\{\overleftarrow{g}\})=|\A|$. Since the set of points satisfying K1 form a $G_{\delta}$ set, we conclude that the factor map $\pi:(X_{\zeta},S,\Z^{d})\to (\overleftarrow{\Z^{d}}_{(L_{\zeta}^{n}(\Z^{d}))},+_{(L_{\zeta}^{n}(\Z^{d}))},\Z^{d})$ is almost $|\mathcal{A}|$-to-1.
	\end{proof}

In general, this $d$-dimensional odometer is not the maximal equicontinuous factor of aperiodic constant-shape substitutions. 

\subsection{The maximal equicontinuous factor of substitutive subshifts}

For completeness, we will describe the maximal equicontinuous factor of substitutive subshifts (see \cite{frank2005multidimensional} for the case with diagonal expansion matrices and \cite{dekking1978spectrum} for the one-dimensional case). The results are well known in the literature but nowhere written.

A subgroup $\mathfrak{L}\leq \Z^{d}$ is called a \emph{lattice} if it is isomorphic to $\Z^{d}$, i.e., it has finite index. For a lattice $\mathfrak{L}$ of $\Z^{d}$, we define the \emph{dual lattice} of $\mathfrak{L}$, as the subgroup
$$\mathfrak{L}^{*}=\{{\bm x}\in \R^{d}\colon \left\langle {\bm x}, {\bm n}\right\rangle \in \Z,\ \forall {\bm n}\in \mathfrak{L} \}.$$

We have that $(\Z^{d})^{*}=\Z^{d}$, and for any $R\in \mathcal{M}(d,\Z)$ the set $R(\Z^{d})$ is a lattice of $\Z^{d}$ with dual lattice equal to $(R^{*})^{-1}(\Z^{d})$, where $R^{*}$ stands for the \emph{algebraic adjoint} of $R$. Let $\mathfrak{L}_{1},\mathfrak{L}_{2}$ be two lattices of $\Z^{d}$. We denote by $\mathfrak{L}_{1} \vee \mathfrak{L}_{2}$ the smallest lattice that contains $\mathfrak{L}_{1}$ and $\mathfrak{L}_{2}$, i.e., if a lattice $\mathfrak{L}$ containing $\mathfrak{L}_{1}$ and $\mathfrak{L}_{2}$, then must contains $\mathfrak{L}_{1} \vee \mathfrak{L}_{2}$. 

Fix $x\in X_{\zeta}$. We define the \emph{set of return times} as
$$\mathcal{R}(X_{\zeta})=\{{\bm j}\in \Z^{d}\colon \exists {\bm k}\in \Z^{d},\ x_{{\bm k}+{\bm j}}=x_{{\bm k}}\}.$$

\noindent By minimality of $X_{\zeta}$, this set is well-defined independently of $x\in X_{\zeta}$ and it is syndetic, i.e., there exists a finite subset $A\Subset \Z^{d}$ such that $\mathcal{R}(X_{\zeta})+A=\Z^{d}$. We define $\mathfrak{L}(\mathcal{R}(X_{\zeta}))$ as the smallest lattice containing $\mathcal{R}(X_{\zeta})$. The \emph{height lattice} $\mathcal{H}(X_{\zeta})$ of a substitutive subshift $\zeta$ is the smallest lattice containing $\mathfrak{L}(\mathcal{R}(X_{\zeta}))$ such that $\mathcal{H}(X_{\zeta})\cap L_{\zeta}(\Z^{d})\leq L_{\zeta}(\mathcal{H}(X_{\zeta}))$. Notice that the last property is equivalent to $\mathcal{H}(X_{\zeta})\cap L_{\zeta}^{n}(\Z^{d})\leq L_{\zeta}^{n}(\mathcal{H}(X_{\zeta}))$, for any $n>0$. The height lattice is \emph{trivial} whenever $\mathcal{H}(X_{\zeta})=\Z^{d}$. 

\begin{example}[An example of a height lattice]
	\stepcounter{sigmavariable}
	\newcounter{onevariable}
	\setcounter{onevariable}{\value{sigmavariable}}
	\stepcounter{sigmavariable}
	\newcounter{twovariable}
	\setcounter{twovariable}{\value{sigmavariable}}
	\stepcounter{sigmavariable}
	
	Consider the product substitution $\sigma_{\theonevariable}$ of the one-dimensional substitutions $\sigma_{\thetwovariable}: 0\mapsto 010$, $1\mapsto 201$, $2\mapsto 102$ (with height 2) and $\sigma_{\thesigmavariable}: a\mapsto abcd$, $b\mapsto bcde$, $c\mapsto cdef$, $d\mapsto defa$, $e\mapsto efab$, $f\mapsto fabc$ (with height 3). The substitution $\sigma_{\theonevariable}$ has an alphabet of cardinality 18 with $L_{\sigma_{\theonevariable}}=\left(\begin{array}{cc}
		3 & 0\\ 0& 4
	\end{array}\right)$ and $F_{1}^{\sigma_{\theonevariable}}=\llbracket 0,2\rrbracket \times \llbracket 0,3\rrbracket$. A direct computation shows that the height lattice $\mathcal{H}((X_{\sigma_{\theonevariable}},S,\Z^{2}))$ is equal to $2\Z\times 3\Z$.
\end{example}

In the following, we will give a description for the height lattice, adapting the proof for the one-dimensional case \cite{dekking1978spectrum,queffelec2010substitution}. For ${\bm k}\in \Z^{d}$, we define $\mathcal{R}_{{\bm k}}(X_{\zeta})=\{{\bm j}\in \Z^{d}\colon x_{{\bm j}+{\bm k}}=x_{\bm k}\}$. Let $\mathfrak{L}(\mathcal{R}_{{\bm k}}(X_{\zeta}))$ be the smallest lattice containing $\mathcal{R}_{{\bm k}}(X_{\zeta})$ and $\mathcal{H}_{{\bm k}}(X_{\zeta})$ be the smallest lattice containing $\mathfrak{L}(\mathcal{R}_{{\bm k}}(X_{\zeta}))$ such that $\mathcal{H}_{{\bm k}}(X_{\zeta})\cap L_{\zeta}(\Z^{d})\leq L_{\zeta}(\mathcal{H}_{{\bm k}}(X_{\zeta}))$. 

\begin{lemma} Let $\zeta$ be an aperiodic primitive  constant-shape substitution. Then, for any ${\bm k}_{1},{\bm k}_{2}\in \Z^{d}$ the sets $\mathcal{H}_{{\bm k}_{1}}(X_{\zeta})$, $\mathcal{H}_{{\bm k}_{2}}(X_{\zeta})$ are the same. In particular, for any ${\bm k}\in \Z^{d}$, $\mathcal{H}_{{\bm k}}(X_{\zeta})$ is equal to $\mathcal{H}(X_{\zeta})$.
\end{lemma} 

\begin{proof}
	Let $x\in X_{\zeta}$ be a fixed point of the substitution, ${\bm k}_{1}, {\bm k}_{2}\in  \Z^{d}$ and $N$ be large enough such that $x_{{\bm k}_{2}}\sqsubseteq \zeta^{N}(x_{{\bm k}_{1}})$. Set ${\bm m}$ such that $x_{{\bm k}_{2}}=\zeta^{N}(x_{{\bm k}_{1}})_{{\bm m}}$. Since $x=\zeta^{N}(x)$ for any ${\bm j}\in \mathcal{R}_{{\bm k}_{1}}(X_{\zeta})$, the set $L_{\zeta}^{N}({\bm k}_{1}+{\bm j})+{\bm m}$ is included in $\mathcal{R}_{{\bm k}_{2}}(X_{\zeta})$. Hence $L_{\zeta}^{N}({\bm j})$ is in $\mathfrak{L}(\mathcal{R}_{{\bm k}_{2}}(X_{\zeta}))$ and therefore in $\mathcal{H}_{{\bm k}_{2}}(X_{\zeta})$. By definition of $\mathcal{H}_{{\bm k}_{2}}(X_{\zeta})$ and invertibility of $L_{\zeta}^{N}$, we conclude that ${\bm j}\in \mathcal{H}_{{\bm k}_{2}}(X_{\zeta})$. Since it is the smallest lattice satisfying the property, $\mathcal{H}_{{\bm k}_{1}}(X_{\zeta})$ is a subgroup of $\mathcal{H}_{{\bm k}_{2}}(X_{\zeta})$, and by reciprocity of the arguments these sets are the same. We conclude the second equality by observing that $\mathcal{H}(X_{\zeta})=\sum\limits_{{\bm k}\in \Z^{d}}\mathcal{H}_{{\bm k}}(X_{\zeta})$. 
\end{proof}

As in the one-dimensional case, to study the maximal equicontinuous factor of a substitutive subshift, we study their \emph{eigenvalues}. A vector ${\bm x}\in \R^{d}$ is said to be an eigenvalue for the topological dynamical system $(X_{\zeta},S,\Z^{d})$ (the measure-preserving system $(X_{\zeta},\mu,S,\Z^{d})$) if there exists a continuous function $f:X_{\zeta}\to\mathbb{C}$ ($f \in L^{2}(X_{\zeta},\mu_{\zeta})$, respectively) such that for every ${\bm n}\in \Z^{d}$, $f\circ S^{{\bm n}}=e^{2\pi i \left\langle {\bm x},{\bm n}\right\rangle} f$ in $X_{\zeta}$ ($f\circ S^{{\bm n}}=e^{2\pi i \left\langle {\bm x},{\bm n}\right\rangle} f$ in $X_{\zeta}$, $\mu_{\zeta}$-a.e. in $X_{\zeta}$, respectively). In \cite{solomyak2007eigenfunctions} it was proved the following result, generalizing the characterization of eigenvalues for the one-dimensional case \cite{dekking1978spectrum}.

\begin{theorem}[\cite{solomyak2007eigenfunctions}]\label{CharacterizationEigenvalues} Let $\zeta$ be an aperiodic primitive constant-shape substitution which has a fixed point in $X_{\zeta}$. Then the following are equivalent for ${\bm x} \in \R^{d}$:
	\begin{enumerate}
		\item The vector ${\bm x}$ is a continuous eigenvalue for the topological dynamical system $(X_{\zeta},S,\R^{d})$.
		\item The vector ${\bm x}$ is a measurable eigenvalue for the measure-preserving system $(X_{\zeta},\mu_{\zeta},S,\R^{d})$.
		\item The vector ${\bm x}$ satisfies the following condition:
		$$\lim\limits_{n\to \infty}e^{2\pi i \left\langle L_{\zeta}^{n}{\bm j},{\bm x}\right\rangle}=1,\ \forall\, {\bm j}\in \mathcal{R}(X_{\zeta}).$$
		\item The vector ${\bm x}$ satisfies the following condition:
		$$\lim\limits_{n\to \infty}e^{2\pi i \left\langle L_{\zeta}^{n}{\bm j},{\bm x}\right\rangle}=1,\ \forall\, {\bm j}\in \mathcal{H}(X_{\zeta}).$$
	\end{enumerate}
\end{theorem}

\begin{remark} The same results are satisfied for $\Z^{d}$-topological actions adapting the arguments given in \cite{solomyak2007eigenfunctions}. Note that Condition (4) in Theorem \ref{CharacterizationEigenvalues} is not proved in \cite{solomyak2007eigenfunctions} but it can be easily checked noticing the set of points satisfying $\left\langle (L_{\zeta}^{*})^{N}{\bm x},{\bm j}\right\rangle\in \Z$ is a lattice.
\end{remark}

This implies that the set of continuous (and measurable) eigenvalues of $(X_{\zeta},S,\Z^{d})$ corresponds to the set $\bigcup\limits_{n\geq 0} (L_{\zeta}^{*})^{-n}(\mathcal{H}^{*}(X_{\zeta}))$, where $\mathcal{H}^{*}(X_{\zeta})$ is the dual lattice of $\mathcal{H}(X_{\zeta})$. In particular the set of eigenvalues $E(X_{\zeta},S,\Z^{d})$ of $(X_{\zeta},S,\Z^{d})$ is a subset of $\mathbb{Q}^{d}$. A direct consequence of \cref{CharacterizationEigenvalues} is a description of the maximal equicontinuous factor of $(X_{\zeta},S,\Z^{d})$.

\begin{proposition}\label{MaximalEquiContinuousFactorMultidimensionalSubstitution}
	Let $\zeta$ be an aperiodic primitive  constant-shape substitution. The maximal equicontinuous factor of the substitutive subshift of $(X_{\zeta},S,\Z^{d})$ is the odometer system
	$$(\overleftarrow{\Z^{d}}_{(L_{\zeta}^{n}(\mathcal{H}(X_{\zeta})))},+_{L_{\zeta}^{n}(\mathcal{H}(X_{\zeta}))},\Z^{d})=\left(\lim\limits_{\leftarrow n} (\Z^{d}/L_{\zeta}^{n}(\mathcal{H}(X_{\zeta})),\alpha_{n}),+_{(L_{\zeta}^{n}(\mathcal{H}(X_{\zeta})))},\Z^{d}\right).$$
\end{proposition}

For completeness, we provide another description for the maximal equicontinuous factor in a general setting. Let $(X,T,\Z^{d})$ be a topological dynamical system, where $X$ is a Cantor set, and $\Gamma \leq \Z^{d}$ a finite-index subgroup. We say that $(X,T,\Z^{d})$ admits a \emph{$\Gamma$-minimal partition} if there exists a closed partition $\bigdotcup\limits_{{\bm g}\in \Z^{d}/\Gamma}X_{{\bm g}}=X$ such that for all ${\bm g}\in \Z^{d}/\Gamma$ the set of return times $\mathcal{RT}(X_{{\bm g}})=\{{\bm n}\in \Z^{d}\colon T^{-{\bm n}}(X_{{\bm g}})\cap X_{{\bm g}}\neq \emptyset\}$ is equal to $\Gamma$ and the topological dynamical system $(X_{{\bm g}},\Gamma)$ is a minimal system. Note that, since its a finite partition, the sets $X_{{\bm g}}$ are clopen. The minimality of the induced actions implies that there is at most one $\Gamma$-minimal partition, up to a permutation of the sets $\{X_{{\bm g}}\}_{{\bm g}\in \Z^{d}/\Gamma}$. 

A $\Gamma$-minimal partition is associated with an equicontinuous factor of $(X,T,\Z^{d})$ in the following way: We enumerate $X_{{\bm g}}$ such that for all ${\bm g}\in \Z^{d}/\Gamma$ and ${\bm n} \in \Z^{d}$ we have that $T^{{\bm n}}X_{{\bm g}}=X_{{\bm g}+{\bm n} \mod \Gamma}$. Then, the map $\pi:(X,T,\Z^{d})\to (\Z^{d}/\Gamma,+_{\Gamma},\Z^{d})$, such that $\pi(x)={\bm g}$ if and only if $x\in X_{{\bm g}}$, is a factor map onto $(\Z^{d}/\Gamma,+_{\Gamma},\Z^{d})$, where $\Z^{d}$ acts by quotient translations onto $\Z^{d}/\Gamma$. The following proposition shows the connection between $\Gamma$-minimal partitions and eigenvalues of a topological dynamical system.

\begin{proposition}\label{CharacterizationMinimalPartition}
	Let $(X,T,\Z^{d})$ be a minimal topological dynamical system and $\Gamma\leq \Z^{d}$ be a finite-index subgroup. The system $(X,T,\Z^{d})$ admits a $\Gamma$-minimal partition, if and only if $\Gamma^{*}\subseteq E(X,T,\Z^{d})$.
\end{proposition}

\begin{proof}		
	Let $L\in \mathcal{M}(d,\Z)$ be such that $\Gamma=L(\Z^{d})$. We recall that $L([0,1)^{d})\cap \Z^{d}$ is a fundamental domain of $L(\Z^{d})$ in $\Z^{d}$.
	
	Let $\{X_{{\bm g}}\}_{{\bm g}\in L([0,1)^{d})\cap \Z^{d}}$ be a $\Gamma$-minimal partition satisfying $X_{{\bm g}}=T^{{\bm g}}(X_{{\bm 0}})$ for all ${\bm g}\in L([0,1)^{d})\cap\Z^{d}$. Let ${\bm x}$ be in $(L^{*})^{-1}(\Z^{d})$. Define $f$ as the map $f=\sum\limits_{{\bm g}\in L([0,1)^{d})\cap \Z^{d}}e^{2\pi i\left\langle {\bm x},{\bm g}\right\rangle}\mathds{1}_{X_{{\bm g}}}$. Since the sets $\{X_{{\bm g}}\}_{{\bm g}\in L([0,1)^{d})\cap \Z^{d}}$ are clopen, the map $f$ is continuous. Let $x\in X_{{\bm 0}}$, ${\bm{m}}\in \Z^{d}$ and ${\bm{m}}_{1}\in \Z^{d},$ ${\bm{m}}_{2}\in L([0,1)^{d})\cap \Z^{d}$ be such that ${\bm{m}}=L(\bm{m}_{1})+\bm{m}_{2}$. Note that $T^{{\bm m}}x$ is in $X_{{\bm m}_{2}}$ and since ${\bm x}$ is in $(L^{*})^{-1}(\Z^{d})$, we have that
		\begin{align}\label{ComputationForEigenfunction}
		e^{2\pi i\left\langle {\bm x},{\bm m}\right\rangle}=e^{2\pi i \left\langle {\bm x}, R({\bm m}_{1})+{\bm m_{2}}\right\rangle}=e^{2\pi i(\left\langle {\bm x}, R({\bm m}_{1}) \right\rangle + \left\langle {\bm x}, {\bm m}_{2}\right\rangle)}=e^{2\pi i\left\langle {\bm x},{\bm m}_{2}\right\rangle}.
	\end{align}
	
	We found a continuous map such that $f(T^{{\bm m}}x)=e^{2\pi i \left \langle {\bm x},{\bm m}\right\rangle}f(x)$ for all $x\in X$ and ${\bm m}\in \Z^{d}$. We conclude that ${\bm x} \in E(X,T,\Z^{d})$.
	
	On the other hand, let $y$ be in $X$. For $j\in \{1,\ldots,d\}$ we denote ${\bm x}_{j}=(L^{*})^{-1}(e_{j})$. Since $\Gamma^{*}\subseteq E(X,T,\Z^{d})$ there exists a map $f_{j}:X\to \mathbb{C}$ such that $f_{j}(T^{{\bm m}}y)=e^{2\pi i \left\langle {\bm x}_{j},{\bm m}\right\rangle}f_{j}(y)$ for all ${\bm m}\in \Z^{d}$. Since the eigenspaces are one-dimensional we choose $f_{j}$ such that $f_{j}(y)=1$. By \eqref{ComputationForEigenfunction}, the values of $e^{2\pi i \left\langle {\bm x}_{j},{\bm m}\right\rangle}$ only depend on ${\bm m}\in L([0,1)^{d})\cap \Z^{d}$. Now, for any $j\in \{1,\ldots,d\}$ and ${\bm m}\in L([0,1)^{d})\cap \Z^{d}$ we denote $X_{{\bm m}}^{j}=f_{j}^{-1}(\{e^{2\pi i \left\langle {\bm x}_{j},{\bm m}\right\rangle }\})$. By minimality, for each $j\in \{1,\ldots,d\}$, the set $X$ is equal to $\bigdotcup\limits_{{\bm m}\in L([0,1)^{d})\cap \Z^{d}}X_{{\bm m}}^{j}$. We define $X_{{\bm m}}=\bigcap\limits_{j\in \{1,\ldots,d\}}X_{{\bm m}}^{j}$. We will prove that $\{X_{{\bm m}}\}_{{\bm m}\in L([0,1)^{d})\cap \Z^{d}}$ is a $\Gamma$-minimal partition. 
	
	Indeed, note that $X_{{\bm m}}$ is $\Gamma$-invariant. Now, assume that there exist ${\bm n}_{1}, {\bm n}_{2} \in L([0,1)^{d})\cap \Z^{d}$ such that ${\bm n}_{1}\neq {\bm n}_{2}$ and $X_{{\bm n}_{1}}\cap X_{{\bm n}_{2}}\neq \emptyset$. Using the $\Z^{d}$-action on $X$, this is equivalent to the existence of ${\bm m}\in R([0,1)^{d})\cap \Z^{d}$ with ${\bm m}\neq {\bm 0}$ and $X_{{\bm m}}\cap X_{{\bm 0}}\neq \emptyset$. This means that for all $j\in \{1,\ldots,d\}$, $e^{2\pi i \left \langle {\bm x}_{j},{\bm m}\right\rangle}$ is equal to 1, i.e., $\left\langle {\bm x}_{j},{\bm m}\right\rangle\in \Z$. Since $\Gamma^{*}=\left\langle \{{\bm x}_{1},\ldots,{\bm x}_{d}\}\right\rangle$, we have that ${\bm m}\in (\Gamma^{*})^{*}=\Gamma$ which is a contradiction, so all of these sets are disjoint.
	
	By minimality, we have that $X=\bigdotcup\limits_{{\bm m}\in R([0,1)^{d})\cap \Z^{d}}\overline{\mathcal{O}(T^{{\bm m}}y,\Gamma)}$. Since $\overline{\mathcal{O}(T^{{\bm m}}y,\Gamma)}$ is included in $X_{{\bm m}}$, we conclude that $\{X_{{\bm m}}\}_{{\bm m} \in R([0,1)^{d})\cap \Z^{d}}$ is a clopen partition of $X$ and $\overline{\mathcal{O}(T^{{\bm m}}y,\Gamma)}=X_{{\bm m}}$, so the action of $\Gamma$ on $X_{{\bm m}}$ is minimal. A direct computation shows that the set of return times of each $X_{{\bm m}}$ is $\Gamma$. Hence $\{X_{{\bm m}}\}_{{\bm m}\in R([0,1)^{d})\cap \Z^{d}}$ is a $\Gamma$-minimal partition.
\end{proof}

\begin{remark}\label{RemarkMaximalEquicontinuousFactor} The following statements can be easily verified.
	\begin{enumerate}
		\item In the case of an aperiodic primitive  constant-shape substitution $\zeta$, the recognizability property implies that for all $n>0$, the sets $\{S^{{\bm j}}\zeta^{n}(X_{\zeta})\}_{{\bm j}\in F_{n}^{\zeta}}$ form a $L_{\zeta}^{n}(\Z^{d})$-minimal partition.
		
		\item By \cref{CharacterizationEigenvalues} and \cref{CharacterizationMinimalPartition} for any aperiodic primitive constant-shape substitution there exists a $\mathcal{H}(X_{\zeta})$-minimal partition. 
	\end{enumerate}
\end{remark}

\subsection{Aperiodic symbolic factors of substitutive subshifts are conjugate to substitutive subshifts}

In the following we will extend the results of C. M$\ddot{\text{u}}$lner and R. Yassawi in \cite{yassawi2020automorphisms}, by proving that aperiodic symbolic factors of a substitutive subshift $(X_{\zeta},S,\Z^{d})$ are conjugate to substitutive subshifts where the substitutions have the same expansion matrix and same support as $\zeta$. First we have the following consequence of \cref{RecognizabilityFactors}.

\begin{remark}\label{RemarkRecognizabilityFactorOdometer}
It is straightforward to check that \cref{RecognizabilityFactors} implies that if $x,x'\in X_{\zeta}$ are such that $\tau(x)=\tau(x')$, then $\pi(x)$ is equal to $\pi(x')$, where $\pi:(X_{\zeta},S,\Z^{d})\to (\overleftarrow{\Z^{d}}_{(L_{\zeta}^{n}(\Z^{d})},+_{(L_{\zeta}^{n}(\Z^{d}))},\Z^{d})$ is the factor map.
\end{remark}

To prove the result we will introduce some notions as in \cite{yassawi2020automorphisms}.

Let $\zeta$ be an aperiodic primitive constant-shape substitution. We consider a labeled directed graph $G_{\zeta}$ with vertex set $E=\mathcal{A}^{2}$ and we put an edge $(a,b)$ to $(c,d)$ with label ${\bm f}\in F_{1}^{\zeta}$ if $\zeta(a)_{{\bm f}}=c$, $\zeta(b)_{{\bm f}}=d$. Note that the diagonal $\Delta_{\A}=\{(a,a)\colon a \in \A\}$ is a stable set, i.e., $E(\Delta_{\A})=\Delta_{\A}$. Let $P=(a_{0},b_{0})(a_{1},b_{1})(a_{2},b_{2})$ be a path, by definition, there is an edge from $(a_{0},b_{0})$ to $(a_{1},b_{1})$ with some label ${\bm f}_{0}$ and an edge $(a_{1},b_{1})$ to $(a_{2},b_{2})$ with some label ${\bm f}_{1}$, i.e, $\zeta(a_{0})_{{\bm f}_{0}}=a_{1}$, $\zeta(b_{0})_{{\bm f}_{0}}=b_{1}$, and $\zeta(a_{1})_{{\bm f}_{1}}=a_{2}$, $\zeta(b_{1})_{{\bm f}_{1}}=b_{2}$. Then, we have that $\zeta^{2}(a_{0})_{L_{\zeta}{\bm f}_{0}+{\bm f}_{1}}=a_{2}$, $\zeta^{2}(b_{0})_{L_{\zeta}{\bm f}_{0}+{\bm f}_{1}}=b_{2}$. This means that the paths indicate the simultaneous positions of the letters in the iterates of the substitution.

The following definitions were introduced in \cite{yassawi2020automorphisms}.

\begin{definition}
	Let $\zeta$ be an aperiodic primitive constant-shape substitution.
	\begin{enumerate}
		\item We say that a pair $(a,b)\in \A^{2}\setminus \Delta_{\A}$ is a \emph{periodic-pair} if there is a cycle in $G_{\zeta}$ which starts and ends in $(a,b)$. We define $n(a,b)=\min\{|P|\colon P\ \text{is a cycle in}\ (a,b)\}$ and we denote
		$$n(\zeta)=\lcm\{n(a,b)\colon (a,b)\ \text{is a periodic pair}\},$$
		
		\noindent we call the substitution \emph{pair-aperiodic} if $n(\zeta)=1$.
		
		\item We call a pair $(a,b)\in \A\setminus \Delta_{\A}$ an \emph{asymptotically disjoint pair} if for any $k>0$, there exists a path $P=(a_{0},b_{0})\ldots(a_{k},b_{k})$ in $G_{\zeta}$ of length $k$ with $(a_{0},b_{0})=(a,b)$ and $(a_{k},b_{k})\notin \Delta_{\A}$.
	\end{enumerate}
\end{definition}

\begin{remark}\label{Remarkpairaperiodic} The following statements can be easily verified:
	\begin{enumerate}
		\item As for the case of $\zeta$-periodic points for the substitution, we can replace the substitution $\zeta$ for an appropriate power, i.e, $\zeta^{n(\zeta)}$, so we may assume that the substitution is pair-aperiodic.
		
		\item If the substitution $\zeta$ is bijective, every $(a,b)\in \A^{2}\setminus \Delta_{\A}$ is an asymptotically disjoint pair.
		
		\item Assume that $\zeta$ is pair-aperiodic. If $(a,b)\in \A^{2}$ is not an asymptotically disjoint pair, let $\overline{k}$ be the minimum length of a path from $(a,b)$ such that any path of length $\overline{k}$ has an end in $\Delta_{\A}$. Since the shortest path path has no cycle, we have that $\overline{k}\leq |\A|^{2}$. Indeed, if $\overline{k}>|A|^{2}$, there exists a cycle as a subpath in $P$, i.e., one of the vertex $(c,d)$ is a periodic-pair. Since $\zeta$ is pair-aperiodic, there exists ${\bm f}\in F_{1}^{\zeta}$ such that $\zeta(c)_{{\bm f}}=c\neq d = \zeta(d)_{{\bm f}}$. So we can create a path of length arbitrarily large from $(a,b)$ that does not reach $\Delta_{A}$, which is a contradiction. This implies that $\zeta^{|\A|^{2}}(a)=\zeta^{|\A|^{2}}(b)$.  
	\end{enumerate}
\end{remark}

\begin{definition}
	Let $\A,\mathcal{B}$ be two finite alphabets, $\zeta$ be an aperiodic primitive  constant-shape substitution with alphabet $\A$ and $\tau:\A\to\B$. We say that $a,b\in \A$ with $a\neq b$ are \emph{indistinguishable} (by $(\zeta,\tau)$) if for all $n\geq 0$ we have that $\tau(\zeta^{n}(a))=\tau(\zeta^{n}(b))$.
\end{definition}

With these definitions we are ready to prove the next result which is a multidimensional analogue of Theorem 22 in \cite{yassawi2020automorphisms}. 

\begin{theorem}\label{FactorConjugateSubstitution}
	Let $(Y,S,\Z^{d})$ be an aperiodic symbolic factor (with alphabet $\B$) of a substitutive subshift $(X_{\zeta},S,\Z^{d})$, with $\zeta$ being an aperiodic primitive constant-shape substitution with alphabet $\A$. Then, there exists an aperiodic primitive constant-shape substitution $\zeta'$ with alphabet $\mathcal{C}$ having the same expansion matrix and support of a power of $\zeta$ and a conjugacy $\tau':(X_{\zeta'},S,\Z^{d})\to(Y,S,\Z^{d})$ via a letter-to-letter map.
\end{theorem}

\begin{proof}
	By \cref{AllFactorAre0BlockMap} and \cref{Remarkpairaperiodic} we can assume that the factor map $\tau:(X_{\zeta},S,\Z^{d})\to (Y,S,\Z^{d})$ is induced via a letter-to-letter map and $\zeta$ is a pair-aperiodic substitution.
	
	We define an equivalence relation $a\sim b$ in $\A$, such that $a\sim b$ if $a, b$ are indistinguishable. We define a substitution $\zeta'$ on $\A/\sim $ as  $\zeta'([a])_{{\bm f}}=[\zeta(a)_{{\bm f}}]$, for ${\bm f}\in F_{1}^{\zeta}$, Note that the map $\T':\A/\!\! \sim\,\to \mathcal{B}$, given by $\T'([a])=\tau(a)$ are well defined. These maps satisfy the following property:
	
	\begin{align}\label{Distinguishablepairs}
		\text{Every pair in }\A/\sim\text{ is distinguishable.}
	\end{align}
	
	It is straightforward to check that primitivity of $\zeta$ implies primitivity of $\zeta'$.
	
	We will prove that every periodic pair in $\A/\sim$ by $\zeta'$ is the quotient of a periodic pair in $\A$ by $\zeta$. Assume now that $([a],[b])$ is a periodic pair, i.e, there exists a cycle $P_{1}=([a_{0}],[b_{0}])\ldots([a_{k}],[b_{k}])$ in $G_{\zeta'}$ with $([a_{0}],[b_{0}])=([a_{k}],[b_{k}])=([a],[b])$. Let  $P_{2}=(c_{0},d_{0})\ldots(c_{k},d_{k})$ be a path in $G_{\zeta}$ with $[c_{i}]=[a_{i}]$ and $[d_{i}]=[b_{i}]$ for $0\leq i\leq k$ having the same label of edges as $P_{1}$. Now, repeating this process we get a path $P_{3}$ in $G_{\zeta}$ of length $(\max\{|[a]|,|[b]|\}+1)k$ from $(c,d)$ repeating the labels of the path $P_{1}$ with a period $k$. By the Pigeonhole Principle, there exist two subpaths $P_{4}=(e_{0},f_{0})\ldots(e_{l_{1}k},f_{l_{1}k})$, $P_{5}=(g_{0},h_{0})\ldots (g_{l_{1}}k,h_{l_{2}k})$ of $P_{3}$, having the same labels of the edges of $P_{1}$ repeating them with period $k$, such that $e_{0}=e_{l_{1}k}$, $h_{0}=h_{l_{2}k}$ and $[e_{0}]=[a_{0}]$, $[h_{0}]=[b_{0}]$. Now consider the cycle in $G_{\zeta}$ $(u_{0},v_{0})\ldots(u_{l_{1}l_{2}k^{2}},v_{l_{1}l_{2}k^{2}})$ where $u_{l_{1}kj}\ldots u_{l_{1}k(j+1)}=e_{0},\ldots,e_{l_{1}k}$ and $v_{l_{2}km}\ldots v_{l_{2}k(m+1)}=h_{0}\ldots h_{l_{2}k}$ for all $0\leq j<l_{2}k$, $0\leq m<l_{1}k$. Since $\zeta$ is pair-aperiodic, there exists ${\bm f}\in F_{1}^{\zeta}$ such that $\zeta(e_{0})_{{\bm f}}=e_{0}$, $\zeta(h_{0})_{{\bm f}}=h_{0}$. We then conclude that $\zeta'([a])_{{\bm f}}=[a]$, $\zeta'([b])_{{\bm f}}=[b]$, i.e., $\zeta'$ is pair-aperiodic.
	
	On the other hand, for all $n>0$, we have that $\tau(\zeta^{n}(a))=\tau'(\zeta^{'n}([a]))$, hence $Y$ has the same language as $\tau'(X_{\zeta'})$, so they are equal, since subshifts are uniquely determined by their language.
	
	Finally, we prove that $\tau':(X_{\zeta'},S,\Z^{d})\to (Y,S,\Z^{d})$ is a conjugacy. Let $x,x'\in X_{\zeta'}$, with $\tau'(x)=\tau'(x')$. By the recognizability property of $X_{\zeta'}$, we can write $x=S^{{\bm f}_{1}}\zeta^{2|\A|^{2}}(y)$, $x'=S^{{\bm f}_{2}}\zeta^{2|\A|^{2}}(y')$. By \cref{RemarkRecognizabilityFactorOdometer} we have that ${\bm f}_{1}={\bm f}_{2}$.  
	
	If there exists ${\bm n}\in \Z^{d}$ such that $(y_{{\bm n}},y'_{{\bm n}})\in \A^{2}$ is an asymptotically disjoint pair we can find a periodic pair $([a],[b])$ and a path $P=([a_{0}],[b_{0}])\ldots([a_{k}],[b_{k}])$ in $G_{\zeta'}$ with $([a_{0}],[b_{0}])=(y_{{\bm n}},y'_{{\bm n}})$ and $([a_{k}],[b_{k}])=([a],[b])$. with $k\leq |A|^{2}$. Since $\zeta'$ is pair-aperiodic, we have that there exists ${\bm f}\in F_{2|A|^{2}}^{\zeta}$ such that $(\zeta')^{2|\A|^{2}}(y_{{\bm n}})_{{\bm f}}=[a]$, $(\zeta')^{2|\A|^{2}}(y'_{{\bm n}})_{{\bm f}}=[b]$. Note that $\tau'(\zeta'^{2|A|^{2}}(y_{{\bm n}}))=\zeta'^{2|\A|^{2}}(\tau'(y_{{\bm n}}))=\zeta'^{2|\A|^{2}}(\tau'(y_{{\bm n}}))=\tau'(\zeta'^{2|A|^{2}}(y'_{{\bm n}}))$, so $\tau'(a)$ is equal to $\tau'(b)$. This implies that $([a],[b])$ are indistinguishable, which contradicts \eqref{Distinguishablepairs}.
	
	If $(y_{{\bm n}},y'_{{\bm n}})$ is not an asymptotically disjoint pair for any ${\bm n}\in \Z^{d}$,  by Remark \ref{Remarkpairaperiodic} we have that $\zeta^{2|\A|^{2}}(y_{{\bm n}})=\zeta^{2|\A|^{2}}(y'_{{\bm n}})$, i.e., $x=x'$. We conclude $\tau'$ is bijective and then a conjugacy.
\end{proof}

\section{Measurable morphisms between constant-shape substitutions}\label{Sectionproof}

In this section, we study different types of homomorphisms between substitutive subshifts. First, we extend a result of B. Host and F. Parreau from \cite{host1989homomorphismes} to the multidimensional case (\cref{MainTheorem}), which establishes that any measurable factor between two substitutive subshifts given by aperiodic constant-shape substitutions, under some combinatorial property, induces a continuous one. Then, we deduce some consequences: Every substitutive subshift given by an aperiodic constant-shape substitution satisfying the combinatorial property is coalescent (\cref{Coalescence}), the quotient group $\Aut(X_{\zeta},S,\Z^{d})/\left\langle S \right\rangle$ is finite, generalizing the results in \cite{bustos2022admissible} and we give some conditions where the automorphism group of a substitutive subshift is isomorphic to a direct product of $\Z^{d}$ with a finite group (\cref{TrivialFactorsImpliesDirectProduct}). Finally, we extend \cref{MainTheorem} to homomorphisms associated with matrices commuting with a power of the expansion matrix of the substitution (\cref{NormalizerHostParreau}). This leads to the same rigidity properties about these homomorphisms (\cref{CommutingCentralizerVirtuallyZd}) and for some restricted normalizer group. Notice that in \cref{SectionBijectiveSubstitutions} we will give sufficient conditions to ensure that the former result is a complete characterization of the normalizer group. 	

\subsection{Measurable factors imply continuous ones for substitutive subshifts}\label{SubsectionMeasurableFactors}

Let $\zeta$ be a constant-shape substitution. Recall that the substitutive subshift $(X_{\zeta},S,\Z^{d})$ is uniquely ergodic and we denote $\mu_{\zeta}$ the unique ergodic measure. For any $n>0$, the image measure of $\mu_{\zeta}$ by $\zeta^{n}$ is equal to$|F_{n}^{\zeta}|\cdot \mu_{\zeta}|_{\zeta^{n}(X_{\zeta})}$, and as in the one-dimensional case \cite{queffelec2010substitution}, the unique ergodic measure satisfies
$$\forall U\in \mathcal{F}_{X_{\zeta}},\  \mu_{\zeta}(U)=\dfrac{1}{|F_{n}^{\zeta}|}\displaystyle\int_{X_{\zeta}} \left|\left\{{\bm f}\in F_{n}^{\zeta}: S^{{\bm f}}\zeta^{n}(x)\in U\right\}\right|d\mu_{\zeta}(x),$$

\noindent where $\mathcal{F}_{X_{\zeta}}$ corresponds to the Borel sets of $X_{\zeta}$.

As in \cite{host1989homomorphismes} we use the notion of \emph{reducibility} of a constant-shape substitution. For any pair of distinct letters $a\neq b\in \A$ and $n>0$ we consider the sequence 

$$d_{n}(\zeta^{n}(a),\zeta^{n}(b))=\dfrac{\left|\left\{{\bm f} \in F_{n}^{\zeta}\colon \zeta^{n}(a)_{{\bm f}}\neq \zeta^{n}(b)_{{\bm f}}\right\}\right|}{|F_{n}^{\zeta}|}.$$

This sequence is decreasing for all of the pairs $a,b\in \A$. We say that the  constant-shape substitution is \emph{reduced} if ${\min\limits_{\substack{n\in \NN\\a \neq b\in \A}}d_{n}(\zeta^{n}(a),\zeta^{n}(b))>0}$. For instance, every bijective constant-shape substitution is reduced. The following theorem is a multidimensional analogue of Theorem 1.3 in \cite{host1989homomorphismes}:

\begin{theorem}\label{MainTheorem} Let $(X_{\zeta_{1}},S,\Z^{d})$, $(X_{\zeta_{2}},S,\Z^{d})$ be two substitutive subshift from two aperiodic primitive  constant-shape substitutions $\zeta_{1}$, $\zeta_{2}$ from finite alphabets $\A$ and $\B$, with the same expansion matrix $L$ and same support $F_{1}$. If $\zeta_{2}$ is reduced, then for every measurable factor $\phi:(X_{\zeta_{1}},\mu_{\zeta_{1}},S,\Z^{d})\to (X_{\zeta_{2}},\mu_{\zeta_{2}},S,\Z^{d})$, there exists ${\bm j}\in \Z^{d}$ such that $S^{{\bm j}}\phi$ is equal $\mu_{\zeta_{1}}$-a.e. to a continuous factor map $\psi:(X_{\zeta_{1}},S,\Z^{d})\to(X_{\zeta_{2}},S,\Z^{d})$, satisfying the following two properties: 
	
	\begin{enumerate}[label=(\arabic*),ref=\text{(}\arabic*\text{)}]
		\item \label{FixedRadiusofCommutingFactors} The factor map $\psi$ is a sliding block code of radius $\Vert F_{1}^{\zeta}\Vert\left(1+\Vert L_{\zeta}^{-1}\Vert \left(2+1/(1-\Vert L_{\zeta}^{-1}\Vert)\right)\right)$.
		\item\label{CommutationPropertySlidingBlickCodeSubstitution} There exist an integer
		$n>0$ and ${\bm p}\in F_{n}^{\zeta}$ such that, $S^{{\bm p}}\psi\zeta_{1}^{n}=\zeta_{2}^{n}\psi$.
	\end{enumerate} 
\end{theorem}

\begin{remark}\label{RemarkMainTheorem} The following statements can be easily verified.
	\begin{enumerate}
		\item If $L_{\zeta}$ is a diagonal matrix, then $\Vert L_{\zeta}^{-1}\Vert \leq 1/2$, so the radius of $\psi$ is bounded by $3\Vert F_{1}^{\zeta}\Vert$.
		
		\item Since the set of sliding block codes of radius $\Vert F_{1}^{\zeta}\Vert\left(1+\Vert L_{\zeta}^{-1}\Vert \left(2+1/(1-\Vert L_{\zeta}^{-1}\Vert)\right)\right)$ between $X_{\zeta_{1}}$ and $X_{\zeta_{2}}$ is finite, we may consider an appropriate iteration of $\zeta_{1}$ and $\zeta_{2}$ such that any factor $\psi$ satisfying \ref{CommutationPropertySlidingBlickCodeSubstitution} in \cref{MainTheorem} satisfies $S^{{\bm p}}\psi \zeta_{1}=\zeta_{2}\psi$.
		
		\item As also mentioned in \cite{host1989homomorphismes}, we may assume that ${\bm p}\notin (L_{\zeta}^{n}-\id)(\Z^{d})$, because it is equivalent to find a factor map commuting with the substitution map, i.e., $\psi \zeta^{n}=\zeta^{n}\psi$, with ${\bm p}\in F_{1}^{\zeta}$.
	\end{enumerate}
	
\end{remark}

If a substitution is not reduced, we consider an equivalence relation calling two letters $a,b$ equivalent when $d_{n}(\zeta^{n}(a),\zeta^{n}(b))\to 0$. If two letters $a,b\in \A$ are equivalent, then $(\zeta(a))_{{\bm f}}\sim (\zeta(b))_{{\bm f}}$ for all ${\bm f}\in F_{1}^{\zeta}$. We define a substitution $\tilde{\zeta}$ on $\A/\sim$ given by $(\tilde{\zeta}([a]))_{{\bm f}}=[\zeta(a)_{{\bm f}}]$ for ${\bm f}\in F_{1}^{\zeta}$. We have a natural letter-to-letter factor map $\tilde{\phi}:(X_{\zeta},S,\Z^{d})\to(X_{\tilde{\zeta}},S,\Z^{d})$, and it is called the \emph{reduced substitution} of $\zeta$.

In the one-dimensional case, if $(X_{\zeta},\mu_{\zeta},S,\Z)$ does not have purely discrete spectrum, it can be proved, using the results in \cite{dekking1978spectrum}, that $(X_{\tilde{\zeta}},S,\Z)$ is aperiodic. An ergodic system $(X,\mu,T,\Z)$ has \emph{purely discrete spectrum} if the vector space spanned by eigenfunctions of the Koopman operator $U_{T}(f)=f\circ T \in L^{2}(X,\mu)$ is dense in $L^{2}(X,\mu)$. It is well known that any ergodic system with purely discrete spectrum is conjugate to an ergodic rotation on a compact abelian group \cite{halmos1942operator}.

In the multidimensional case this is not true in general, as we can see in \cref{ExampleNotAperiodicReducedSubstitution}.

\begin{example}[An aperiodic constant-shape substitution, with a periodic reduced substitution]\label{ExampleNotAperiodicReducedSubstitution}
	\stepcounter{sigmavariable}
	\setcounter{onevariable}{\value{sigmavariable}}
	\stepcounter{sigmavariable}
	\setcounter{twovariable}{\value{sigmavariable}}
	\stepcounter{sigmavariable}
	
	Consider the substitution $\sigma_{\theonevariable}$ with $L_{\sigma_{\theonevariable}}=\left(\begin{array}{cc}
		2 & 0\\ 0 & 2
	\end{array}\right)$ and $F_{1}^{\sigma_{\theonevariable}}=\llbracket 0,1\rrbracket^{2}$, given by
	$$\begin{array}{cccccccccccc}
		\sigma_{\theonevariable}: \\
		&0\mapsto & \begin{array}{cc}
			1 & 3\\ 0 & 2
		\end{array},& \quad & 1\mapsto & \begin{array}{cc}
			0 & 2\\ 0 & 2
		\end{array},& \quad &  2\mapsto & \begin{array}{cc}
			3 & 1\\ 2 & 0
		\end{array},& \quad & 3\mapsto & \begin{array}{cc}
			2 & 0\\ 2 & 0
		\end{array}.
	\end{array}$$	
	
	This substitution corresponds to the product substitution between the Thue-Morse substitution ($\sigma_{\thetwovariable}: 0\mapsto 01,\ 1\mapsto 10$) and the doubling sequence substitution ($\sigma_{\thesigmavariable}: a\mapsto ab,\ b\mapsto aa$). Note that $a\sim b$, since for all $n>0$, $\sigma_{\thesigmavariable}^{n}(a)$ and $\sigma_{\thesigmavariable}^{n}(b)$ differ only in their last letters. The following is a pattern of $\sigma_{\theonevariable}$:
	
	\begin{figure}[H]
	$$\begin{array}{cccccc}
		1331311331131331\\
		0220200220020220\\
		0220200220020220\\
		0220200220020220\\
		1331311331131331\\
		0220200220020220
	\end{array}$$
	\caption{A patttern of the substitution $\sigma_{\theonevariable}$.}
	\end{figure}

	The system $(X_{\sigma_{\thetwovariable}}\times \overleftarrow{\Z}_{(2^{n}\Z)},S\times +_{(2^{n}\Z)},\Z^{2})$ is a topological factor of $(X_{\sigma_{\theonevariable}},S,\Z^{2})$, hence the substitutive subshift $(X_{\sigma_{\theonevariable}},S,\Z^{2})$ does not have purely discrete spectrum. The reduced substitution of $\sigma_{\theonevariable}$ is defined with the same expansion matrix and support, given by:
	
	$$\begin{array}{cccccccccccc}
		\tilde{\sigma}_{\theonevariable}: \\
		&a\mapsto & \begin{array}{cc}
			a & b\\ a & b
		\end{array},& \quad & b\mapsto & \begin{array}{cc}
			b & a\\ b & a,
		\end{array}
	\end{array}$$
	
	\noindent where every element in $\{0\}\times \Z$ is a nontrivial period of $\tilde{\sigma}_{\theonevariable}$.
\end{example}

However, as proved in \cite{host1989homomorphismes} for the one-dimensional case, if the reduced substitution system is aperiodic, then $(X_{\tilde{\zeta}},\mu_{\zeta},S,\Z^{d})$ is metrically isomorphic to $(X_{\zeta},\mu_{\tilde{\zeta}},S,\Z^{d})$.

\begin{proposition}\label{SpectrumNotPurelyDiscrete}
	Let $\zeta$ be an aperiodic primitive  constant-shape substitution. If the reduced substitution $\tilde{\zeta}$ is aperiodic, then the natural factor map between $(X_{\zeta},\mu_{\zeta},S,\Z^{d})$ and $(X_{\tilde{\zeta}},\mu_{\tilde{\zeta}},S,\Z^{d})$ is a metric isomorphism.
\end{proposition}

\begin{proof}
	Let $\pi:(X_{\zeta},S,\Z^{d})\to (\overleftarrow{\Z^{d}}_{(L_{\zeta}^{n})},+_{(L_{\zeta}^{n})},\Z^{d})$ defined in \cref{SectionSubstitutionExtensionOdometer}. By \cref{RecognizabilityFactors}, for every $n>0$,  $\pi_{n}(x)$ is equal to $\pi_{n}(\tilde{\phi}(x))$. In particular, if $x,y\in X_{\zeta}$ satisfies $\tilde{\phi}(x)=\tilde{\phi}(y)$, then $\pi(x),\pi(y)$ are equal.
	
	Set $U=\{x\in X_{\zeta}: \exists y\in X_{\zeta},\ \tilde{\phi}(x)=\tilde{\phi}(y),\ x_{{\bm 0}}\neq y_{{\bm 0}}\}$. It is enough to prove that $U$ is a null-set.
	
	Let $n>0$, ${\bm f}\in F_{n}^{\zeta}$ and $x\in X_{\zeta}$ be such that $S^{{\bm f}}\zeta^{n}(x)\in U$. Then, there exists $y\in X_{\zeta}$ with $\tilde{\phi}(y)=\tilde{\phi}(S^{{\bm f}}\zeta^{n}(x))$ and $y_{0}\neq \zeta^{n}(x_{0})_{{\bm f}}$. Then $\pi_{n}(y)=\pi_{n}(S^{{\bm f}}\zeta^{n}(x))$ and is equal to ${\bm f}$. Moreover, there exists $z\in X_{\zeta}$ with $y=S^{{\bm f}}\zeta^{n}(z)$, so $\tilde{\phi}(x)$ is equal to $\tilde{\phi}(z)$. This implies that $(\zeta^{n}z_{{\bm 0}})_{{\bm j}}, (\zeta^{n}x_{{\bm 0}})_{{\bm j}}$ are equivalent for all ${\bm j}\in F_{n}^{\zeta}$. Note that $(\zeta^{n}z_{{\bm 0}})_{\bm f}=y_{{\bm 0}}$, so it is different from $(\zeta^{n}x_{{\bm 0}})_{{\bm f}}$. We define the set 
	
	$$G_{n}=\bigcup\limits_{a,b\in \A}\left\{{\bm f}\in F_{n}^{\zeta}\colon [(\zeta^{n}a)_{{\bm f}}]= [(\zeta^{n}b)_{{\bm f}}],\ (\zeta^{n}a)_{{\bm f}} \neq (\zeta^{n}b)_{{\bm f}}\right\}.$$
	
	We deduce from the previous paragraph that
	\begin{align*}
		\mu_{\zeta}(U) & =\dfrac{1}{|F_{n}^{\zeta}|}\displaystyle\int \left|\left\{{\bm f}\in F_{n}^{\zeta}: S^{{\bm f}}\zeta^{n}(x)\in U\right\}\right|d\mu(x) \\ & \leq \dfrac{|G_{n}|}{|F_{n}^{\zeta}|}.
	\end{align*}
	
	For any $a,b\in \A$ we denote $\mathcal{D}_{n}^{a,b}=\left\{(c,d)\in \A^{2}\colon \exists {\bm f}\in F_{n}^{\zeta},\  (\zeta^{n}a)_{{\bm f}}=c,\ (\zeta^{n}b)_{{\bm f}}=d,\ [c]\neq [d]\right\}$ and $\mathcal{E}_{n}^{a,b}=\left\{(c,d)\in \A^{2}\colon \exists {\bm f}\in F_{n}^{\zeta},\  (\zeta^{n}a)_{{\bm f}}=c,\ (\zeta^{n}b)_{{\bm f}}=d,\ [c]=[d]\right\}$. Set $\varepsilon>0$ and let $j>0$ be large enough such that for any $a,b\in \A$ 
	$$d_{j}(\zeta^{j}(a),\zeta^{j}(b))\leq \lim\limits_{k\to \infty} d_{k}(\zeta^{k}(a),\zeta^{k}(b))+\varepsilon.$$
	
	\noindent Fix $a,b\in \A$. Note that
	\begin{align*}
		d_{n+j}(\zeta^{n+j}(a),\zeta^{n+j}(b)) & = \dfrac{1}{|F_{n+j}^{\zeta}|}\sum\limits_{(c,d)\in \mathcal{D}_{n}^{a,b}} |F_{j}^{\zeta}|d_{j}(\zeta^{j}(c),\zeta^{j}(d))+\sum\limits_{(c,d)\in \mathcal{E}_{n}^{a,b}} |F_{j}^{\zeta}|d_{j}(\zeta^{j}(c),\zeta^{j}(d))\\
		& = \dfrac{1}{|F_{n}^{\zeta}|}\left(\sum\limits_{(c,d)\in \mathcal{D}_{n}^{a,b}} d_{j}(\zeta^{j}(c),\zeta^{j}(d))+\sum\limits_{(c,d)\in \mathcal{E}_{n}^{a,b}} d_{j}(\zeta^{j}(c),\zeta^{j}(d))\right)\\
		& \leq \dfrac{1}{|F_{n}^{\zeta}|}\left(\sum\limits_{(c,d)\in \mathcal{D}_{n}^{a,b}} (\lim\limits_{k\to \infty}d_{k}(\zeta^{k}(c),\zeta^{k}(d))+\varepsilon)+\sum\limits_{(c,d)\in \mathcal{E}_{n}^{a,b}} \varepsilon\right)\\
		& \leq \dfrac{\varepsilon(|\mathcal{D}_{n}^{a,b}|+|\mathcal{E}_{n}^{a,b}|)}{|F_{n}^{\zeta}|}+\dfrac{1}{|F_{n}^{\zeta}|}\sum\limits_{(c,d)\in \mathcal{D}_{n}^{a,b}}\lim\limits_{k\to \infty}d_{k}(\zeta^{k}(c),\zeta^{k}(d)).
	\end{align*}
	
	\noindent Since this is for every $\varepsilon>0$ and $\lim\limits_{k\to \infty}d_{k}(\zeta^{k}(c),\zeta^{k}(d))\leq 1$, we have that $d_{n+j}(\zeta^{n+j}(a),\zeta^{n+j}(b))\leq |\mathcal{D}_{n}^{a,b}|/|F_{n}^{\zeta}|$ and this is true for every $j$ large enough, so
	$$\lim\limits_{k\to\infty} d_{k}(\zeta^{k}(a),\zeta^{k}(b))\leq \dfrac{|\mathcal{D}_{n}^{a,b}|}{|F_{n}^{\zeta}|},$$
	
	\noindent hence
	\begin{align*}
		\mu_{\zeta}(U) &  \leq \sum\limits_{a,b\in \A}\left(d_{n}(\zeta^{n}(a),\zeta^{n}(b)) - \lim\limits_{k\to \infty} d_{k}(\zeta^{k}(a),\zeta^{k}(b))\right).
	\end{align*}
	
	When $n\to \infty$, the right expression goes to zero, and we conclude that $\mu_{\zeta}(U)=0$.		
\end{proof}

Now, to prove \cref{MainTheorem}, we assume that $\zeta_{2}$ is an aperiodic primitive reduced constant-shape substitution. We denote by $\eta=\min\limits_{\substack{n\in \NN\\a \neq b\in \B}}d_{n}(\zeta_{2}^{n}(a),\zeta_{2}^{n}(b))$ and $R$ the radius from the recognizability property of $\zeta_{2}$. Let $\phi$ be in $m\Fac(X_{\zeta_{1}},X_{\zeta_{2}},S,\Z^{d})$. The map $\pi_{n}(x)-\pi_{n}(\phi x)\ (\Mod L^{n}(\Z^{d}))$ is invariant under the $\Z^{d}$-action, so it is constant $\mu_{\zeta_{1}}$-a.e. We denote this constant by ${\bm p}_{n}(\phi)$ in $F_{n}^{\zeta_{1}}$. The set $S^{{\bm p}_{n}(\phi)} \phi \zeta_{1}^{n}(X_{\zeta_{1}})$ is included in $\zeta_{2}^{n}(X_{\zeta_{2}})$ up to a $\mu_{\zeta_{2}}$-null set. Recall that the recognizability property implies that $\zeta_{1}^{n}$ is a homeomorphism from $X_{\zeta_{1}}$ to $\zeta_{1}^{n}(X_{\zeta_{1}})$, so for $\mu_{\zeta_{1}}$-almost all $x\in X_{\zeta_{1}}$ there exists a unique point $y\in X_{\zeta_{2}}$ such that $S^{{\bm p}_{n}(\phi)}\phi\zeta_{1}^{n}(x) = \zeta_{2}^{n}(y)$, which we denote $\phi_{n}(x)$.  So, for every $\phi\in m\Fac(X_{\zeta_{1}},X_{\zeta_{2}},S,\Z^{d})$ we consider a sequence $({\bm p}_{n}(\phi))_{n\geq 0}$ and a sequence of maps $(\phi_{n})_{n}\in m\Fac(X_{\zeta_{1}},X_{\zeta_{2}},S,\Z^{d})$ such that
$${\bm p}_{n}(\phi)\in F_{n}^{\zeta_{1}},\quad S^{{\bm p}_{n}(\phi)}\phi\zeta_{1}^{n}(x) = \zeta_{2}^{n}(\phi_{n}(x)).$$

It is straightforward to check that the sequence $({\bm p}_{n}(\phi))_{n>0}$ satisfies the recurrence ${\bm p}_{n+1}(\phi)={\bm p}_{n}(\phi)+L_{\zeta_{1}}^{n}{\bm p}_{1}(\phi_{n})\quad {(\text{mod}\ L_{\zeta_{1}}^{n+1}(\Z^{d}))}$. We also have the recurrence $(\phi_{n})_{1}=\phi_{n+1}$.

Now, for $\phi,\theta\in m\Fac(X_{\zeta_{1}},X_{\zeta_{2}},S,\Z^{d})$, we denote $d(\phi,\theta)=\mu_{\zeta_{1}}(\{x\in X_{\zeta_{1}}: (\phi x)_{{\bm 0}}\neq (\theta x)_{{\bm 0}}\})$. We also denote, for any $r>0$, the quantity $C(r)=|B({\bm 0},r)\cap \Z^{d}|$.

\begin{lemma}\label{EqualityofNormalizers}
	If $d(\phi,\theta)$ is smaller than $\eta/C(R)$, then $\phi,\theta$ are equal $\mu_{\zeta_{1}}$-a.e. in $X_{\zeta_{1}}$.
\end{lemma}

\begin{proof}
	For all $n\geq 0$, we denote $U_{n}=\{x\in X_{\zeta_{1}}: (\phi_{n}x)_{{\bm 0}}\neq (\theta_{n}x)_{{\bm 0}}\}$. We will prove by induction on $n\geq 0$ that ${\bm p}_{n}(\phi)={\bm p}_{n}(\theta)$ and $\mu_{\zeta_{1}}(U_{n})<1/C(R)$. By hypothesis, this is true for $n=0$. Now, suppose that ${\bm p}_{n}(\phi)={\bm p}_{n}(\theta)$ and $\mu_{\zeta_{1}}(U_{n})<1/C(R)$ for some $n\geq 0$. We note that the map $\pi_{1}(\phi_{n}x)-\pi_{1}(\theta_{n}x)$ is equal to $({\bm p}_{1}(\theta_{n})-{\bm p}_{1}(\phi_{n}))\ (\text{mod}\ L_{\zeta_{1}}^{n}(\Z^{d}))$ for $\mu_{\zeta_{1}}$-a.e $x$ in $X_{\zeta_{1}}$. By the recognizability property, this map vanishes on the set $\{x\in X_{\zeta_{1}}: (\phi_{n}x)|_{B({\bm 0},R)}= (\theta_{n}x)|_{B({\bm 0},R)}\}$, which has a positive measure by hypothesis. We conclude that ${\bm p}_{n+1}(\phi)={\bm p}_{n+1}(\theta)$.
	
	Now, let $x$ be in $U_{n+1}$. Then, there exist at least $\eta|F_{n+1}^{\zeta_{1}}|$ indices ${\bm f}\in F_{n+1}^{\zeta_{1}}$ such that $(\zeta_{2}^{n+1}\phi_{n+1}(x))_{{\bm f}}\neq (\zeta_{2}^{n+1}\theta_{n+1}(x))_{{\bm f}}$, i.e., $(S^{{\bm p}_{n+1}(\phi)}\phi\zeta_{1}^{n+1}(x))_{{\bm f}}\neq (S^{{\bm p}_{n+1}(\theta)}\theta\zeta_{1}^{n+1}(x))_{{\bm f}}$, so we have that
	$$\eta |F_{n+1}^{\zeta_{1}}|\mu_{\zeta_{1}}(U_{n+1})\leq \int |\{{\bm g} \in F_{n+1}^{\zeta_{1}}+{\bm p}_{n+1}(\phi): S^{{\bm g}}\zeta_{1}^{n+1}x \in U_{0}\}|d\mu_{\zeta_{1}}(x)= |F_{n+1}^{\zeta_{1}}|\mu_{\zeta_{1}}(U_{0})$$
	
	\noindent Hence, $\mu_{\zeta_{1}}(U_{n+1})$ is less than $1/C(R)$. 
	
	Now, we will prove that $\phi=\theta$ for $\mu_{\zeta_{1}}$-a.e. $x\in X_{\zeta_{1}}$. Let $r>0$ be an integer. Since $(F_{n}^{\zeta_{1}})_{n>0}$ is F\o lner, we choose $n>0$ large enough such that $|(F_{n}^{\zeta_{1}})^{\circ r}|/|\{F_{n}^{\zeta_{1}}\}|\geq 1/2$. Set ${\bm p}={\bm p}_{n}(\phi)$. If $x\in U_{0}^{c}$, then $(S^{{\bm p}}\phi\zeta_{1}^{n}x)|_{F_{n}^{\zeta_{1}}}=(S^{{\bm p}}\theta\zeta_{1}^{n}x)|_{F_{n}^{\zeta_{1}}}$, so we have that $(S^{{\bm p}+{\bm f}}\phi\zeta_{1}^{n}x)_{B({\bm 0},r)\cap \Z^{d}}=(S^{{\bm p}+{\bm f}}\theta\zeta_{1}^{n}x)_{B({\bm 0},r)\cap \Z^{d}}$ for ${\bm f} \in (F_{n}^{\zeta_{1}})^{\circ r}$. This implies that
	
	\begin{align*}
		\mu_{\zeta_{1}}\left(\left\{x\in X: (\phi x)|_{B({\bm 0},r)}= (\theta x)_{B({\bm 0},r)}\right\}\right) & \geq \dfrac{|(F_{n}^{\zeta_{1}})^{\circ r}|}{|F_{n}^{\zeta_{1}}|}\mu_{\zeta_{1}}(U_{n}^{c})\\
		& \geq \dfrac{1}{2}\left(1-\dfrac{1}{C(R)}\right)>0.
	\end{align*}
	
	Finally, the set $\{x\in X\colon \phi(x)=\theta(x)\}$ is the decreasing intersection of these sets, so it has a positive measure. By ergodicity, $\phi, \theta$ are equal $\mu_{\zeta_{1}}$-a.e. in $X_{\zeta_{1}}$. 
\end{proof}

\begin{lemma}\label{convergenceofslidingblockcodes} 
	Let $\phi$ be in $m\Fac(X_{\zeta_{1}},X_{\zeta_{2}},S,\Z^{d})$. Then there exists a sequence $(\psi_{n})$ of sliding block codes of radius $\Vert F_{1}^{\zeta_{1}}\Vert\left(1+\Vert L_{\zeta_{1}}^{-1}\Vert \left(2+1/(1-\Vert L_{\zeta_{1}}^{-1}\Vert)\right)\right)$ such that $d(\phi_{n},\psi_{n})\to 0$.
\end{lemma}

\begin{proof}
	Set $\varepsilon>0$. By Lusin's theorem, there exist an integer $\ell>0$ and a continuous map $f:X_{\zeta_{1}}\to \mathcal{B}$ such that $f(x)$ only depends on $x|_{B({\bm 0},\ell)}$, and the measure of the set $V=\{x\in X_{\zeta_{1}}\colon (\phi x)_{{\bm 0}}\neq f(x)\}$ is less than $\varepsilon \eta/6$. Since $(F_{n}^{\zeta_{1}})_{n>0}$ is F\o lner, we choose $n>0$ large enough such that $$\dfrac{|(F_{n}^{\zeta_{1}})^{\circ \ell}|}{|F_{n}^{\zeta_{1}}|}>1-\dfrac{\eta}{3}.$$
	
	Set ${\bm p}={\bm p}_{n}(\phi)$. For every $x\in X_{\zeta_{1}}$, we denote $J(x)=\{{\bm f}\in (F_{n}^{\zeta_{1}})^{\circ \ell}\colon\ S^{{\bm f}+{\bm p}}\zeta_{1}^{n}x\notin V\}$ and $W=\{x\in X_{\zeta_{1}}\colon |J(x)|>(1-\eta/2)|F_{n}^{\zeta_{1}}|\}$. Note that
	
	\begin{align*}
		\frac{\varepsilon \eta}{6} & >\mu_{\zeta_{1}}(V)\\
		& \geq \frac{1}{|F_{n}^{\zeta_{1}}|}\displaystyle\int_{W^{c}} |(F_{n}^{\zeta_{1}})^{\circ \ell}|-|J(x)|)d\mu(x)\\
		& \geq (1-\mu_{\zeta_{1}}(W))\frac{\eta}{6}.
	\end{align*}
	
	So we have that $\mu_{\zeta_{1}}(W)>1-\varepsilon$.
	
	Using \cref{SetDforFiniteInvariantOrbit} with $F=F_{1}^{\zeta}+F_{1}^{\zeta}$ and $A=\{{\bm 0}\}$, we find a set $C\Subset \Z^{d}$ such that ${\bm 0}\in C$, $\Vert C \Vert \leq \Vert F_{1}^{\zeta}\Vert\left(1+\Vert L_{\zeta}^{-1}\Vert \left(2+1/(1-\Vert L_{\zeta}^{-1}\Vert)\right)\right)$ and by \cref{RemarkSetDforInvariantOrbit} we have that $F_{n}^{\zeta}+F_{n}^{\zeta}\subseteq L_{\zeta}^{n}(C)+F_{n}^{\zeta}$. If $x,y\in W$ with $x|_{C}=y|_{C}$, then for every ${\bm f}\in (F_{n}^{\zeta_{1}})^{\circ \ell}$, we have that $(S^{{\bm f}+{\bm p}}\zeta_{1}^{n}x)|_{B({\bm 0},\ell)\cap \Z^{d}}=(S^{{\bm f}+{\bm p}}\zeta_{1}^{n}y)|_{B({\bm 0},\ell)\cap \Z^{d}}$, so $f(S^{{\bm f}+{\bm p}}\zeta_{1}^{n}x)=f(S^{{\bm f}+{\bm p}}\zeta_{1}^{n}y)$. Moreover, we note that for ${\bm f}$ in $J(x)\cap J(y)$ that, 
	$(\zeta_{2}^{n}\phi_{n}x)_{{\bm f}} = (S^{{\bm p}}\phi\zeta_{1}^{n}x)_{{\bm f}}=f(S^{{\bm f} + {\bm p}}\phi\zeta_{1}^{n}x) =f(S^{{\bm f} + {\bm p}}\phi\zeta_{1}^{n}y) =(\zeta_{2}^{n}\phi_{n}y)_{{\bm f}}$. Since $x,y\in W$, there is strictly more than $(1-\eta)|F_{n}^{\zeta_{1}}|$ elements in $J(x)\cap J(y)$, so  $(\phi_{n}x)_{{\bm 0}}$ is equal to $(\phi_{n}y)_{{\bm 0}}$ by definition of $\eta$. Hence, for every $x$ in $W$, $(\phi_{n}x)_{{\bm 0}}$ only depends on $x|_{C}$.		
\end{proof}

Finally, to prove \cref{MainTheorem} we use similar arguments given in \cite{host1989homomorphismes} that we describe for completeness. 

\begin{proof}[Proof of Theorem \ref{MainTheorem}]
	For fixed alphabets $\A$ and $\B$, there exist a finite number of sliding block codes of radius $\Vert F_{1}^{\zeta}\Vert\left(1+\Vert L_{\zeta}^{-1}\Vert \left(2+1/(1-\Vert L_{\zeta}^{-1}\Vert)\right)\right)$. By Lemma \ref{convergenceofslidingblockcodes}, there exist two different integers $m,k\geq 0$ such that $d(\phi_{m},\phi_{m+k})<\eta/C(R)$, so by Lemma \ref{EqualityofNormalizers}, we have that $\phi_{m}=\phi_{m+k}$, $\mu_{\zeta_{1}}$-a.e. 
	
	Let $n\geq m$ be a multiple of $k$. Note that $(\phi_{n})_{k}=\phi_{n+k}=(\phi_{m+k})_{n-m}=(\phi_{m})_{n-m}=\phi_{n}$, $\mu_{\zeta_{1}}$-a.e. This implies that for all $r\in \NN$, $\phi_{n}$ is equal to $(\phi_{n})_{rk}$, $\mu_{\zeta_{1}}$-a.e. and then $\phi_{n}$ is equal to a sliding block code of radius $\Vert F_{1}^{\zeta}\Vert\left(1+\Vert L_{\zeta}^{-1}\Vert \left(2+1/(1-\Vert L_{\zeta}^{-1}\Vert)\right)\right)$, $\mu_{\zeta_{1}}$-a.e. in $X_{\zeta_{1}}$. Since $\phi_{n}$ is equal to $\phi_{2n}$, $\mu_{\zeta_{1}}$-a.e., we denote $\psi=\phi_{n}$ and ${\bm p}={\bm p}_{n}(\psi)$. By definition of ${\bm p}$, we have that $S^{{\bm p}}\psi\zeta_{1}^{n}=\zeta_{2}^{n}\psi$. 
	
	Set ${\bm j}={\bm p}_{n}(\phi)-{\bm p}$, then
	$$S^{{\bm j}}\phi\zeta_{1}^{n}=S^{{\bm p}_{n}(\phi)-{\bm p}}\phi\zeta_{1}^{n}=S^{-{\bm p}}\zeta_{2}^{n}\psi=\psi\zeta_{1}^{n},\quad \mu_{\zeta_{1}}-\text{a.e},$$
	this implies that $S^{{\bm j}}\phi$ and $\psi$ coincides in $\zeta_{1}^{n}(X_{\zeta_{1}})$ $\mu_{\zeta_{1}}$-almost everywhere, and by ergodicity in the whole set $X_{\zeta_{1}}$ $\mu_{\zeta_{1}}$-almost everywhere. 
\end{proof}

In \cite{durand2000linearly} it was proved that linearly recurrent subshifts (in particular substitutive subshifts) have a finite number of topological Cantor factors, up to conjugacy. In our context, \cref{FactorConjugateSubstitution} together with \cref{MainTheorem} implies that a substitutive subshift has finitely many reduced substitutive factors, up to conjugacy. However, the reduced hypothesis does not cover all the aperiodic symbolic factors a substitutive subshifts may have, and this property is not invariant by conjugation, leaving the following question:

\begin{question}
	Does all aperiodic substitutive subshifts have finitely many aperiodic symbolic factors, up to conjugacy?
\end{question}

\subsection{Applications of rigidity results on factors}\label{SectionApplicationsOfMainTheorem}

As applications of \cref{MainTheorem}, we get some results on the coalescence and the automorphism group of substitutive subshifts. Since the set of sliding block codes $\Vert F_{1}^{\zeta}\Vert\left(1+\Vert L_{\zeta}^{-1}\Vert \left(2+1/(1-\Vert L_{\zeta}^{-1}\Vert)\right)\right)$ is finite, we will assume here (up to considering a power of $\zeta$) that if a factor map $\psi\in \End(X_{\zeta},S,\Z^{d})$ satisfies Property \ref{CommutationPropertySlidingBlickCodeSubstitution} in \cref{MainTheorem}, then it does so for $n=1$, i.e., there exists ${\bm p}\in F_{1}^{\zeta}$ such that $S^{{\bm p}}\psi \zeta=\zeta \psi$.

\subsubsection{Coalescence of substitutive subshifts}\label{CoalescenceAndFactors} In \cite{durand2000linearly} it was proved that one-dimensional linearly recurrent subshifts (in particular substitutive subshifts) are coalescent, i.e, any endomorphism is an automorphism. Here we use \cref{MainTheorem} to obtain that substitutive subshifts are also coalescent, for aperiodic primitive reduced constant-shape substitutions.

\begin{proposition}\label{Coalescence}
	Let $\zeta$ be an aperiodic primitive reduced  constant-shape substitution. Then the substitutive subshift $(X_{\zeta},S,\Z^{d})$ is coalescent.
\end{proposition}

\begin{proof}
	Set $\phi\in \End(X_{\zeta},S,\Z^{d})$. \cref{MainTheorem} ensures that there exists ${\bm j}\in \Z^{d}$ such that $S^{{\bm j}}\phi$ is equal to a sliding block code $\psi$ of a fixed radius satisfying $S^{{\bm p}}\psi\zeta=\zeta \psi$, for some ${\bm p}\in F_{1}^{\zeta}$. Let $\overline{x}\in X_{\zeta}$ be in a $\zeta$-invariant orbit, i.e., there exists ${\bm j}\in \Z^{d}$ such that $\zeta(\overline{x})=S^{{\bm j}}\overline{x}$. Note that
	$$S^{{\bm p}}\psi\zeta(\overline{x})=S^{{\bm p}+{\bm j}}\psi\overline{x}=\zeta \psi(\overline{x}),$$
	
	\noindent so, if the orbit of $x$ is in a $\zeta$-invariant orbit, then $\psi(x)$ is also in a $\zeta$-invariant orbit. By Proposition \ref{FinitelyManyInvariantOrbits}, there exist finitely many $\zeta$-invariant orbits, hence for $n$ large enough, we can find $x\in X_{\zeta}$ with $x$ and $\psi^{n}(x)$ being in the same orbit, i.e., there exists ${\bm m}\in \Z^{d}$ such that $S^{{\bm m}}\psi^{n}(x)=x$. The minimality of $(X_{\zeta},S,\Z^{d})$ allows to conclude that $\psi^{n}=S^{-{\bm m}}$. Hence $\psi$ is invertible, which implies that $\phi$ is invertible.
\end{proof}

\subsubsection{The automorphism group of substitutive subshifts.} Since the set of sliding block codes of radius $\Vert F_{1}^{\zeta}\Vert\left(1+\Vert L_{\zeta}^{-1}\Vert \left(2+1/(1-\Vert L_{\zeta}^{-1}\Vert)\right)\right)$ between $X_{\zeta}$ to itself is finite, we get the following result as a direct corollary of \cref{MainTheorem}. 

\begin{proposition}\label{AutomoprhismVirtuallyZd}
	Let $(X_{\zeta},S,\Z^{d})$ be a substitutive subshift from an aperiodic primitive reduced  constant-shape substitution $\zeta$. Then, the quotient group $\Aut(X_{\zeta},S,\Z^{d})/\left\langle S\right\rangle$ is finite. A bound for $|\Aut(X_{\zeta},S,\Z^{d})/\left\langle S \right\rangle|$ is given by an explicit formula depending only on $d$, $|\A|$, $\Vert L_{\zeta}^{-1}\Vert$, $\Vert F_{1}^{\zeta}\Vert$.
\end{proposition}

In the special case, where any automorphism of $(X_{\zeta},S,\Z^{d})$ satisfies Property \ref{CommutationPropertySlidingBlickCodeSubstitution} of \cref{MainTheorem} with ${\bm p}={\bm 0}$, i.e., commutes with the substitution map (like in bijective substitutions as we will prove in \cref{SectionBijectiveSubstitutions}), we have a more rigid result.

\begin{corollary}\label{TrivialFactorsImpliesDirectProduct}
	Let $\zeta$ be an aperiodic primitive reduced substitution. If any automorphism $\psi\in \Aut(X_{\zeta},S,\Z^{d})$ satisfying Property \ref{CommutationPropertySlidingBlickCodeSubstitution} in \cref{MainTheorem} commutes with the substitution map, i.e., $\psi\zeta=\zeta\psi$, then the automorphism group is isomorphic to a direct product of $\Z^{d}$ (generated by the shift action) with a finite group. 
\end{corollary}

\begin{proof}
	Note that an automorphism $\phi$ commutes with the substitution map if and only if $\phi$ is equal to $\phi_{n}$, for all $n>0$.
	
	Now, by \cref{convergenceofslidingblockcodes}, the Property \ref{CommutationPropertySlidingBlickCodeSubstitution} implies Property \ref{FixedRadiusofCommutingFactors} of \cref{MainTheorem}. So, the group of automorphisms commuting with the substitution map is finite. To conclude we just need to observe that the pair $({\bm j}_{\phi},\psi_{\phi})$ in \cref{MainTheorem} associated with any automorphism $\phi$ is unique. Indeed, set $\phi\in \Aut(X_{\zeta},S,\Z^{d})$ and ${\bm j}_{1},{\bm j}_{2}\in \Z^{d}$, $\psi_{1},\psi_{2}\in \Aut(X_{\zeta},S,\Z^{d})$ commuting with the substitution map such that $S^{{\bm j}_{i}}\phi=\psi_{i}$, for $i\in \{1,2\}$. Then, $S^{{\bm j}_{2}-{\bm j}_{1}}\psi_{1}$ is equal to $\psi_{2}$. Hence, for any $n>0$ 
	
	$$\begin{array}{cl}
		S^{{\bm j}_{2}-{\bm j}_{1}}\psi_{1}\zeta^{n} & =\zeta^{n}(S^{{\bm j}_{2}-{\bm j}_{1}}\psi_{1})\\
		& = S^{L_{\zeta}^{n}({\bm j}_{2}-{\bm j}_{1})}\zeta^{n}\psi_{1},
	\end{array}$$

\noindent which implies that $(\id-L_{\zeta}^{n})({\bm j}_{2}-{\bm j}_{1})=0$, so ${\bm j}_{2}={\bm j}_{1}$ and then $\psi_{1}=\psi_{2}$.
\end{proof}	

\subsection{Rigidity properties for homomorphisms between substitutive subshifts and applications} 

This subsection is devoted to homomorphisms between substitutive subshifts. We recall that for $M\in GL(d,\Z)$, a map $\phi:(X,T,\Z^{d})\to (Y,T,\Z^{d})$ between two topological dynamical systems is said to be a homomorphism associated with $M$ if for all ${\bm m}\in \Z^{d}$ we have that $\phi\circ S^{{\bm m}}=S^{M{\bm m}}\circ \phi$. We can also define measurable homomorphisms in the measure-theoretic setting. First, we establish a necessary condition for the matrices $M$ with $m\Hom_{M}(X_{\zeta_{1}},X_{\zeta_{2}},S,\Z^{d})$ being non empty, whenever $\zeta_{1}$, $\zeta_{2}$ are two aperiodic primitive constant-shape substitutions with the same expansion matrix and support. Then, we prove an analogue of \cref{MainTheorem} (\cref{NormalizerHostParreau}) establishing that measurable homomorphisms induce continuous ones when the matrix $M$ commutes with some power of the expansion matrix. Finally, we give an explicit bound on the norm of these matrices for the quotient group of a restricted normalizer semigroup with respect to the shift action (\cref{CommutingCentralizerVirtuallyZd}).

\begin{lemma}\label{LemmaNessesaryConditionNormalizer}
	Let $\zeta_{1},\zeta_{2}$ be two aperiodic primitive  constant-shape substitutions having the same expansion matrix $L$ and same support $F$. If $M\in GL(d,\Z)$ is such that $m\Hom_{M}(X_{\zeta_{1}},X_{\zeta_{2}},S,\Z^{d})\neq \emptyset$, then for all $n>0$ there exists $m(n)>0$ such that \begin{equation}\label{heightcondition}\tag{Normalizer Condition}ML^{m(n)}(\mathcal{H}(X_{\zeta_{1}}))\leq L^{n}(\mathcal{H}(X_{\zeta_{2}})).\end{equation}
\end{lemma}

\begin{proof}
	Let $\phi$ be in $m\Hom_{M}(X_{\zeta_{1}},X_{\zeta_{2}},S,\Z^{d})$, and ${\bm x}\in E(X_{\zeta_{2}},\mu_{\zeta_{2}},S,\Z^{d})$. We will prove that $M^{*}{\bm x}\in E(X_{\zeta_{1}},\mu_{\zeta_{1}},S,\Z^{d})$. Indeed, let $f\in L^{2}(X_{\zeta_{2}},\mu_{\zeta_{2}})$ be such that for all ${\bm m}\in \Z^{d}$, $f\circ S^{{\bm m}}=e^{2\pi i \left\langle {\bm x},{\bm m}\right\rangle}\cdot f$, $\mu_{\zeta_{2}}$-a.e. in $X_{\zeta_{2}}$. Then, we have that
	$$(f\circ \phi)\circ S^{{\bm m}}=(f \circ S^{M{\bm m}})\circ \phi=e^{2\pi i \left \langle {\bm x},M{\bm m}\right\rangle}\cdot f\circ \phi=e^{2\pi i \left\langle M^{*}{\bm x},{\bm m}\right\rangle}\cdot f\circ \phi,\ \mu_{\zeta_{1}}\text{-a.e.\ in}\ X_{\zeta_{1}}.$$
	
	By \cref{CharacterizationEigenvalues} and \cref{CharacterizationMinimalPartition}, for any $n>0$, the system $(\Z^{d}/M^{-1}L^{n}(\mathcal{H}(X_{\zeta_{2}})),+,\Z^{d})$ is a finite factor of the odometer system $(\overleftarrow{\Z^{d}}_{(L^{n}(\mathcal{H}(X_{\zeta_{1}})))},+_{{(L^{n}(\mathcal{H}(X_{\zeta_{1}})))}},\Z^{d})$, which implies that the odometer system $(\overleftarrow{\Z^{d}}_{(M^{-1}L^{n}(\mathcal{H}(X_{\zeta_{2}})))},+_{{(M^{-1}L^{n}(\mathcal{H}(X_{\zeta_{1}})))}},\Z^{d})$ is a factor of $(\overleftarrow{\Z^{d}}_{(L^{n}(\mathcal{H}(X_{\zeta_{1}})))},+_{{(L^{n}(\mathcal{H}(X_{\zeta_{1}})))}},\Z^{d})$. By \cref{CharacterizationFactorOdometer}, we conclude that for any $n>0$, there exists $m(n)>0$ such that $L^{m(n)}(\mathcal{H}(X_{\zeta_{1}}))\leq M^{-1}L^{n}(\mathcal{H}(X_{\zeta_{2}})).$
\end{proof}

A consequence of \cref{Coalescence} is that if a homomorphism is associated with a matrix with finite order, then it is an isomorphism.

\begin{lemma}\label{NormalizerCoalescenceFiniteGroup}
	Let $\zeta$ be an aperiodic primitive reduced  constant-shape substitution. If $M\in GL(d,\Z)$ has finite order, then any homomorphism $\phi\in N_{M}(X_{\zeta},S,\Z^{d})$ is invertible.
\end{lemma}

\begin{proof}
	Since $M$ has finite order, there exists $n>0$ such that $M^{n}=\id_{\R^{d}}$. This implies that $\phi^{n}\in \End(X_{\zeta},S,\Z^{d})$. By \cref{Coalescence}, $\phi^{n}$ is invertible, so $\phi$ is also invertible. 
\end{proof}

Now, we will prove an analogue of \cref{MainTheorem} for homomorphisms associated with matrices commuting with a power of the expansion matrix $L$. As mentioned before, a priori this does not cover all the homomorphisms between substitutive subshifts. 

\begin{theorem}\label{NormalizerHostParreau} Let $(X_{\zeta_{1}},S,\Z^{d})$, $(X_{\zeta_{2}},S,\Z^{d})$ be two substitutive subshifts from two aperiodic primitive constant-shape substitutions $\zeta_{1}$, $\zeta_{2}$ from finite alphabets $\A$ and $\B$, with the same support $F_{1}$ and expansion matrix $L$. Let $M\in GL(d,\Z)$ be a matrix commuting with a power of $L$, i.e., there exists $n>0$ such that $ML^{n}=L^{n}M$. If $\zeta_{2}$ is reduced, then for every measurable homomorphism $\phi$, associated with $M$, there exists ${\bm j}\in \Z^{d}$ such that $S^{{\bm j}}\phi$ is equal $\mu_{\zeta_{1}}$-a.e. to a continuous homomorphism $\psi$, associated with $M$, satisfying the following two properties:
	
	\begin{enumerate}[label=(\arabic*),ref=\text{(}\arabic*\text{)}]
		\item $\psi$ is given by a block map of radius $\Vert F_{1}^{\zeta}\Vert\Vert L_{\zeta_{1}}^{-1}\Vert \left(1+\Vert M\Vert\right)\left(2-\Vert L_{\zeta}^{-1}\Vert\right)/\left(1-\Vert L_{\zeta}^{-1}\Vert\right)$.
		\item\label{CommutingNormalizerWithSubstitution} There exist an integer
		$n>0$ and ${\bm q}\in F_{n}^{\zeta}$ such that $S^{{\bm q}}\psi\zeta_{1}^{n}=\zeta_{2}^{n}\psi$.
	\end{enumerate} 
\end{theorem}

\begin{remark}
	Let $\psi\in \Hom_{M}(X_{\zeta_{1}},X_{\zeta_{2}},S,\Z^{d})$ satisfying property (2) of \cref{NormalizerHostParreau}. For any ${\bm m}$ in $\Z^{d}$, we have that $S^{{\bm q}}\psi(\zeta_{1}^{n}(S^{{\bm m}}x))=\zeta_{2}^{n}(\psi(S^{{\bm m}}x))$, and $S^{{\bm q}}\psi(\zeta_{1}^{n}(S^{{\bm m}}x))=S^{{\bm q}+ML_{\zeta_{1}}^{n}{\bm m}}\psi(\zeta_{1}^{n}(x))$, $\zeta_{2}^{n}(\psi(S^{{\bm m}}x))=S^{L_{\zeta_{1}}^{n}M{\bm m}}\zeta_{2}^{n}(\psi(x))$, it follows that, $ML_{\zeta_{1}}^{n}{\bm m}=L_{\zeta_{1}}^{n}M{\bm m}$, i.e., $M$ and $L_{\zeta_{1}}^{n}$ commute. Hence this hypothesis is optimal to obtain property (2). Note that if $L$ is an integer multiple of the identity, then any matrix $M\in GL(d,\Z)$ commutes with $L$.
\end{remark}

The proof of \cref{NormalizerHostParreau} follows the same strategy as the one of \cref{MainTheorem}, except for some small modifications. Since the substitution $\zeta_{1}$ is primitive, we can replace it by some power $\zeta_{1}^{n}$, so we may assume that $M$ commutes with the expansion matrix of $\zeta_{1}$.  We replace the term ${\bm p}_{n}(\phi)$ by the map $\pi_{n}(x)-M^{-1}\pi_{n}(\phi x)\ (\Mod L_{\zeta}^{n}(\Z^{d}))$, with $\pi_{n}(x)$ and $M^{-1}\pi_{n}(\phi x)$ being the representative classes in $F_{n}^{\zeta}$. The commutation assumption implies that, for any $n>0$ the map $M$ defines a bijection in $\Z^{d}/L^{n}(\Z^{d})$, also denoted by $M$, i.e., ${\bm n}={\bm m}\ (\text{mod}\ L^{n}(\Z^{d}))$, if and only if $M{\bm n}=M{\bm m}\ (\text{mod}\ L^{n}(\Z^{d}))$. With this, the map ${\bm p}_{n}(\phi)$ is invariant under the shift action. Since $(X_{\zeta_{1}},\mu_{\zeta_{1}},S,\Z^{d})$ is ergodic, the map ${\bm p}_{n}(\phi)\in F_{n}^{\zeta}$ is a constant map $\mu_{\zeta_{1}}$-a.e. in $X_{\zeta_{1}}$ and the set $S^{M{\bm p}_{n}(\phi)}\phi\zeta_{1}^{n}(X_{\zeta_{1}})$ is included, up to a $\mu_{\zeta_{2}}$-null set, in $\zeta_{2}^{n}(X_{\zeta_{2}})$. We can define the map $\phi_{n}$ for $\mu_{\zeta_{1}}$-a.e. in $X_{\zeta_{1}}$ as the unique point $y\in X_{\zeta_{2}}$ such that $S^{M {\bm p}_{n}(\phi)}\phi\zeta_{1}^{n}(x) = \zeta_{2}^{n}(y)$, where $M{\bm p}_{n}(\phi)$ is the representative element in $F_{n}^{\zeta}$. It is straightforward to check that $\phi_{n}\circ S^{\bm{n}}=S^{M{\bm n}}\circ \phi_{n}$ for all ${\bm n}\in \Z^{d}$, so $\phi_{n}$ is in $m\Hom_{M}(X_{\zeta_{1}},X_{\zeta_{2}},S,\Z^{d})$. The sequences ${\bm p}_{n}(\phi)$ and $(\phi_{n})$ satisfy the same recurrences given in \cref{SubsectionMeasurableFactors}: 
$${\bm p}_{n+1}(\phi)={\bm p}_{n}(\phi)+L_{\zeta}^{n}{\bm p}_{1}(\phi_{n}),\quad (\phi_{n})_{1}=\phi_{n+1}.$$

As in \cref{MainTheorem} we need the following adaptations of \cref{EqualityofNormalizers} and \cref{convergenceofslidingblockcodes} for homomorphisms. The proof are the same, so we omit them.

\begin{lemma}\label{EqualityofNormalizersforNormalizerHostPaeerau}
	If $\phi,\psi\in mN_{M}(X_{\zeta_{1}},X_{\zeta_{2}},S,\Z^{d})$ are such that $d(\phi,\psi)$ is smaller than $\eta/C(R)$, then $\phi,\psi$ are equal $\mu_{\zeta_{1}}$-a.e in $X_{\zeta_{1}}$.
\end{lemma}

\begin{lemma}\label{convergenceofslidingblockcodesforNormalizerHostPaeerau} 
	Let $\phi\in mN_{M}(X_{\zeta_{1}},X_{\zeta_{2}},S,\Z^{d})$. Then there exists a sequence $(\psi_{n})$ of homomorphisms associated with $M$ of radius $\Vert F_{1}^{\zeta}\Vert\Vert L_{\zeta_{1}}^{-1}\Vert \left(1+\Vert M\Vert\right)\left(2-\Vert L_{\zeta}^{-1}\Vert\right)/\left(1-\Vert L_{\zeta}^{-1}\Vert\right)$ such that $d(\phi_{n},\psi_{n})\to 0$.
\end{lemma}

To finish the proof of \cref{NormalizerHostParreau}, we proceed exactly as in the proof of \cref{MainTheorem}

\begin{proof}[Proof of \cref{NormalizerHostParreau}]
	For fixed alphabets $\A$ and $\B$, there exists a finite number of homomorphisms associated with $M$ of radius $\Vert F_{1}^{\zeta}\Vert\Vert L_{\zeta_{1}}^{-1}\Vert \left(1+\Vert M\Vert\right)\left(2-\Vert L_{\zeta}^{-1}\Vert\right)/\left(1-\Vert L_{\zeta}^{-1}\Vert\right)$. By \cref{convergenceofslidingblockcodesforNormalizerHostPaeerau}, there exist two different integers $m,k\geq 0$ such that $d(\phi_{m},\phi_{m+k})<\eta/C(R)$, so by Lemma \ref{EqualityofNormalizersforNormalizerHostPaeerau}, we have that $\phi_{m}=\phi_{m+k}$, $\mu_{\zeta_{1}}$-a.e.. 
	
	Let $n\geq m$ be a multiple of $k$. We have $(\phi_{n})_{k}=\phi_{n+k}=(\phi_{m+k})_{n-m}=(\phi_{m})_{n-m}=\phi_{n}$, $\mu_{\zeta_{1}}$-a.e. This implies for all $r\in \NN$, $\phi_{n}$ is equal to $(\phi_{n})_{rk}$, $\mu_{\zeta_{1}}$-a.e., and then $\phi_{n}$ is equal to a homomorphism associated with $M$ of radius $\Vert F_{1}^{\zeta}\Vert\Vert L_{\zeta_{1}}^{-1}\Vert \left(1+\Vert M\Vert\right)\left(2-\Vert L_{\zeta}^{-1}\Vert\right)/\left(1-\Vert L_{\zeta}^{-1}\Vert\right)$, $\mu_{\zeta_{1}}$-a.e. in $X_{\zeta_{1}}$. Since $\phi_{n}$ is equal to $\phi_{2n}$ $\mu_{\zeta_{1}}$-a.e, we denote $\psi=\phi_{n}$ and ${\bm p}={\bm p}_{n}(\psi)$. By definition of ${\bm p}$, we have $S^{M{\bm p}}\psi\zeta_{1}^{n}=\zeta_{2}^{n}\psi$. 
	
	Set ${\bm j}=M({\bm p}_{n}(\phi)-{\bm p})$, then
	$$S^{{\bm j}}\phi\zeta_{1}^{n}=S^{M({\bm p}_{n}(\phi)-{\bm p})}\phi\zeta_{1}^{n}=S^{-M{\bm p}}\zeta_{2}^{n}\psi=\psi\zeta_{1}^{n},\quad \mu_{\zeta_{1}}-\text{a.e},$$
	this implies that $S^{{\bm j}}\phi$ and $\psi$ coincides in $\zeta_{1}^{n}(X_{\zeta_{1}})$ $\mu_{\zeta_{1}}$-a.e. The ergodicity of $(X_{\zeta_{1}},\mu_{\zeta_{1}},S,\Z^{d})$ lets us conclude that $\phi$ and $\psi$ are equal $\mu_{\zeta_{1}}$-a.e. 
\end{proof}

In the case $\zeta_{1}=\zeta_{2}$, we can consider a restricted normalizer group: the group of isomorphisms associated with matrices commuting with some power of the expansion matrix $L_{\zeta_{1}}$: $$NC(X_{\zeta_{1}},S,\Z^{d})=\bigcup\limits_{\substack{M\in GL(d,\Z)\\ ML_{\zeta_{1}}^{n}=L_{\zeta_{1}}^{n}M,\ \text{for some\ } n}}(N_{M}(X,T,\Z^{d})\cap \Homeo(X)),$$

This set is a group under composition and $\left \langle S \right\rangle$, $\Aut(X_{\zeta_{1}},S,\Z^{d})$ are normal subgroups of $NC(X_{\zeta_{1}},S,\Z^{d})$. We obtain a similar result on this restricted normalizer group as for the automorphism group (\cref{AutomoprhismVirtuallyZd}).

\begin{proposition}\label{CommutingCentralizerVirtuallyZd}
	Let $(X_{\zeta},S,\Z^{d})$ be a subshift from a reduced aperiodic primitive  constant-shape substitution $\zeta$ from a finite alphabet. If the set of matrices $M\in \vec{N}(X_{\zeta},S,\Z^{d})$ commuting with a power of the expansion matrix $L_{\zeta}$ is finite, then the quotient group $NC(X_{\zeta},S,\Z^{d})/\left\langle S \right\rangle$ is finite. A bound for $|NC(X_{\zeta},S,\Z^{d})/\left\langle S\right\rangle|$ is given by an explicit formula depending only on $d$, $|\A|$, $\Vert L_{\zeta}^{-1}\Vert$, $\Vert F_{1}^{\zeta}\Vert$, and $\sup\limits_{N_{M}(X_{\zeta},S,\Z^{d})\neq \emptyset}\Vert M\Vert$. 
\end{proposition}

\begin{proof}
	Let $\psi\in NC(X_{\zeta},S,\Z^{d})$, satisfying Property \ref{CommutingNormalizerWithSubstitution} of \cref{NormalizerHostParreau}. Following the proof of \cref{Coalescence}, $\psi$ acts as a permutation of the $\zeta$-invariant orbits. Since the set of matrices $M\in \vec{N}(X_{\zeta},S,\Z^{d})$ commuting with a power of $L_{\zeta}$ is finite, there exists $n>0$ such that $\psi^{n}$ is an automorphism of $X_{\zeta}$. By \cref{AutomoprhismVirtuallyZd}, we have that $\psi^{n}$ has finite order, which implies that $\psi$ has finite order. The bound for $|NC(X_{\zeta},S,\Z^{d})/\left\langle S\right\rangle|$ is given by \cref{NormalizerHostParreau}.
\end{proof}

\begin{remark}
	Note that, if $L_{\zeta}$ is a diagonal matrix and all of its entries are different, the set of matrices $M\in GL(d,\Z)$ commuting with a power of $L_{\zeta}$ is finite.
\end{remark}

\section{Precisions on bijective  constant-shape substitutions}\label{SectionBijectiveSubstitutions}

Bijective substitutions are of great interest because of their mixed dynamic spectrum. They are never almost 1-to-1 extensions of their maximal equicontinuous factor. Bijective substitutions were studied before in \cite{frank2005multidimensional} for block substitutions, where it was proved that the substitutive subshift is measurable-theoretic isomorphic to a skew product of one-dimensional odometers. Also, \cite{bustos2022admissible} studied the normalizer group of bijective block substitutions. We extend the study by describing the normalizer group for general constant-shape substitutions. To do this, we relate the linear representation group with different types of supports of the substitution and non-diagonal expansion matrices.

We start with a characterization of the automorphism group of substitutive subshifts from aperiodic bijective primitive constant-shape substitutions. Then, we describe the nondeterministic directions of a substitutive subshift, by the supporting hyperplanes to $\conv(F_{n}^{\zeta})$ (\cref{NonExpansiveHalfspacesSupportingConvexHull}). We then emphasize when the convex hull of the digit tile $\conv(T_{\zeta})$ is a polytope, because this implies strong geometrical restrictions on the supporting hyperplanes. Conversely, we provide a checkable combinatorial condition to ensure a vector to be nondeterministic (\cref{CorollaryNonExpansiveHalfspacesBijectiveSubstitutions}). Finally, we deduce dynamical consequences for $(X_{\zeta},S,\Z^{d})$ on the normalizer group. For instance, the normalizer group is virtually generated by the shift action (\cref{FinalTheoremNormalizerGroupPolytopeCase}) and we provide restrictions on the linear representation group (\cref{PropositionAboutMatricesPolytopeCase}). It follows all the results of the previous sections may apply in this case.

\subsection{The automorphism group of substitutive subshifts from bijective constant-shape substitutions.} Since bijective substitutions are reduced, \cref{AutomoprhismVirtuallyZd} implies that the automorphism group of the substitutive subshift $(X_{\zeta},S,\Z^{d})$ is virtually $\Z^{d}$. In fact, we have a more rigid result in the bijective case as shown in the following proposition.

\begin{proposition}
	Let $\zeta_{1},\zeta_{2}$ be two aperiodic bijective primitive constant-shape substitutions with the same expansion matrix $L$ and support $F_{1}$. Then, any factor map $\psi:X_{\zeta_{1}}\to X_{\zeta_{2}}$ satisfying Property \ref{CommutationPropertySlidingBlickCodeSubstitution} in \cref{MainTheorem} is induced by a letter-to-letter map. In particular, the automorphism group $\Aut(X_{\zeta},S,\Z^{d})$ of a substitutive subshift given by an aperiodic bijective primitive constant-shape substitution is isomorphic to the direct product of $\Z^{d}$, generated by the shift action, with a finite group given by a permutation of the letters in the alphabet $\A$.
\end{proposition}

\begin{proof}
	Let $\psi\in \Fac(X_{\zeta_{1}},X_{\zeta_{2}},S,\Z^{d})$ satisfying Property \ref{CommutationPropertySlidingBlickCodeSubstitution} in \cref{MainTheorem}, i.e., there exists ${\bm p}\in F_{1}$ such that $S^{{\bm p}}\psi \zeta_{1}=\zeta_{2}\psi$. Suppose that ${\bm p}\neq {\bm 0}$. Let $n>0$ be large enough such that the set $F_{n}^{\circ C}=\{{\bm f}\in F_{n}\colon {\bm f}+C\subseteq F_{n}\}$ is nonempty, where $C$ is the set defined in \cref{convergenceofslidingblockcodes}. Then, for any $x\in X_{\zeta_{1}}$, the coordinate $x_{{\bm 0}}$ determines the pattern $\zeta_{1}^{n}(x)|_{F_{n}}$. Hence the pattern $(S^{{\bm p}_{n}}\psi \zeta_{1}^{n})|_{{\bm p}_{n}+(F_{n})^{\circ C}}$ is also completely determined by the letter in $x_{{\bm 0}}$, where ${\bm p}_{n}=\sum\limits_{i=0}^{n-1}L^{i}({\bm p})$. Set ${\bm m}\in \Z^{d}$ such that $({\bm p}_{n}+(F_{n})^{\circ C})\cap (L^{n}({\bm m})+F_{n})\neq \emptyset$. Since $S^{{\bm p}_{n}}\psi \zeta_{1}^{n}$ is equal to $\zeta_{2}^{n}\psi$ and $\zeta_{2}$ is bijective, the coordinate $x_{{\bm 0}}$ determines $\psi(x)_{{\bm m}}$. This implies that $S^{-{\bm m}}\psi$ is a factor map induced by a letter-to-letter map. Set $\phi=S^{-{\bm m}}\psi$, we get that
	$$S^{{\bm p}_{n}+{\bm m}-L^{n}({\bm m})} \phi\zeta_{1}^{n}=\zeta_{2}^{n}\phi,$$
	
	\noindent and, by bijectivity, the coordinate $x_{{\bm 0}}$ determines two coordinates of $\psi$, unless for any $n\in \NN$ large enough ${\bm p}_{n}+{\bm m}$ is in $L_{\zeta}^{n}(\Z^{d})$, i.e., for any $n$ large enough, there exists ${\bm r}_{n}\in \Z^{d}$ such that ${\bm p}_{n}+{\bm m}=L^{n}({\bm r}_{n})$. Note that ${\bm p}+{\bm r}_{n}=L({\bm r}_{n+1})$, which implies that
	$$\Vert {\bm r}_{n+1}\Vert \leq \Vert L^{-1}\Vert (\Vert {\bm r}_{n}\Vert+\Vert {\bm p}\Vert),$$
	
	\noindent so $({\bm r}_{n})_{n>0}$ is a bounded sequence. Hence, there exist $n>0$ and $N>1$ such that ${\bm r}_{n+N}={\bm r}_{n}$, which implies that ${\bm p}_{N-1}\in (L^{N}-\id)(\Z^{d})$. This is not possible by \cref{RemarkMainTheorem}.
	
	If ${\bm p}_{n}+{\bm m}\notin L_{\zeta}^{n}(\Z^{d})$, then $x_{{\bm 0}}$ determines two coordinates ${\bm n}_{1}$, ${\bm n}_{2}$ of $\psi(x)$, so determines the coordinates ${\bm 0}$ and ${\bm n}_{2}-{\bm n}_{1}$ of $\psi_{1}(x)=S^{-{\bm n}_{1}}\psi(x)$. Since $\psi_{1}$ is also induced via a letter-to-letter map. Note that, the map $\Psi_{1}:\A\to\B$ inducing $\psi_{1}$ is bijective. If not, there are two fixed points $x,y$ with $\Psi_{1}(x_{{\bm 0}})=\Psi_{2}(y_{{\bm 0}})$ and $x_{{\bm 0}}\neq y_{{\bm 0}}$ generate two points with the same image, which is a contradiction. It follows that $x_{{\bm 0}}$ determines $x_{{\bm n}_{2}-{\bm n}_{1}}$, and then $x_{k({\bm n}_{2}-{\bm n}_{1})}$ for all $k\in \Z$, so $x$ has a nontrivial period, which is a contradiction. Finally, we conclude by \cref{TrivialFactorsImpliesDirectProduct}.
	\end{proof}

\subsection{Nondeterministic directions of substitutive subshifts from extremally permutative constant-shape substitutions}

In this section, we give a characterization of the nondeterministic directions (defined in \cref{SectionNonDeterminsiticDirections}) of a substitutive subshift $(X_{\zeta},S,\Z^{d})$, in the case where $\zeta$ is \emph{extremally permutative}. A starting remark is that, for each $n>0$, the set of directions $\SS^{d-1}$ is stratified by the opposite normal fan $\N(\conv(F_{n}^{\zeta}))$ (see \cref{Subsectionconvexgeometry}). Our description of the nondeterministic directions is given in terms of union of these fans. 

Following the notion of left and right permutative morphisms (see \cite{berthe2019recognizability}), we say that a constant-shape substitution $\zeta$ is \emph{extremally permutative} if the restriction $p_{{\bm f}}$ of $\zeta$ in ${\bm f}$ is bijective for all ${\bm f}\in \Ext(\conv(F_{1}^{\zeta}))$. Since $\Ext(\conv(A+B))\subseteq \Ext(\conv(A))+\Ext(\conv(B)))$, a substitution is extremally permutative if and only if for any $n>0$ and ${\bm f}\in \Ext(\conv(F_{n}^{\zeta})$ the restriction $p_{{\bm f}}$ of $\zeta^{n}$ in ${\bm f}$ is bijective. Since $(F_{n}^{\zeta})_{n>0}$ is a F\o lner sequence, there exists $n>0$ such that $\conv(F_{n}^{\zeta})$ is a nondegenerate polytope, so up to considering a power of $\zeta$, we may assume that $\conv(F_{1}^{\zeta})$ is a nondegenerate polytope. Using the recognizability property of substitutions and some basic results in convex geometry we prove the following result.

\begin{theorem}\label{NonExpansiveHalfspacesSupportingConvexHull}
	Let $\zeta$ be an aperiodic extremally permutative primitive constant-shape substitution. Then, the set of nondeterministic directions $\ND(X_{\zeta},S,\Z^{d})$ of its substitutive subshift $(X_{\zeta},S,\Z^{d})$ is the intersection of $\SS^{d-1}$ with a nonempty union of limits of nested sequences of opposite normal cones of the form $\hat{N}_{{\bm G}_{n}}(\conv(F_{n}^{\zeta}))$, where ${\bm G}_{n}$ is a face of $\conv(F_{n}^{\zeta})$, for some integer $n>0$.
\end{theorem}

This theorem gives topological constraints on the set of nondeterministic directions. Actually, we will see that the convex hull of any digit tile is a polytope when $L_{\zeta}=\lambda \id_{\R^{d}}$, i.e. it has a finite number of extreme points (\cref{NecessaryAndSufficientConditionForPolytope}). In this case by \cref{NonExpansiveHalfspacesSupportingConvexHull}, the set of nondeterministic directions $\ND(X_{\zeta},S,\Z^{d})$ is a finite union of closed intervals (eventually degenerated). More explicitly, in the two-dimensional case, we obtain the following corollary, showing in particular that it cannot be a Cantor set.

\begin{corollary}
	In the two dimensional case, under the hypothesis of \cref{NonExpansiveHalfspacesSupportingConvexHull}, either the set $\ND(X_{\zeta},S,\Z^{2})$ has nonempty interior, either it has at most 2 accumulation points.
\end{corollary}

\begin{proof}
	Assume that the set of nondeterministic directions $\ND(X_{\zeta},S,\Z^{2})$ has empty interior. By \cref{NonExpansiveHalfspacesSupportingConvexHull}, the elements of $\ND(X_{\zeta},S,\Z^{2})$ are limits of normal vectors to edges of $\conv(F_{n}^{\zeta})$ for some $n>0$. In \cite{strichartz1999geometry} it was proved that such vectors are normalized vectors of the form $(L_{\zeta}^{*})^{-k}{\bm u}_{k}$, with ${\bm u}_{k}\in \SS^{1}$ being a normal vector to an edge of $\conv(F_{1}^{\zeta})$ for some $k>0$. Hence, their accumulation points are accumulation points of orbits of the projective action $L_{\zeta}^{*}$ on the circle $\SS^{1}$. A standard analysis of this action (you can check Theorem 3 in \cite{leader1991limit}) provides that the cardinality of the accumulation points is at most 2, when a power of one of the $L_{\zeta}$-eigenvalues is a real number. Otherwise, the projective orbits of $L_{\zeta}^{*}$ are dense in the circle. Since $\ND(X_{\zeta},S,\Z^{2})$ is closed, it is the whole circle, which is a contradiction.
\end{proof}

\begin{proof}[Proof of \cref{NonExpansiveHalfspacesSupportingConvexHull}]
	Let ${\bm v}$ be a nondeterministic direction for $(X_{\zeta},S,\Z^{d})$, and $x_{1}\neq x_{2}\in X_{\zeta}$ such that $x_{1}|_{H_{\bm v}}=x_{2}|_{H_{\bm v}}$. Consider the set $D=\{{\bm n}\in \Z^{d}\colon x_{1}({\bm n})\neq x_{2}({\bm n})\}$. Since $D\subseteq \R^{d}\setminus H_{{\bm v}}$, its convex hull is also contained in $\R^{d}\setminus H_{{\bm v}}$. We have 2 possibilities:
	
	\begin{enumerate}
		\item The convex hull $\conv(D)$ has at least one extreme point. By \cref{PropertyConvexHull}, all the extreme points of $\conv(D)$ belong to $D$. Now, if ${\bm v}\in \Z^{d}$ is an extremal ray of $\conv(D)$, then for any extremal point ${\bm n}\in D$, the map $\dist({\bm n}+t{\bm v},H_{{\bm v}})$ must be increasing (if not, there exists $t^{*}>0$ with ${\bm n}+t^{*}{\bm v}\in H_{{\bm v}}$, which is a contradiction). Hence, the distance map to $H$ restricted to $\conv(D)$ is minimized in the extreme points of $\conv(D)$. Since the extreme points of $\conv(D)$ are in $D$, we can use the shift action in $x_{1}$, $x_{2}$ and assume that $x_{1}({\bm 0})\neq x_{2}({\bm 0})$.
		
		\item If $\conv(D)$ does not have extreme points, then contains a line. In this case, the hyperplane $\partial H_{{\bm v}}$ must be parallel to this line. Using similar arguments, we can assume that $x_{1}({\bm 0})\neq x_{2}({\bm 0})$. 
	\end{enumerate}
	
	So we can assume that ${\bm 0}$ is in a face ${\bm F}_{0}$ of smallest dimension of $\overline{\conv(D)}$ and then ${\bm v}\in \hat{N}_{{\bm F}_{0}}(\overline{\conv(D)})$. In fact, any element in $\hat{N}_{{\bm F}_{0}}(\overline{\conv(D)})\cap \SS^{d-1}$ is a nondeterministic direction for $(X_{\zeta},S,\Z^{d})$.

	Now, for any $k>0$ consider $R^{(k)}>0$ as the recognizability radius for $\zeta^{k}$ given by \cref{RecognizabilityFactors} and $R=4R^{(k)}$. Since $x_{1}$ and $x_{2}$ coincide in an arbitrarily large ball, they have the same image under the maximal equicontinuous factor, hence $\pi_{n}(x_{1})=\pi_{n}(x_{2}) \in F_{n}^{\zeta}$ for any $n>0$. By \cref{AperiodicUniformylBounded}, there exist $n>0$ and two words $\texttt{w}_{1}^{(n)}, \texttt{w}_{2}^{(n)}\in \mathcal{L}_{\overline{K}_{\zeta}}(X_{\zeta})$ such that $x_{i}|_{B({\bm 0},R)\cap \Z^{d}}\sqsubseteq \zeta^{n}(\texttt{w}_{i}^{(n)})$, for $i\in \{1,2\}$. By the Pigeonhole Principle, there exist an infinite set $E\subseteq \NN$, two patterns $\texttt{w}_{1}, \texttt{w}_{2}\in \mathcal{L}_{\overline{K}_{\zeta}}(X_{\zeta})$ and ${\bm k}_{1}, {\bm k}_{2}\in \overline{K}_{\zeta}$ such that for all $n\in E$, $x_{i}|_{B({\bm 0},R)\cap \Z^{d}}=\zeta^{n}(\texttt{w}_{i})|_{L_{\zeta}^{n}({\bm k}_{i})+\pi_{n}(x_{1})+B({\bm 0},R)\cap \Z^{d}}, i\in \{1,2\}$. The recognizability property implies the origin is in the boundary of $\conv(F_{k}^{\zeta}-\pi_{k}(x_{1}))$. Letting $k$ to infinity, for all $n>0$, the origin is in the boundary of $\conv(F_{n}^{\zeta}-\pi_{n}(x_{1}))$. If ${\bm G}_{n}$ is the face of smallest dimension containing $\pi_{n}(x_{1})$ in $\conv(F_{n}^{\zeta})$, then $\hat{N}_{{\bm F}_{0}}(\overline{\conv(D)})$ is included in $\hat{N}_{{\bm G}_{n}}(\conv(F_{n}^{\zeta}))$, for all $n>0$.
	
 We will show now the converse. We separate the proof in two cases.
	
	Suppose first that $\conv(D)$ is closed. Let ${\bm F}_{1}$ be a face of $\conv(D)$ containing ${\bm 0}$ of codimension $1$. By \cref{PropertyConvexHull}, we have that ${\bm F}_{1}=\conv({\bm F}_{1}\cap D)$. Fix ${\bm t}\in {\bm F}_{1}\cap D$ different from ${\bm 0}$. Since $F_{n}^{\zeta}$ is a fundamental domain of $L_{\zeta}^{n}(\Z^{d})$, there are ${\bm h}_{n}\in F_{n}^{\zeta}$ and ${\bm z}_{n}({\bm t})\in \Z^{d}$ such that ${\bm t}={\bm h}_{n}({\bm t})-\pi_{n}(x_{1})+L_{\zeta}^{n}({\bm z}_{n}({\bm t}))$ and ${\bm t}$, being in the boundary of $\conv(D)$, lies in the boundary of $\conv(F_{n}^{\zeta}-\pi_{n}(x_{1})+L_{\zeta}^{n}({\bm z}_{n}({\bm t})))$. Since this last set is a translated of $\conv(F_{n}^{\zeta}-\pi_{n}(x_{1}))$ and are both subsets of $\conv(D)$, a basic geometrical argument ensures that for any $n>0$, ${\bm h}_{n}({\bm t})$ and $\pi_{n}(x_{1})$ are in the same face of $\conv(F_{n}^{\zeta})$. The same arguments imply that ${\bm h}_{n}({\bm t}_{1})$, ${\bm h}_{n}({\bm t}_{2})$ are in the same face for any ${\bm t}_{1},{\bm t}_{2}\in {\bm F}_{1}\cap D$. Furthermore, ${\bm h}_{n}({\bm t})$ and $\pi_{n}(x_{1})$ are different for any $n\in E$ large enough. Indeed, assume the converse, taking $R>\Vert t\Vert$, we have that $x_{1}|_{\bm t}=\zeta^{n}(\texttt{w}_{i})_{L_{\zeta}^{n}({\bm k}_{3})+\pi_{n}(x_{1})}$, for some ${\bm k}_{3}\in \overline{K}_{\zeta}$ for infinitely many $n\in E$. Since
	$$\Vert {\bm t} \Vert = \Vert  L_{\zeta}^{n}({\bm k}_{1}-{\bm k}_{3})\Vert,\ \text{for infinitely many}\ n\in E,$$
	
	\noindent we have that necessarily ${\bm k}_{1}={\bm k}_{3}$, which is a contradiction. 
	
	\noindent Consider the face ${\bm H}_{n}$ of $\conv(F_{n}^{\zeta})$ of smallest dimension generated by $\{{\bm h}_{n}({\bm t})\}_{{\bm t}\in {\bm F}_{1}\cap D}$. Notice that $\hat{N}_{{\bm F}_{1}}(\conv(D))\subseteq \bigcap\limits_{n>0} \hat{N}_{{\bm H}_{n}}(\conv(F_{n}^{\zeta}))$. We will prove that $\hat{N}_{{\bm F}_{1}}(\conv(D))= \bigcap\limits_{n>0} \hat{N}_{{\bm H}_{n}}(\conv(F_{n}^{\zeta}))$. By construction of $x_{1}$, $x_{2}$, the set $\{{\bm z}_{n}({\bm t})\}_{{\bm t}\in {\bm F}_{1}\cap D}$ is bounded (for all $n$ large enough it belong to $\overline{K}_{\zeta}-\overline{K}_{\zeta}$), so there exists ${\bm t}\in {\bm F}_{1}$ and $\varepsilon>0$ small enough such that for all ${\bm t}'\in B({\bm t},\varepsilon)\cap {\bm F}_{1}$ we have that ${\bm z}_{n}({\bm t}')={\bm z}_{n}({\bm t})$ for all $n\in E$ large enough. Hence ${\bm H}_{n}$ is a face of codimension 1. An argument of dimensions ensures that $\hat{N}_{{\bm F}_{1}}(\conv(D))=\bigcap\limits_{n\in E}\hat{N}_{{\bm H}_{n}}(\conv(F_{n})^{\zeta})$.
	
	Suppose now that $\conv(D)$ is not closed. Let ${\bm F}_{1}$ be a face of $\overline{\conv(D)}$ of codimension 1 containing ${\bm 0}$ and ${\bm w}\in \hat{N}_{{\bm F}_{1}}(\overline{\conv(D)})$. We will find a sequence of faces ${\bm H}_{n}$ of $\conv(F_{n}^{\zeta})$ such that $\hat{N}_{{\bm F}_{1}}(\overline{\conv(D)})=\bigcap\limits_{n>0}\hat{N}_{{\bm H}_{n}}(\conv(F_{n}^{\zeta}))$. By definition, we have that
	\begin{equation}\label{NormalConeFaceCodim1}
		\forall {\bm t}\in {\bm F}_{1}, \left\langle {\bm w},{\bm t}\right\rangle=\inf\limits_{{\bm n}\in D}\left\langle {\bm w},{\bm n}\right\rangle =0.
	\end{equation}
	
	Set ${\bm t}\in {\bm F}_{1}$ and consider a sequence $({\bm t}^{m})_{m>0}\subseteq \conv(D)$ converging to ${\bm t}$. By Caratheodory's theorem for any $m>0$, we can write ${\bm t}^{m}=\sum\limits_{i=0}^{d}t_{i}^{m}{\bm f}_{i}^{m}$, with $t_{i}\geq 0$,  $\sum\limits_{i=0}^{d}t_{i}^{m}=1$ and ${\bm f}_{i}^{m}\in D$. Then, \eqref{NormalConeFaceCodim1} implies that for all $i$, $t_{i}^{m}\left\langle {\bm w},{\bm f}_{i}^{m}\right\rangle \xrightarrow[m\to \infty]\ 0$. The only difficulty to get the result, concern the indices $i$ such that $\liminf\limits_{m\to \infty} t_{i}^{m}>0$. For such $i$, we have that $\left\langle {\bm w},{\bm f}_{i}^{m}\right\rangle \xrightarrow[m\to \infty]\ 0$. Using the recognizability property, for all $n\in E$ we write ${\bm f}_{i}^{m}={\bm h}(m,i,n)-\pi_{n}(x_{1})+L_{\zeta}^{n}({\bm z}(m,i,n))$, with ${\bm h}(m,i,n)\in F_{n}^{\zeta}$ and ${\bm z}(m,i,n)\in \Z^{d}$. Since $\left\langle {\bm w}, {\bm h}(m,i,n)-\pi_{n}(x_{1})\right\rangle\geq 0$ and $\left\langle {\bm w}, L_{\zeta}^{n}({\bm z}(m,i,n))\right\rangle \geq 0$, we have that for all $n>0$
	$$\left\langle {\bm w}, {\bm h}(m,i,n)-\pi_{n}(x_{1})\right\rangle\xrightarrow[m\to \infty]\ 0\quad \wedge\quad \left\langle {\bm w}, L_{\zeta}^{n}({\bm z}(m,i,n))\right\rangle\xrightarrow[m\to \infty]\ 0.$$
	
	Since $F_{n}^{\zeta}$ is finite, we conclude that for all $n\in E$, there exists $m(n)$ such that for all $m\geq m(n)$,
	\begin{equation}\label{EquationCones}
		\left\langle {\bm w}, {\bm h}(m,i,n)\right\rangle=\left\langle {\bm w},\pi_{n}(x_{1})\right\rangle =  0.
	\end{equation}
	
	The same argument as the former one when $\conv(D)$ is closed, gives that ${\bm h}(m,i,n)\neq \pi_{n}(x_{1})$ and if $i\neq j$, ${\bm h}(m,i,n)\neq {\bm h}(m,j,n)$ for all $n$ large enough.
	
	Now, for any $n>0$, we define ${\bm H}_{n}$ as the face of $\conv(F_{n}^{\zeta})$ of  smallest dimension containing $\pi_{n}(x_{1})$ and $\left\{{\bm h}(m,i,n)\colon {\bm t}\in {\bm F}_{1}, 0\leq i\leq d, m\geq m(n) \ \text{with}\ \liminf\limits_{m\to \infty} t_{i}^{m}>0\right\}$. In particular, \eqref{EquationCones} shows that $\hat{N}_{{\bm F}_{1}}(\overline{\conv(D)}) \subseteq \bigcap\limits_{n>0} \hat{N}_{{\bm H}_{n}}(\conv(F_{n}^{\zeta}))$. We claim $\bigcap\limits_{n>0} \hat{N}_{{\bm H}_{n}}(\conv(F_{n}^{\zeta}))=\hat{N}_{{\bm F}_{1}}(\overline{\conv(D)})$. First,  note that taking subsequences if its necessary, for all $n\in E$ we get the following limits
	$$\lim\limits_{m\to \infty}\sum\limits_{i=1}^{d}t_{i}^{m}{\bm h}(m,i,n)={\bm h}_{n}({\bm t})\quad \wedge\quad \lim\limits_{m\to \infty}\sum\limits_{i=1}^{d}t_{i}^{m}{\bm z}(m,i,n)={\bm z}_{n}({\bm t}).$$
	
	Hence, for all $n\in E$, ${\bm h}_{n}(t)$ is in ${\bm H}_{n}$. Also ${\bm z}_{n}({\bm t})$ is in $\conv(\overline{K}_{\zeta}-\overline{K}_{\zeta})$ for all $n\in E$ large enough. A geometric argument shows that there exists ${\bm t}\in {\bm F}_{1}$ and $\varepsilon>0$ small enough such that for all ${\bm t}'\in {\bm F}_{1}\cap B({\bm t},\varepsilon)$, ${\bm z}_{n}({\bm t}')={\bm z}_{n}({\bm t})$. So ${\bm H}_{n}$ is a face of codimension 1 for all $n\in E$ large enough. We then conclude that $\bigcap\limits_{n\in E}{\hat N}_{{\bm H}_{n}}(\conv(F_{n}^{\zeta}))=\hat{N}_{{\bm F}_{1}}(\overline{\conv(D)})$. 
	
	Thus, the extremal rays of $\hat{N}_{{\bm F}_{0}}(\overline{\conv(D)})$ are equal to sets of the form $\bigcap\limits_{n>0}\hat{N}_{{\bm H}_{n}}(\conv(F_{n}^{\zeta}))$, with ${\bm H}_{n}$ being faces, eventually, of codimension 1 of $\conv(F_{n}^{\zeta})$ containing $\pi_{n}(x_{1})$.
\end{proof}

Then, to determine the nondeterministic directions for $(X_{\zeta},S,\Z^{d})$ we need to study the supporting hyperplanes to $\conv(F_{n}^{\zeta})$. To do this, we focus on the convex hull of the digit tile of the substitution. In general, this convex hull is not a polytope, i.e., it can have infinitely many extreme points, even if the expansion matrix is diagonal, as we see in \cref{ExampleNonPolytopeDigitSet}:

\begin{example}[A digit tile, with a nonpolytope convex hull]\label{ExampleNonPolytopeDigitSet}
	Consider $L=\left(\begin{array}{cc}
		2 & 0 \\ 0 & 3
	\end{array}\right)$ and $F_{1}=\left\{(0,0),(0,1),(0,2),(1,0),(1,2),(1,-2)\right\}$.
	
	\begin{center}
		\begin{figure}[H]
			\begin{tabular}{cc}
			\begin{tikzpicture}[scale=0.7]
				\draw[help lines, color=gray!30, dashed] (-0.9,-1.9) grid (1.9,2.9);
				\draw[->,thick] (-1,0)--(2,0) node[right]{$x$};
				\draw[->,thick] (0,-1)--(0,2) node[above]{$y$};
				
				\coordinate (A1) at (0,0);
				\coordinate (A2) at (0,1);
				\coordinate (A3) at (0,2);
				\coordinate (A4) at (1,0);
				\coordinate (A5) at (1,2);
				\coordinate (A6) at (1,-2);
				
				\draw (A1) -- (A3) -- (A5) -- (A6) -- (A1);         
				
				\fill [blue!20, fill opacity = .5]   (A1)--(A3)--(A5)--(A6)--(A1)--cycle;
				
				\foreach \v in {A1,A2,A3,A4,A5,A6}  \draw[fill=gray] (\v) circle (2pt);		
			\end{tikzpicture} & \includegraphics[scale=0.2]{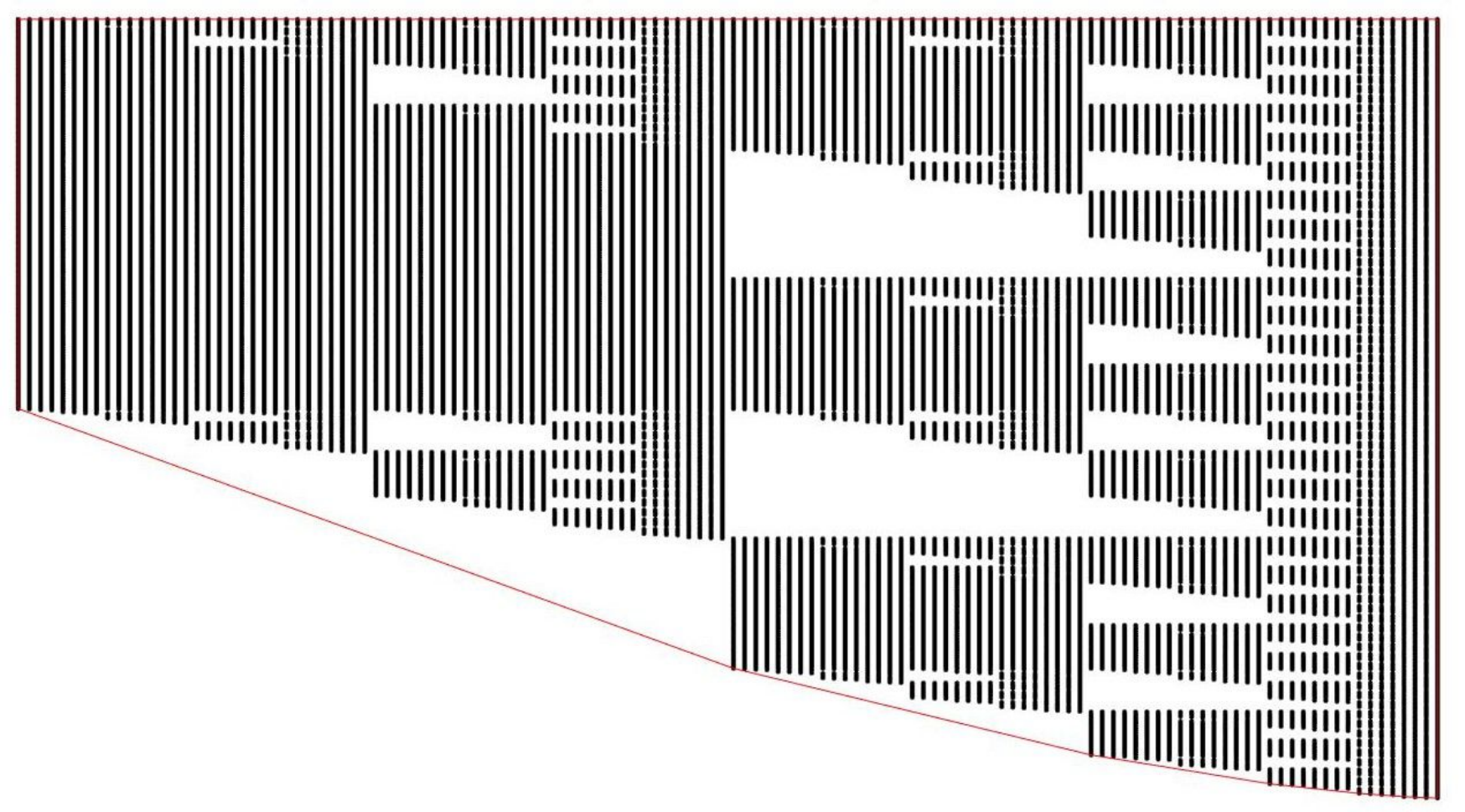}
		\end{tabular}
			\caption{The fundamental domain and an approximation of the digit tile of \cref{ExampleNonPolytopeDigitSet}.}
		\end{figure}
	\end{center}
	
	A direct computation shows that, for any $n>0$, the set of extreme points of $\conv(F_{n})$ is the set $\{(0,0),(0,3^{n}-1),(2^{n}-1,3^{n}-1)\}\cup\{(2^{n}-2^{k},3^{k}-3^{n})\colon 0\leq k\leq n-1\}$, which implies that
	
	$$\Ext(\conv(T(L,F_{1})))=\{(0,0),(0,1),(1,1),(1,-1)\}\cup\{(1-2^{-k},-1+3^{-k})\colon k\geq 0\}.$$
\end{example}

\subsection{The polytope case} 	
Here, we focus in the case when the convex hull of the digit tile is a polytope. We present some known results about this set that we will use in the rest of this article.

\begin{definition}
	We say that a substitution $\zeta$ is a \emph{polytope substitution} if it is extremally permutative, and the convex hull of the digit tile $T_{\zeta}=T(L_{\zeta},F_{1}^{\zeta})$ is a polytope.
\end{definition}

From now on, we only consider polytope substitutions. This geometrical hypothesis implies several algebraic restrictions on the expansion matrix $L_{\zeta}$ (\cref{AlgebraicPropertiesExpansionMap}) and some dynamical consequences for the substitutive subshift $(X_{\zeta},S,\Z^{d})$ (\cref{FinalTheoremNormalizerGroupPolytopeCase}).  

We recall here some results characterizing the polytope case in terms of the extreme points of $\conv(F_{n}^{\zeta})$ \cite{kirat2010remarksselfaffine}, and the inward unit normal vectors of the $(d-1)$-dimensional faces of $\conv(T_{\zeta})$ \cite{strichartz1999geometry}. 

\begin{theorem}\label{NecessaryAndSufficientConditionForPolytope}
	Let $T$ be the digit tile for an expansion matrix $L\in M_{d}(\R)$ and a fundamental domain $F_{1}\subseteq \R^{d}$. The following statements are equivalent:
	\begin{enumerate}
		\item The convex hull of the digit tile $T(L,F_{1})$ is a polytope.
		\item \cite[Theorem 4.2]{strichartz1999geometry} The inward unit normal vectors of the $(d-1)$-dimensional faces of $\conv(F_{1})$ are eigenvectors of $(L^{*})^{k}$ for some $k$.
		
		\item \cite[Theorem 2.2]{kirat2010remarksselfaffine} The cardinality of $\Ext(\conv(F_{n}))$ and $\Ext(\conv(F_{n+1}))$ are the same for some $n>0$. In such a case, for any $m>n$, $|\Ext(\conv(F_{m}))|=|\Ext(\conv(F_{n}))|$, and then $|\Ext(\conv(T(L,F_
	{1}))|=|\Ext(\conv(F_{n}))|$.
	\end{enumerate}		 
\end{theorem}

\begin{remark}
	In the case $L=\lambda \id_{\R^{d}}$, with $\lambda >1$, a direct computation shows that the statements (2) and (3) of \cref{NecessaryAndSufficientConditionForPolytope} are satisfied without taking any power of $L$.
\end{remark}

A big family for the polytope case is when a power of the expansion matrix $L$ is an integer multiple of the identity, because for any fundamental domain $F$ of $L(\Z^{d})$, the convex hull of the digit tile generated by $L$ and $F$ is a polytope. In particular, all the convex hull of the digit tiles of the examples in \cref{ExamplesApproximationDigitTiles} are polytopes. Nevertheless, it is not the only case where \cref{NecessaryAndSufficientConditionForPolytope} can be applied. 

\begin{example}
	\begin{enumerate}
		\item (Example of a non self-similar matrix with a polytope digit tile)\label{ExamplePolytopeNonSelfSimilar}
		Consider $L=\left(\begin{array}{cc}
			2 & 0 \\ -1 & 3
		\end{array}\right)$ and $F_{1}=\left\{(0,0),(1,2),(-2,-1),(-2,-3),(1,0),(-1,-1)\right\}$.
		
		\begin{center}
			\begin{figure}[H]
				\begin{tabular}{cc}
					\begin{tikzpicture}[scale=0.7]
						\draw[help lines, color=gray!30, dashed] (-1.9,-2.9) grid (1.9,2.9);
						\draw[->,thick] (-2,0)--(2,0) node[right]{$x$};
						\draw[->,thick] (0,-3)--(0,2) node[above]{$y$};
						
						\coordinate (A1) at (-2,-1);
						\coordinate (A2) at (0,0);
						\coordinate (A3) at (1,2);
						\coordinate (A4) at (-1,-1);
						\coordinate (A5) at (1,0);
						\coordinate (A6) at (-2,-3);
						
						\draw (A1) -- (A3) -- (A5) -- (A6) -- (A1);         
						
						\fill [blue!20, fill opacity = .5]   (A1)--(A3)--(A5)--(A6)--(A1)--cycle;
						
						\foreach \v in {A1,A2,A3,A4,A5,A6}  \draw[fill=gray] (\v) circle (2pt);		
					\end{tikzpicture} & \includegraphics[scale=0.2]{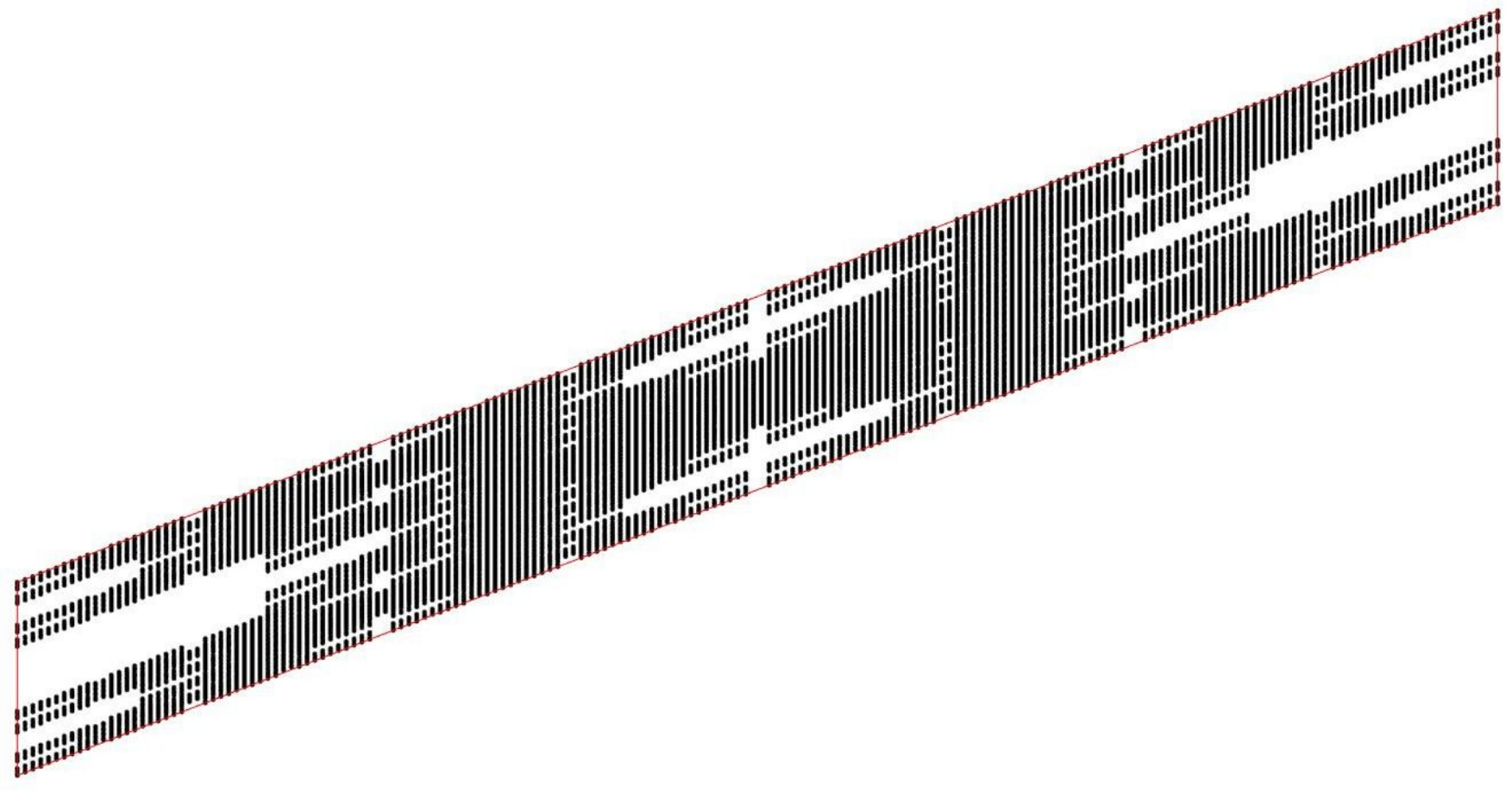}
				\end{tabular}
				\caption{The fundamental domain and an approximation of the digit tile of a non self-similar matrix.}
			\end{figure}
		\end{center}
		
		We have that $L^{*}=\left(\begin{array}{cc}
			2 & -1 \\ 0 & 3
		\end{array}\right)$ with eigenvectors equal to $\{(-1,1),(1,0)\}$. A direct computation shows that the set of extreme points of $\conv(F_{n})$ is equal to $\{(2^{n}-1,2^{n}-(3^{n}+1)/2),(2^{n}-1,(2^{n+1}+3^{n}-3)/2),(-2(2^{n}-1),(3^{n}+3-2^{n+2})/2),(-2(2^{n}-1),(5-3^{n}-2^{n+2})/2)\}$, so the set of extreme points of $\conv(T(L,F_{1}))$ is $\{(1,1/2),(1,3/2),(-2,-3/2),(-2,-5/2)\}$.
	
	\item As an example where the statement (3) in \cref{NecessaryAndSufficientConditionForPolytope} to be applied is not necessarily satisfied in $n=1$, consider $L=\left(\begin{array}{cc}
		-2 & 0 \\ 0 & -2
	\end{array}\right)$ and $F_{1}=\left\{(0,0),(1,0),(0,1),(-1,-1)\right\}$. We have that
	$$\begin{array}{cl}
		F_{2}= & \{(-1,-3),(0,-2),(1,-2),(-3,-1),(-1,-1),(0,-1),(-2,0),(-1,0),\\
		& \ \ (0,0),(1,0),(-2,1),(0,1),(1,1),(2,2),(3,2),(2,3)\},
	\end{array}$$
	
	\noindent so $\conv(F_{2})$ has 3 extreme points, while $\conv(F_{1})$ has 6 extreme points as shown in \cref{NegativeEigenvaluesProduceMoreExtremePoints}
	
	\begin{center}
		\begin{figure}[H]
			
			\begin{tikzpicture}[scale=0.5]
				\draw[help lines, color=gray!30, dashed] (-0.9,-2.9) grid (2.9,3.9);
				\draw[->,thick] (-1,0)--(3,0) node[right]{$x$};
				\draw[->,thick] (0,-3)--(0,3) node[above]{$y$};
				
				\coordinate (A1) at (0,0);
				\coordinate (A2) at (0,1);
				\coordinate (A3) at (1,0);
				\coordinate (A4) at (-1,-1);
				
				\draw (A4) -- (A2) -- (A3) -- (A4);         
				
				\fill [blue!20, fill opacity = .5]   (A4)--(A2)--(A3)--(A4)--cycle;
				
				\foreach \v in {A1,A2,A3,A4}  \draw[fill=gray] (\v) circle (2pt);
				
				\node (b1) at (-0.5,1.2)[scale=1]{$F_{1}$};

				\draw[help lines, color=gray!30, dashed] (7.1,-2.9) grid (13.9,3.9);
				\draw[->,thick] (7,0)--(14,0) node[right]{$x$};
				\draw[->,thick] (10,-3)--(10,3) node[above]{$y$};
				
				\coordinate (A5) at (9,-3);
				\coordinate (A6) at (10,-2);
				\coordinate (A7) at (11,-2);
				\coordinate (A8) at (7,-1);
				\coordinate (A9) at (9,-1);
				\coordinate (A10) at (10,-1);
				\coordinate (A11) at (8,0);
				\coordinate (A12) at (9,0);
				\coordinate (A13) at (10,0);
				\coordinate (A14) at (11,0);
				\coordinate (A15) at (8,1);
				\coordinate (A16) at (10,1);
				\coordinate (A17) at (11,1);
				\coordinate (A18) at (12,2);
				\coordinate (A19) at (13,2);
				\coordinate (A20) at (12,3);
				
				\draw (A5) -- (A8) -- (A15) -- (A20) -- (A19) --(A7) -- (A5);         
				
				\fill [blue!20, fill opacity = .5]   (A5)--(A8)--(A15)--(A20)--(A19)--(A7)--(A5)--cycle;
				
				\foreach \v in {A5,A6,A7,A8,A9,A10,A11,A12,A13,A14,A15,A16,A17,A18,A19,A20}  \draw[fill=gray] (\v) circle (2pt);
				
				\node (b2) at (7.5,1.2)[scale=1]{$F_{2}$};

				\coordinate (A21) at (23,0);
				\coordinate (A22) at (24,0);
				\coordinate (A23) at (23,1);
				\coordinate (A24) at (22,-1);
				\coordinate (A25) at (21,0);
				\coordinate (A26) at (22,0);
				\coordinate (A27) at (21,1);
				\coordinate (A28) at (20,-1);
				\coordinate (A29) at (23,-2);
				\coordinate (A30) at (24,-2);
				\coordinate (A31) at (23,-1);
				\coordinate (A32) at (22,-3);
				\coordinate (A33) at (25,2);
				\coordinate (A34) at (26,2);
				\coordinate (A35) at (25,3);
				\coordinate (A36) at (24,1);
				\coordinate (A37) at (27,0);
				\coordinate (A38) at (28,0);
				\coordinate (A39) at (27,1);
				\coordinate (A40) at (26,-1);
				\coordinate (A41) at (25,0);
				\coordinate (A42) at (26,0);
				\coordinate (A43) at (25,1);
				\coordinate (A44) at (24,-1);
				\coordinate (A45) at (27,-2);
				\coordinate (A46) at (28,-2);
				\coordinate (A47) at (27,-1);
				\coordinate (A48) at (26,-3);
				\coordinate (A49) at (29,2);
				\coordinate (A50) at (30,2);
				\coordinate (A51) at (29,3);
				\coordinate (A52) at (28,1);
				\coordinate (A53) at (23,4);
				\coordinate (A54) at (24,4);
				\coordinate (A55) at (23,5);
				\coordinate (A56) at (22,3);
				\coordinate (A57) at (21,4);
				\coordinate (A58) at (22,4);
				\coordinate (A59) at (21,5);
				\coordinate (A60) at (20,3);
				\coordinate (A61) at (23,2);
				\coordinate (A62) at (24,2);
				\coordinate (A63) at (23,3);
				\coordinate (A64) at (22,1);
				\coordinate (A65) at (25,6);
				\coordinate (A66) at (26,6);
				\coordinate (A67) at (25,7);
				\coordinate (A68) at (24,5);
				\coordinate (A69) at (19,-4);
				\coordinate (A70) at (20,-4);
				\coordinate (A71) at (19,-3);
				\coordinate (A72) at (18,-5);
				\coordinate (A73) at (17,-4);
				\coordinate (A74) at (18,-4);
				\coordinate (A75) at (17,-3);
				\coordinate (A76) at (16,-5);
				\coordinate (A77) at (19,-6);
				\coordinate (A78) at (20,-6);
				\coordinate (A79) at (19,-5);
				\coordinate (A80) at (18,-7);
				\coordinate (A81) at (21,-2);
				\coordinate (A82) at (22,-2);
				\coordinate (A83) at (21,-1);
				\coordinate (A84) at (20,-3);
				
				\draw[help lines, color=gray!30, dashed] (15.9,-6.9) grid (29.9,7.9);
				\draw[->,thick] (16,0)--(30,0) node[right]{$x$};
				\draw[->,thick] (23,-7)--(23,7) node[above]{$y$};
				
				\node (b2) at (19,2)[scale=1]{$F_{3}$};
				
				\draw (A46) -- (A50) -- (A51) -- (A66) -- (A67) --(A59) -- (A60) -- (A75) -- (A76) -- (A80) -- (A78) --(A48) -- (A46);
				
				\fill [blue!20, fill opacity = .5] (A46) -- (A50) -- (A51) -- (A66) -- (A67) --(A59) -- (A60) -- (A75) -- (A76) -- (A80) -- (A78) --(A48) -- (A46) -- cycle;
				\foreach \v in {A21,A22,A23,A24,A25,A26,A27,A28,A29,A30,A31,A32,A33,A34,A35,A36,A37,A38,A39,A40,A41,A42,A43,A44,A45,A46,A47,A48,A49,A50,A51,A52,A53,A54,A55,A56,A57,A58,A59,A60,A61,A62,A63,A64,A65,A66,A67,A68,A69,A70,A71,A72,A73,A74,A75,A76,A77,A78,A79,A80,A81,A82,A83,A84}  \draw[fill=gray] (\v) circle (2pt);
			\end{tikzpicture}
			\caption{The sets $F_{1}$, $F_{2}$ and $F_{3}$. Note that $\conv(F_{2})$ and $\conv(F_{3})$ have 6 extreme points.}
			\label{NegativeEigenvaluesProduceMoreExtremePoints}
		\end{figure}
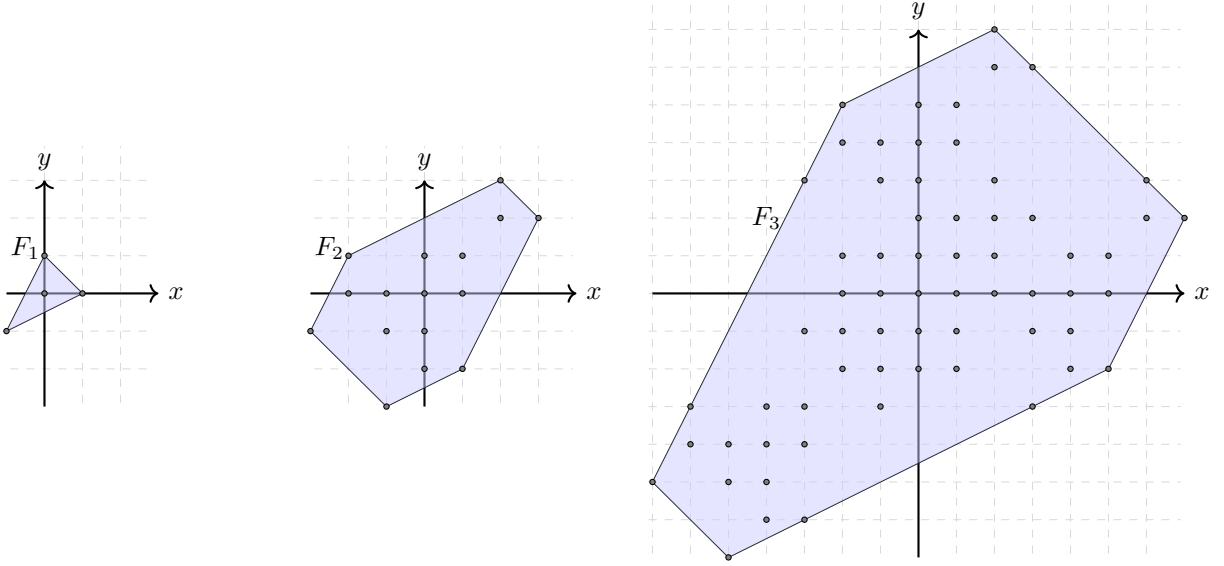
	\end{center}
\end{enumerate}

\end{example}

In \cite{kirat2010remarksselfaffine} the following result was proved about the extreme points of $\conv(T(L,F_{1}))$ as well as the extreme points of $\conv(F_{m})$ for any $m>n$, where $n$ is such that $|\Ext(\conv(F_{n}))|=|\Ext(\conv(F_{n+1}))|$.

\begin{proposition}\cite[Theorem 4.8]{kirat2010remarksselfaffine}
	If $|\Ext(\conv(F_{n}))|=|\Ext(\conv(F_{n+1}))|$, then all the extreme points of $\conv(T(L,F_{1}))$ are of the form $\sum\limits_{j>0}L^{-(n+1)j}\left(\sum\limits_{i=0}^{n}L^{i}({\bm f}_{i})\right)$, with $\sum\limits_{i=0}^{n}L^{i}({\bm f}_{i})$ being an extreme point of $\conv(F_{n+1})$. 
\end{proposition}

\noindent This implies that $\conv(T(L,F_{1}))$ is equal to $(L^{m}-\id)^{-1}\conv(F_{m})$ for all $m>n$.

Now, assume that we are under the condition $|\Ext(\conv(F_{1}))|=|\Ext(\conv(T(L,F_{1}))|$ and for all $n>0$, $\conv(T(L,F_{1}))=(L^{n}-\id)^{-1}\conv(F_{n})$. Let ${\bm u}$ be an inward unit normal vector of a $(d-1)$-dimensional face of $\conv(T(L,F_{1}))$. For each $n>0$, $((L^{n}-\id)^{*})^{-1}{\bm u}$ is an inward normal vector of a $(d-1)$-dimensional face of $F_{n}$. By \cref{NecessaryAndSufficientConditionForPolytope} (1), there exists $k>0$ such that $((L-\id)^{*})^{-1}{\bm u}$ is an eigenvector of $(L^{*})^{k}$. Hence by commutation, ${\bm u}$ is an eigenvector of $(L^{*})^{k}$. Since $\conv(T(L,F_{1}))$ is a polytope, we can take $n>0$ large enough such that any of the inward unit normal vectors of $\conv(F_{1})$ is an eigenvector of the same power $(L^{*})^{n}$. Hence, by the same arguments, up to considering a power of $L$, we may assume that all of the inward unit normal vectors of the $(d-1)$-dimensional faces of $\conv(F_{1})$ are eigenvectors of $L^{*}$. This is equivalent to the hyperplane $\partial H[{\bm u}]=\{{\bm t}\in \R^{d}\colon \left\langle {\bm t},{\bm u}\right\rangle =0\}$ (the vector space of an affine hull of a face of $\conv(F_{1})$) generated by ${\bm u}$, being preserved by $L$, i.e., $L\partial H[{\bm u}]=\partial H[{\bm u}]$. This implies that the normal fan $\N(\conv(F_{n}))$ is the same for all $n>0$, and it is equal to the one of $\conv(T(L,F_{1}))$.

Since for some $n>0$, $\conv(F_{n})$ is nondegenerate (by the F\o lner condition), it has $d$ linearly independent inward normal vectors (that have integer coordinates with no common divisor), which are eigenvectors of $L^{*}$. The polytope condition implies then, the following algebraic restrictions on the expansion matrix $L$. The proof is left to the reader.

\begin{proposition}\label{AlgebraicPropertiesExpansionMap}
	If $|\Ext(\conv(F_{1}^{\zeta})|=|\Ext(\conv(T(L,F_{1})))|$, then the eigenvalues of $L$ are integer numbers. 
	
	Moreover, if ${\bm u}_{1},{\bm u}_{2},{\bm u}_{3}$ are linearly dependent inward unit normal vectors of $(d-1)$-dimensional faces of $\conv(T(L,F_{1}))$, then $L$ restricted to the vector space generated by these vectors acts as an integer multiple of the identity. 
\end{proposition}

In particular, in the two-dimensional case, if the digit tile has 3 or at least 5 edges, then it follows that the expansion matrix is an integer multiple of the identity.	

Finally, up to taking an appropriate power of a substitution, we may assume the following hypothesis 

\begin{enumerate}[label=\text{(PC} \arabic*\text{)},ref=\text{(PC} \arabic*\text{)}]
	\item The expansion matrix $L$ is diagonalizable, with positive integer eigenvalues.\label{PC1}
	
	\item The convex set $\conv(F_{1})$ is nondegenerate and $|\Ext(\conv(F_{1}))|=|\Ext(\conv(T(L,F_{1}))|$.\label{PC2}
	
	\item\label{PC3} Any inward unit normal vector of a $(d-1)$-dimensional face of $\conv(F_{1})$ is an eigenvector of $L^{*}$.
	\item\label{PC4} The set $K$ given by \cref{FiniteSubsetFillsZd} is equal to $(\id -L)^{-1}(F_{1})\cap \Z^{d}$, i.e., for any ${\bm k}\in K$, there exists ${\bm f}\in F_{1}$ such that ${\bm k}=L({\bm k})+{\bm f}$.
\end{enumerate}

\subsection{Dynamical properties of substitutive subshifts from polytope substitutions}

As we saw in the previous subsection, under the hypothesis \ref{PC1}, \ref{PC2}, \ref{PC3} and \ref{PC4}, the normal fan is the same for the convex hull of the supports of $\zeta^{n}$ for any $n>0$, and for the convex hull of the digit tile, so we have the following interpretation of \cref{NonExpansiveHalfspacesSupportingConvexHull} in the polytope case.

\begin{corollary}[Nondeterministic directions in the polytope case]\label{InterpretationPolytopecaseNormalCones}
	Let $\zeta$ be an aperiodic primitive polytope substitution. The set of nondeterministic directions $\ND(X_{\zeta},S,\Z^{d})$ is the intersection of $\SS^{d-1}$ with a nonempty union of opposite normal cones of the form $\hat{N}_{{\bm G}}(\conv(T_{\zeta}))$, where ${\bm G}$ is a face of $\conv(T_{\zeta})$.
\end{corollary}

Hence, in the two-dimensional case, the former corollary implies strong restrictions on the set of nondeterministic directions. For instance, the number of its connected components is bounded by the number of edges of $\conv(T_{\zeta})$.

Now, as shown in the proof of \cref{NonExpansiveHalfspacesSupportingConvexHull}, to establish which opposite normal vectors of $\conv(F_{1}^{\zeta})$ appears as nondeterministic directions for $(X_{\zeta},S,\Z^{d})$ we study the convex sets $\conv(L_{\zeta}^{n}({\bm k})+F_{n}^{\zeta})$ generated by the points ${\bm k}$ in $K_{\zeta}$, which depend on the combinatorics of the substitution. We say that a subset $W\subseteq K_{\zeta}$ is a \emph{set of differences} if there exist two patterns $\texttt{w}_{1},\texttt{w}_{2}\in \mathcal{L}_{K_{\zeta}}(X_{\zeta})$ such that
$\texttt{w}_{1}({\bm k})$ is equal to $\texttt{w}_{2}({\bm k})$ if and only if ${\bm k}$ is in $K_{\zeta}\setminus W$.
	
The next lemma gives a sufficient condition to ensure that a vector ${\bm v}$ to be a nondeterministic direction for $(X_{\zeta},S,\Z^{d})$, seen as the converse of \cref{NonExpansiveHalfspacesSupportingConvexHull} in the polytope case. As in  \cref{AperiodicUniformylBounded}, we consider a set $C\Subset \Z^{d}$ such that for all $n>0$, $C+F_{n}^{\zeta}+F_{n}^{\zeta}\subseteq L_{\zeta}^{n}(C)+F_{n}^{\zeta}$ and $\overline{K}_{\zeta}=K_{\zeta}+C$.

\begin{lemma}\label{ConditionForBeingNonExpansiveBijectiveSubstitutions}
	Let $W\subseteq K_{\zeta}$ be a set of differences, ${\bm k}\in W$, $n>0$, a point ${\bm f}\in \partial \conv(L_{\zeta}^{n}({\bm k})+F_{n}^{\zeta})$ and ${\bm v}\in \SS^{d-1}$ be such that ${\bm f}+\partial H_{{\bm v}}$ supports $\conv(L_{\zeta}^{n}({\bm k})+F_{n}^{\zeta})$ at ${\bm f}$. Suppose that ${\bm f}$ satisfies the following conditions:
	\begin{enumerate}[label=\text{H}\arabic*.]
		\item \label{FirstConditionToBeNonDeterministic} ${\bm f}+\overline{K}_{\zeta}\subseteq L_{\zeta}^{n}(K_{\zeta})+F_{n}^{\zeta}$,
		\item ${\bm f}+ (\overline{K}_{\zeta}\cap H_{{\bm v}})\subseteq L_{\zeta}^{n}(K_{\zeta}\setminus W)+F_{n}^{\zeta}$.
	\end{enumerate}
	
	Then ${\bm v}$ is nondeterministic for $(X_{\zeta},S,\Z^{d})$.
\end{lemma}

\begin{proof}
	Let $\texttt{w}_{1},\texttt{w}_{2}$ be two patterns such that $\texttt{w}_{1}({\bm k}')=\texttt{w}_{2}({\bm k}')$ if and only if ${\bm k}'\in K_{\zeta}\setminus W$. Note that Condition \ref{FirstConditionToBeNonDeterministic} is equivalent to for all $m>0$, $L_{\zeta}^{m}({\bm f})+L_{\zeta}^{m}(\overline{K}_{\zeta})+F_{m}^{\zeta}\subseteq L_{\zeta}^{n+m}(K_{\zeta})+F_{n+m}^{\zeta}$. Since $\overline{K}_{\zeta}=K_{\zeta}+C$, \cref{RemarkSetDforInvariantOrbit} (2) implies that for all $m>0$, 
	\begin{equation}\label{EquivalentConditionH1}
		L_{\zeta}^{m}({{\bm f}})+F_{m}^{\zeta}+(L_{\zeta}^{m}(K_{\zeta})+F_{m}^{\zeta})\subseteq L_{\zeta}^{n+m}(K_{\zeta})+F_{n+m}^{\zeta}.
	\end{equation}
	
	If ${\bm f}$ is an extreme point of $\conv(L_{\zeta}^{n}({\bm k})+F_{n}^{\zeta})$, there exists ${\bm g}\in \Ext(\conv(F_{1}^{\zeta}))$ such that ${\bm f}=L_{\zeta}^{n}({\bm k})+\sum\limits_{i=0}^{n-1}L_{\zeta}^{i}({\bm g})$. If ${\bm f}$ is in the relative interior of a $k$-dimensional face of $\conv(L_{\zeta}^{n}({\bm k})+F_{n}^{\zeta})$, ($1\leq k\leq d-1$), we consider  ${\bm g}\in \Ext(\conv(F_{1}^{\zeta}))$ such that $L_{\zeta}^{n}({\bm k})+\sum\limits_{i=0}^{n-1}L_{\zeta}^{i}({\bm g})$ and ${\bm f}$ are in the same $k$-dimensional face of $\conv(L_{\zeta}^{n}({\bm k})+F_{n}^{\zeta})$ as shown in \cref{FigureHSupportingatF}.
	
	\begin{center}
		\begin{figure}[H]
			\begin{tikzpicture}[scale=0.7]
				\begin{scope}[on background layer]
					\shade let \p1=({0},{-3}),\p2=({0},{3}),
					\n1={90} in 
					[left color=white,right color=gray,middle color=white,shading angle=\n1]
					(\p1) -- (\p2)  --  ($(\p2)!2cm!-90:(\p1)$) -- ($(\p1)!2cm!90:(\p2)$)
					;
				\end{scope} 
				
				\coordinate (A1) at (0,0);
				\coordinate(A2) at (0,-3);
				\coordinate (A3) at (3,3);
				\coordinate(A4) at (0,-2);
				\coordinate (A5) at (6,-3);
				\coordinate (A6) at (6,0);
				
				\node (b3) at (0.2,2.8)[scale=1]{$H_{{\bm v}}$};

				\draw (A3) -- (A1) -- (A2) -- (A5) -- (A6) -- (A3); 
				
				\fill [brown!20, fill opacity = .5]   (A3)--(A1)--(A2)--(A5)--(A6)--(A3)--cycle;

				\node (b1) at (0.2,-2)[scale=1]{${\bm f}$};
				\node (b2) at (2.2,0)[scale=1]{$L_{\zeta}^{n}({\bm k})+\sum\limits_{i=0}^{n-1}L_{\zeta}^{i}({\bm g})$};
				\node (b4) at (6.2,2)[scale=1]{$\conv(L_{\zeta}^{n}({\bm k})+F_{n}^{\zeta})$};
				
				\foreach \v in {A1,A4}  \draw[fill=gray] (\v) circle (2pt);		
			\end{tikzpicture}
			\caption{The hyperplane $\partial H_{{\bm v}}$ supports $\conv(L_{\zeta}^{n}({\bm k})+F_{n}^{\zeta})$ at ${\bm f}$ and $L_{\zeta}^{n}({\bm k})+\sum\limits_{i=0}^{n-1}L_{\zeta}^{i}({\bm g})$.}
			\label{FigureHSupportingatF}
		\end{figure}
	\end{center}
	
	Now, Condition H2 is equivalent to for all $m>0$
	\begin{equation}\label{LH=H}
		L_{\zeta}^{m}({\bm f}+(\overline{K}_{\zeta}\cap H_{{\bm v}}))+F_{m}\subseteq L_{\zeta}^{n+m}(K_{\zeta}\setminus W)+F_{n+m}^{\zeta}.
	\end{equation}
	
	We will prove that for all $m>0$
	
	\begin{equation}\label{LH=Hbest}
		L_{\zeta}^{m}({\bm f})+\sum\limits_{i=0}^{m-1}L_{\zeta}^{i}({\bm g})+((L_{\zeta}^{m}(K_{\zeta})+F_{m}^{\zeta})\cap H_{{\bm v}})\subseteq L_{\zeta}^{m}({\bm f})+(\overline{K}_{\zeta}\cap H_{{\bm v}})+F_{m}^{\zeta}.
	\end{equation} 
	
	We have that
	\begin{equation}\label{EquationSupportingHyperplaneNormalVector}
		\left\langle {\bm v},{\bm g}\right\rangle= \min\limits_{{\bm i}\in F_{1}^{\zeta}}\left\langle {\bm v},{\bm i}\right\rangle.
	\end{equation}
	
	Fix ${\bm n}\in L_{\zeta}^{m}({\bm f})+\sum\limits_{i=0}^{m-1}L_{\zeta}^{i}({\bm g})+((L_{\zeta}^{m}(K_{\zeta})+F_{m}^{\zeta})\cap H_{{\bm v}})$. There exist ${\bm k}_{1}\in K_{\zeta}$, ${\bm j}\in F_{m}^{\zeta}$ and ${\bm h}\in H_{{\bm v}}$ such that
	$${\bm n}=L_{\zeta}^{m}({\bm f})+\sum\limits_{i=0}^{m-1}L_{\zeta}^{i}({\bm g})+L_{\zeta}^{m}({\bm k}_{1})+{\bm j}=L_{\zeta}^{m}({\bm f})+\sum\limits_{i=0}^{m-1}L_{\zeta}^{i}({\bm g})+{\bm h}.$$
	
	On the other hand, by definition of $\overline{K}_{\zeta}$, there exist ${\bm c}\in \overline{K}_{\zeta}$ and ${\bm l}\in F_{m}^{\zeta}$ such that
	$$\sum\limits_{i=0}^{m-1}L_{\zeta}^{i}({\bm g})+L_{\zeta}^{m}({\bm k}_{1})+{\bm j}=L_{\zeta}^{m}({\bm c})+{\bm l}.$$
	
	To prove \eqref{LH=Hbest}, it is enough to show that ${\bm c}\in H_{{\bm v}}$. Indeed, writing ${\bm l}=\sum\limits_{i=0}^{m-1}L_{\zeta}^{i}({\bm l}^{(i)})$, with ${\bm l}^{(i)}\in F_{1}^{\zeta}$ for $0\leq i \leq m-1$, we have that ${\bm c}=L_{\zeta}^{-m}\left(\sum\limits_{i=0}^{m-1}L_{\zeta}^{i}({\bm g}-{\bm l}^{(i)})+{\bm h}\right)$. So, we get that
	$$\begin{array}{cl}
		\left \langle {\bm v},{\bm c}\right\rangle & = \left\langle {\bm v},L_{\zeta}^{-m}\left(\sum\limits_{i=0}^{m-1}L_{\zeta}^{i}({\bm g}-{\bm l}^{(i)})+{\bm h}\right)\right\rangle \\
		& = \dfrac{1}{\lambda^{m}}\left(\sum\limits_{i=0}^{m-1}\left\langle {\bm v},L_{\zeta}^{i}({\bm g}-{\bm l}^{(i)})\right\rangle\right)+\dfrac{1}{\lambda^{m}}\left\langle {\bm v},{\bm h}\right\rangle \\
		& =
		\dfrac{1}{\lambda^{m}}\left(\sum\limits_{i=0}^{m-1}\lambda^{i}\left\langle {\bm v},{\bm g}-{\bm l}^{(i)}\right\rangle\right)+\dfrac{1}{\lambda^{m}}\left\langle {\bm v},{\bm h}\right\rangle.	
	\end{array}$$
	
	Since ${\bm h}\in H_{{\bm v}}$, then $\left\langle {\bm v},{\bm h}\right\rangle\leq 0$, and by \eqref{EquationSupportingHyperplaneNormalVector} we have that $\left\langle {\bm v},{\bm g}-{\bm l}^{(i)}\right\rangle\leq 0$, for all $1\leq i\leq m-1$. Since $\lambda >0$, we conclude that $\left \langle {\bm v},{\bm c}\right\rangle\leq 0$, which implies that ${\bm c}\in H_{{\bm v}}$.

	By \eqref{EquivalentConditionH1} and \cref{FiniteSubsetFillsZd}, the iterations of the substitution on the patterns $\texttt{w}_{1}$, $\texttt{w}_{2}$ leads to two points $x_{1}\neq x_{2}\in X_{\zeta}$ such that, $x_{i}$ is in $\left[\zeta^{n}(\texttt{w}_{1})\right]_{-L_{\zeta}^{n}({\bm f})-\sum\limits_{i=0}^{n-1}L_{\zeta}^{i}({\bm g})}$, for $i\in\{1,2\}$. Finally, \eqref{LH=Hbest} implies that $x_{1}|_{H_{{\bm v}}}=x_{2}|_{H_{{\bm v}}}$. The fact that ${\bm f}$ is in the set of differences of $\texttt{w}_{1}$ and $\texttt{w}_{2}$ ensures that $x_{1}({\bm 0})\neq x_{2}({\bm 0})$, so $x_{1}\neq x_{2}$. We then conclude that ${\bm v}$ is nondeterministic for $(X_{\zeta},S,\Z^{d})$.
\end{proof}

As described in \cref{NonExpansiveHalfspacesSupportingConvexHull}, depending on which faces of $\conv(L_{\zeta}^{n}({\bm k})+F_{n}^{\zeta})$, the point ${\bm f}$ satisfying Condition  \ref{FirstConditionToBeNonDeterministic}  belongs, we may have more nondeterministic directions, obtaining the following corollary:

\begin{corollary}\label{CorollaryNonExpansiveHalfspacesBijectiveSubstitutions}
	Let $W\subseteq K_{\zeta}$ be a set of differences, ${\bm k}\in W$, $n>0$, and a point ${\bm f}\in \partial \conv(L_{\zeta}^{n}({\bm k})+F_{n}^{\zeta})\cap \partial \conv(L_{\zeta}^{n}(W)+F_{n}^{\zeta})$ satisfying Condition \ref{FirstConditionToBeNonDeterministic} of \cref{ConditionForBeingNonExpansiveBijectiveSubstitutions}. Then any ${\bm v}$ in $\hat{N}_{{\bm F}}(\conv(L_{\zeta}^{n}(W)+F_{n}^{\zeta}))\cap \SS^{d-1}$ (with ${\bm F}$ being the face of smallest dimension where ${\bm f}$ belongs in $\conv(L_{\zeta}^{n}(W)+F_{n}^{\zeta})$) is nondeterministic for $(X_{\zeta},S,\Z^{d})$.
\end{corollary}

In addition to this result, the proof of \cref{NonExpansiveHalfspacesSupportingConvexHull} provides that, for some $n>0$, the opposite normal cone $\hat{N}_{{\bm F}}(\conv(L_{\zeta}^{n}(W)+F_{n}^{\zeta}))$ is equal to an opposite normal cone $\hat{N}_{{\bm  G}}(\conv(F_{n}^{\zeta}))$, where ${\bm G}$ is a face of $\conv(F_{n}^{\zeta})$.

In the following we present two examples with different behaviors given by \cref{ConditionForBeingNonExpansiveBijectiveSubstitutions}

\begin{example}[Different behaviors for the nondeterministic directions]\label{ExampleDifferentBehaviorForNonExpansiveHalfSpaces}
	\begin{enumerate}
		\item Consider the 2D-Thue Morse substitution with $L_{\zeta_{TM}}=\left(\begin{array}{cc}
			2 & 0 \\ 0 & 2
		\end{array}\right)$, $F_{1}^{\zeta_{TM}}=\llbracket 0,1 \rrbracket^{2}$, given by
	
		$$\begin{array}{lllllllll}
			\zeta_{TM}:&  & 1 & 0 &  &  & 0 & 1\\ 
				& 0\mapsto & 0 & 1 &  & 1\mapsto & 1 & 0.\\ 
		\end{array}$$
	
		The following is a pattern of $\zeta_{TM}$:
		
		\begin{figure}[H]
			$$\begin{array}{c}
				0110100110010110\\
				1001011001101001\\
				0110100110010110\\
				1001011001101001\\
				1001011001101001\\
				0110100110010110\\
			\end{array}$$
		\caption{A pattern of the substitution $\zeta_{TM}$.}
		\end{figure}
		
		In this case, the set $K_{\zeta_{TM}}$ is equal to $\llbracket-1,0\rrbracket^{2}$. We have that
		$$\mathcal{L}_{K_{\zeta_{TM}}}(X_{TM})=\left\{\begin{array}{cc}
			0 & 1\\
			1 & 0
		\end{array},\begin{array}{cc}
			1 & 0\\
			0 & 1
		\end{array},\begin{array}{cc}
			1 & 0\\
			1 & 0
		\end{array},\begin{array}{cc}
			0 & 1\\
			0 & 1
		\end{array},\begin{array}{cc}
			0 & 0\\
			1 & 1
		\end{array},\begin{array}{cc}
			1 & 1\\
			0 & 0
		\end{array},\begin{array}{cc}
			0 & 0\\
			0 & 0
		\end{array},\begin{array}{cc}
			1 & 1\\
			1 & 1
		\end{array}\right\}.$$
		
		The sets of differences for the 2D-Thue Morse substitution are $\{\{(0,-1),(0,0)\},\{(-1,0),(0,0)\},\{(-1,0),(0,-1)\},\{(-1,-1),(0,0)\},\{(-1,-1),(-1,0)\}\}$. By \cref{ConditionForBeingNonExpansiveBijectiveSubstitutions} it can be proved $\dbinom{1}{0}$, $\dbinom{-1}{0}$, $\dbinom{0}{1}$, $\dbinom{0}{1}$ are the only nondeterministic directions for $(X_{\zeta_{TM}},S,\Z^{2})$.
		
		\item Consider the substitution of the table tiling \cite{robinson1999table}, with $L_{\zeta_{t}}=\left(\begin{array}{cc}
			2 & 0 \\ 0 & 2
		\end{array}\right)$, $F_{1}^{\zeta_{t}}=\llbracket 0,1 \rrbracket^{2}$, given by $\zeta_{t}:$
		
		$$\begin{array}{ccccccccccc}
			0\mapsto & \begin{array}{cc}
				3 & 0\\ 1 & 0
			\end{array},& \quad & 1\mapsto & \begin{array}{cc}
				1 & 1\\ 0 & 2
			\end{array},& \quad &  2\mapsto & \begin{array}{cc}
				2 & 3\\ 2 & 1
			\end{array},& \quad & 3\mapsto & \begin{array}{cc}
				0 & 2\\ 3 & 3
			\end{array}.
		\end{array}$$
	
	The following is a pattern of $\zeta_{t}$:
	
		\begin{figure}[H]
			$$\begin{array}{c}
				3023302302021111\\
				1021102133330202\\
				0230230211110230\\
				3310213302023310\\
				1130231130231130\\
				0210210210210210\\
			\end{array}$$
			\caption{A pattern of the substitution $\zeta_{t}$.}
		\end{figure}
		
		The set $K_{\zeta_{t}}$ is equal to $\llbracket -1,0\rrbracket^{2}$ and the sets of differences is equal to $2^{K_{\zeta_{t}}}\setminus\{\emptyset, K_{\zeta_{t}},\{(-1,-1),(0,0)\},\{(0,-1),(-1,0)\}\}$. By \cref{ConditionForBeingNonExpansiveBijectiveSubstitutions}, it can be proved that the set of nondeterministic directions for $(X_{\zeta},S,\Z^{2})$ is the whole circle $\mathbb{S}^{1}$.
	\end{enumerate}
\end{example}

Now, we proceed to determine the normalizer semigroup $N(X_{\zeta},S,\Z^{d})$ of substitutive subshifts given by polytope substitutions. Set $M\in \vec{N}(X_{\zeta},S,\Z^{d})$. By \cref{NormalizerActsInNonExpansiveHalfspaces} if ${\bm v}$ is a nondeterministic direction for $(X_{\zeta},S,\Z^{d})$, then $M^{*}{\bm v}/\Vert M^{*}{\bm v}\Vert$ is also a nondeterministic direction for $(X_{\zeta},S,\Z^{d})$. Moreover, \cref{NonExpansiveHalfspacesSupportingConvexHull} ensures that the matrix $M$ acts on the opposite normal cones of $\conv(T_{\zeta})$, that appeared as nondeterministic directions for $(X_{\zeta},S,\Z^{d})$. In particular, the matrix $M^{*}$ permutes the hyperplanes defined by the $(d-1)$-dimensional faces of $\conv(T_{\zeta})$ whose unit opposite normal cones are nondeterministic directions for $(X_{\zeta},S,\Z^{d})$. If there are $d$ linearly independent nondeterministic directions for $(X_{\zeta},S,\Z^{d})$, we get the following result about the linear representation semigroup.

\begin{proposition}\label{PropositionAboutMatricesPolytopeCase}
	Let $\zeta$ be an aperiodic primitive polytope substitution. If the set of nondeterministic directions for $(X_{\zeta},S,\Z^{d})$ contains $d$ linearly independent vectors, then
	
	\begin{enumerate}
		\item Any homomorphism $\phi\in N(X_{\zeta},S,\Z^{d})$ is invertible.
		
		\item The linear representation semigroup $\vec{N}(X_{\zeta},S,\Z^{d})$ is a finite group, and it is isomorphic to a subgroup of $GL(d,\Z/3\Z)$.
		
		\item The norm of any matrix $M$ in the linear representation group $\vec{N}(X_{\zeta},S,\Z^{d})$ is bounded by an explicit formula  only depending on the convex hull of the digit tile $T_{\zeta}$.
	\end{enumerate} 
\end{proposition}

\begin{proof}
	By assumption and \cref{InterpretationPolytopecaseNormalCones} there are $d$ linearly independent inward unit normal vectors to the $(d-1)$-dimensional faces of $\conv(T_{\zeta})$ that are nondeterministic directions for $(X_{\zeta},S,\Z^{d})$. Let $n\geq d$ be the number of unit normal vectors to the $(d-1)$-dimensional faces of $\conv(T_{\zeta})$ that are nondeterministic directions. Any matrix $M$ in the linear representation group of $(X_{\zeta},S,\Z^{d})$ permutes the hyperplanes defined by these $(d-1)$-dimensional faces of $\conv(T_{\zeta})$. By condition \ref{PC3}, the normal lines (generated by the inward vectors) of these hyperplanes are invariant by a power of the expansion matrix $L_{\zeta}^{*}$. Hence $M^{*}$ permutes $n$ eigenspaces $\{\mathbb{Q}{\bm v}_{1},\ldots,\mathbb{Q}{\bm v}_{n}\}$ of some power of $L_{\zeta}^{*}$. Moreover, we can assume that these vectors have integer coordinates not having common divisors except $\pm 1$. Note that each vector is unique up to a sign and does not depend on $M$. Since $M$ is in $GL(d,\Z)$, $M^{*}$ is also in $GL(d,\Z)$. This implies that $M^{*}$ sends vectors with integer coordinates with no common divisor, to vectors with the same property. Therefore, for all $1\leq i\leq n$ $(M^{*})^{2n!}{\bm v}_{i}=\pm {\bm v}_{i}$. Since the set $\{{\bm v}_{1},\ldots, {\bm v}_{n}\}$ contains $d$ linearly independent vectors, we get that $(M^{*})^{2n!}$ is the identity matrix, which implies that $M$ has finite order. By \cref{NormalizerCoalescenceFiniteGroup} any homomorphism of $(X_{\zeta},S,\Z^{d})$ is invertible. We recall that a subgroup of $GL(d,\Z)$ is finite if and only if any element in the subgroup has finite order. We then conclude that the linear representation group $\vec{N}(X_{\zeta},S,\Z^{d})$ is finite, and by Minkowski's theorem, we have that $\vec{N}(X_{\zeta},S,\Z^{d})\leq GL(d,\Z/3\Z)$.
	
	Finally, note that $\Vert M\Vert \leq \Vert P\Vert\cdot \Vert Q_{M} \Vert\cdot \Vert P^{-1}\Vert$, where $\Vert P\Vert, \Vert P^{-1}\Vert$ and $\sup\limits_{M\in \vec{N}(X_{\zeta},S,\Z^{d})} \Vert Q_{M}\Vert<\infty$ only depend on the convex hull of the digit tile $T_{\zeta}$.
\end{proof}
	
\begin{remark}In particular, it follows from the proof that if $n=d$ each matrix $Q_{M}$ is a permutation matrix, so  $\vec{N}(X_{\zeta},S,\Z^{d})$ is conjugate to a subgroup of the hyperoctaedral group $W_{d}$. These recover results in \cite{bustos2022admissible} for block substitutions. By the realization result \cite[Theorem 35]{bustos2022admissible} these results obtained are optimal. 
\end{remark}

The following theorem summarizes all the properties satisfied for aperiodic primitive reduced polytope substitutions.

\

\begin{theorem}\label{FinalTheoremNormalizerGroupPolytopeCase}
	Let $\zeta$ be an aperiodic reduced primitive polytope substitution. Then
 \begin{enumerate}
	 	\item \label{CoalescenceOfHomomorphismsPolytopeSubstitutions} The system $(X_{\zeta},S,\Z^{d})$ is coalescent, and also any homomorphism in $N(X_{\zeta},S,\Z^{d})$ is invertible.
 \end{enumerate}
 
 If there are $d$ linearly independent vectors that are nondeterministic directions for $(X_{\zeta},S,\Z^{d})$, we have that
	\begin{enumerate}[resume]			
		\item The normalizer is virtually generated by the shift action.
		
		\item \label{SymmetryGroupPermutesNormalFanPolytopeSubstitutions} The linear representation group $\vec{N}(X_{\zeta},S,\Z^{d})$ acts as a permutation group on the set $\{N_{{\bm G}}(\conv(T_{\zeta}))\colon N_{{\bm G}}(\conv(T_{\zeta})) \subseteq \ND(X_{\zeta},S,\Z^{d}), {\bm G}\ \text{a face of}\ \conv(T_{\zeta})\}$. In particular, if  $\ND(X_{\zeta},S,\Z^{d})=\SS^{d-1}$, then the linear representation group $\vec{N}(X_{\zeta},S,\Z^{d})$ is isomorphic to a subgroup of the automorphism group of the normal fan of $\conv(T_{\zeta})$.
	\end{enumerate}
\end{theorem}

\begin{proof}
	The statement \ref{CoalescenceOfHomomorphismsPolytopeSubstitutions} is true by \cref{Coalescence} and \cref{NormalizerCoalescenceFiniteGroup}.
	
	Now, by the third isomorphism theorem we have that 
	$${\Large \sfrac{N(X_{\zeta},S,\Z^{d})}{\Aut(X_{\zeta},S,\Z^{d})} \cong \left(\sfrac{N(X_{\zeta},S,\Z^{d})}{\left\langle S\right\rangle}\right)/(\sfrac{\Aut(X_{\zeta},S,\Z^{d})}{\left\langle S\right\rangle})}.$$
	
	Then, \cref{PropositionAboutMatricesPolytopeCase} gives that the quotient group $N(X_{\zeta},S,\Z^{d})/\Aut(X_{\zeta},S,\Z^{d})=\vec{N}(X_{\zeta},S,\Z^{d})$ is finite and \cref{AutomoprhismVirtuallyZd} implies that $\Aut(X_{\zeta},S,\Z^{d})/\left\langle S\right\rangle$ is also finite. We conclude that $N(X_{\zeta},S,\Z^{d})/\left\langle S\right\rangle$ is a finite group. 
	
	Finally, the statement \ref{SymmetryGroupPermutesNormalFanPolytopeSubstitutions} is true by \cref{PropositionAboutMatricesPolytopeCase}.
\end{proof}

The hypothesis of having $d$ linearly independent vectors as nondeterministic directions for $(X_{\zeta},S,\Z^{d})$ is decidable by an algorithm, studying all the sets of differences and then applying \cref{ConditionForBeingNonExpansiveBijectiveSubstitutions}. Until now we did not find an aperiodic $d$-dimensional reduced primitive polytope substitution with less than $d$ linearly independent nondeterministic directions leaving the following question:

\begin{question}
	Does there exist an aperiodic primitive reduced polytope substitution $\zeta$ with less than $d$ linearly independent nondeterministic directions for $(X_{\zeta},S,\Z^{d})$?
\end{question}

\bibliographystyle{plain}
\bibliography{sample}

\begin{thebibliography}{10}

\bibitem{baake2019number}
M.~Baake, {\'A}.~Bustos, C.~Huck, M.~Lema{\'n}czyk, and A.~Nickel.
\newblock Number-theoretic positive entropy shifts with small centralizer and
  large normalizer.
\newblock {\em Ergodic Theory and Dynamical Systems}, pages 1--26, 2019.

\bibitem{baake2006structure}
M.~Baake and J.~A.~G. Roberts.
\newblock The structure of reversing symmetry groups.
\newblock {\em Bull. Austral. Math. Soc.}, 73(3):445--459, 2006.

\bibitem{baake2018reversing}
M.~Baake, J.~A.~G. Roberts, and R.~Yassawi.
\newblock Reversing and extended symmetries of shift spaces.
\newblock {\em Discrete Contin. Dyn. Syst.}, 38(2):835--866, 2018.

\bibitem{berthe2019recognizability}
V.~Berth\'{e}, W.~Steiner, J.~M. Thuswaldner, and R.~Yassawi.
\newblock Recognizability for sequences of morphisms.
\newblock {\em Ergodic Theory Dynam. Systems}, 39(11):2896--2931, 2019.

\bibitem{bezuglyi2008fullgroups}
S.~Bezuglyi and K.~Medynets.
\newblock Full groups, flip conjugacy, and orbit equivalence of {C}antor
  minimal systems.
\newblock {\em Colloq. Math.}, 110(2):409--429, 2008.

\bibitem{boyle1997expansive}
M.~Boyle and D.~Lind.
\newblock Expansive subdynamics.
\newblock {\em Trans. Amer. Math. Soc.}, 349(1):55--102, 1997.

\bibitem{bustos2020extended}
\'{A}. Bustos.
\newblock Extended symmetry groups of multidimensional subshifts with
  hierarchical structure.
\newblock {\em Discrete Contin. Dyn. Syst.}, 40(10):5869--5895, 2020.

\bibitem{bustos2022admissible}
{\'A}.~Bustos, D.~Luz, and N.~Ma{\~n}ibo.
\newblock Admissible reversing and extended symmetries for bijective
  substitutions.
\newblock {\em Discrete Comput. Geom.}, 2022.

\bibitem{cortez2006toeplitz}
M.~I. Cortez.
\newblock {${\Bbb Z}^d$} {T}oeplitz arrays.
\newblock {\em Discrete Contin. Dyn. Syst.}, 15(3):859--881, 2006.

\bibitem{cortez2008godometers}
M.~I. Cortez and S.~Petite.
\newblock {$G$}-odometers and their almost one-to-one extensions.
\newblock {\em J. Lond. Math. Soc. (2)}, 78(1):1--20, 2008.

\bibitem{cyr2015nonexpansive}
V.~Cyr and B.~Kra.
\newblock Nonexpansive {$\mathbb{Z}^2$}-subdynamics and {N}ivat's conjecture.
\newblock {\em Trans. Amer. Math. Soc.}, 367(9):6487--6537, 2015.

\bibitem{dekking1978spectrum}
F.~M. Dekking.
\newblock The spectrum of dynamical systems arising from substitutions of
  constant length.
\newblock {\em Z. Wahrscheinlichkeitstheorie und Verw. Gebiete},
  41(3):221--239, 1977/78.

\bibitem{donoso2016lowcomplexity}
S.~Donoso, F.~Durand, A.~Maass, and S.~Petite.
\newblock On automorphism groups of low complexity subshifts.
\newblock {\em Ergodic Theory Dynam. Systems}, 36(1):64--95, 2016.

\bibitem{downarowicz2008finiterank}
T.~Downarowicz and A.~Maass.
\newblock Finite-rank {B}ratteli-{V}ershik diagrams are expansive.
\newblock {\em Ergodic Theory Dynam. Systems}, 28(3):739--747, 2008.

\bibitem{durand2000linearly}
F.~Durand.
\newblock Linearly recurrent subshifts have a finite number of non-periodic
  subshift factors.
\newblock {\em Ergodic Theory Dynam. Systems}, 20(4):1061--1078, 2000.

\bibitem{durand2016constant}
F.~Durand and J.~Leroy.
\newblock The constant of recognizability is computable for primitive
  morphisms.
\newblock {\em arXiv preprint arXiv:1610.05577}, 2016.

\bibitem{durand2018decidability}
F.~Durand and J.~Leroy.
\newblock Decidability of the isomorphism and the factorization between minimal
  substitution subshifts.
\newblock {\em arXiv preprint arXiv:1806.04891}, 2018.

\bibitem{ehrenfeucht1978simplifcations}
A.~Ehrenfeucht and G.~Rozenberg.
\newblock Simplifications of homomorphisms.
\newblock {\em Inform. and Control}, 38(3):298--309, 1978.

\bibitem{fagnot1997facteurs}
I.~Fagnot.
\newblock Sur les facteurs des mots automatiques.
\newblock {\em Theoret. Comput. Sci.}, 172(1-2):67--89, 1997.

\bibitem{fogg2002substitutions}
N.~P. Fogg.
\newblock {\em Substitutions in dynamics, arithmetics and combinatorics},
  volume 1794 of {\em Lecture Notes in Mathematics}.
\newblock Springer-Verlag, Berlin, 2002.
\newblock Edited by V. Berth\'{e}, S. Ferenczi, C. Mauduit and A. Siegel.

\bibitem{frank2005multidimensional}
N.~P. Frank.
\newblock Multidimensional constant-length substitution sequences.
\newblock {\em Topology Appl.}, 152(1-2):44--69, 2005.

\bibitem{frank2022spectral}
N.~P. Frank and N.~Ma\~{n}ibo.
\newblock Spectral theory of spin substitutions.
\newblock {\em Discrete Contin. Dyn. Syst.}, 42(11):5399--5435, 2022.

\bibitem{gahler2012combinatorics}
F.~G\"{a}hler, A.~Julien, and J.~Savinien.
\newblock Combinatorics and topology of the {R}obinson tiling.
\newblock {\em C. R. Math. Acad. Sci. Paris}, 350(11-12):627--631, 2012.

\bibitem{goodson1999inverse}
G.~R. Goodson.
\newblock Inverse conjugacies and reversing symmetry groups.
\newblock {\em Amer. Math. Monthly}, 106(1):19--26, 1999.

\bibitem{gottschalk1963substitution}
W.~H. Gottschalk.
\newblock Substitution minimal sets.
\newblock {\em Trans. Amer. Math. Soc.}, 109:467--491, 1963.

\bibitem{guillon2015determinism}
P.~Guillon, J.~Kari, and C.~Zinoviadis.
\newblock Determinism in subshifts.
\newblock {\em Unpublished manuscript}, 2015.

\bibitem{halmos1942operator}
P.~R. Halmos and J.~von Neumann.
\newblock Operator methods in classical mechanics. {II}.
\newblock {\em Ann. of Math. (2)}, 43:332--350, 1942.

\bibitem{hedlund1969endomorphism}
G.~A. Hedlund.
\newblock Endomorphisms and automorphisms of the shift dynamical system.
\newblock {\em Math. Systems Theory}, 3:320--375, 1969.

\bibitem{host1989homomorphismes}
B.~Host and F.~Parreau.
\newblock Homomorphismes entre syst\`emes dynamiques d\'{e}finis par
  substitutions.
\newblock {\em Ergodic Theory Dynam. Systems}, 9(3):469--477, 1989.

\bibitem{kirat2010remarksselfaffine}
I.~Kirat and I.~Kocyigit.
\newblock Remarks on self-affine fractals with polytope convex hulls.
\newblock {\em Fractals}, 18(4):483--498, 2010.

\bibitem{leader1991limit}
J.~J. Leader.
\newblock Limit orbits of a power iteration for dominant eigenvalue problems.
\newblock {\em Appl. Math. Lett.}, 4(4):41--44, 1991.

\bibitem{lee2003consequences}
J.-Y. Lee, R.~V. Moody, and B.~Solomyak.
\newblock Consequences of pure point diffraction spectra for multiset
  substitution systems.
\newblock {\em Discrete Comput. Geom.}, 29(4):525--560, 2003.

\bibitem{lemanczyk1988metric}
M.~Lema\'{n}czyk and M.~K. Mentzen.
\newblock On metric properties of substitutions.
\newblock {\em Compositio Math.}, 65(3):241--263, 1988.

\bibitem{mosse1992puissances}
B.~Moss\'{e}.
\newblock Puissances de mots et reconnaissabilit\'{e} des points fixes d'une
  substitution.
\newblock {\em Theoret. Comput. Sci.}, 99(2):327--334, 1992.

\bibitem{yassawi2020automorphisms}
C.~M\"{u}llner and R.~Yassawi.
\newblock Automorphisms of automatic shifts.
\newblock {\em Ergodic Theory Dynam. Systems}, 41(5):1530--1559, 2021.

\bibitem{penrose1974role}
R.~Penrose.
\newblock The role of aesthetics in pure and applied mathematical research.
\newblock {\em Bull. Inst. Math. Appl.}, 10:266--271, 1974.

\bibitem{queffelec2010substitution}
M.~Queff\'{e}lec.
\newblock {\em Substitution dynamical systems---spectral analysis}, volume 1294
  of {\em Lecture Notes in Mathematics}.
\newblock Springer-Verlag, Berlin, second edition, 2010.

\bibitem{robinson1999table}
E.~A. Robinson, Jr.
\newblock On the table and the chair.
\newblock {\em Indag. Math. (N.S.)}, 10(4):581--599, 1999.

\bibitem{rockafellar1970convex}
R.~T. Rockafellar.
\newblock {\em Convex analysis}.
\newblock Princeton Mathematical Series, No. 28. Princeton University Press,
  Princeton, N.J., 1970.

\bibitem{salo2015blockmaps}
V.~Salo and I.~T\"{o}rm\"{a}.
\newblock Block maps between primitive uniform and {P}isot substitutions.
\newblock {\em Ergodic Theory Dynam. Systems}, 35(7):2292--2310, 2015.

\bibitem{solomyakrecognizability}
B.~Solomyak.
\newblock Nonperiodicity implies unique composition for self-similar
  translationally finite tilings.
\newblock {\em Discrete Comput. Geom.}, 20(2):265--279, 1998.

\bibitem{solomyak2007eigenfunctions}
B.~Solomyak.
\newblock Eigenfunctions for substitution tiling systems.
\newblock In {\em Probability and number theory---{K}anazawa 2005}, volume~49
  of {\em Adv. Stud. Pure Math.}, pages 433--454. Math. Soc. Japan, Tokyo,
  2007.

\bibitem{strichartz1999geometry}
R.~S. Strichartz and Y.~Wang.
\newblock Geometry of self-affine tiles. {I}.
\newblock {\em Indiana Univ. Math. J.}, 48(1):1--23, 1999.

\bibitem{vince2000digit}
A.~Vince.
\newblock Digit tiling of {E}uclidean space.
\newblock In {\em Directions in mathematical quasicrystals}, volume~13 of {\em
  CRM Monogr. Ser.}, pages 329--370. Amer. Math. Soc., Providence, RI, 2000.

\end{thebibliography}
\end{document}